\documentclass[letterpaper,11pt,reqno]{amsart}
\usepackage[margin=1.2in]{geometry}
\usepackage{eucal}
\usepackage{tikz}
\usepackage{xypic}
\usepackage{amsfonts,amssymb}
\usepackage{epsfig}
\usepackage{enumerate}
\usepackage{enumitem}
\usepackage{ mathrsfs }
\usetikzlibrary{arrows}
\usetikzlibrary{decorations.pathreplacing}

\newtheorem{theorem}{Theorem}[section]
\newtheorem{lemma}[theorem]{Lemma}
\newtheorem{prop}[theorem]{Proposition}
\newtheorem{corollary}[theorem]{Corollary}

\theoremstyle{definition}
\newtheorem{defn}[theorem]{Definition}
\newtheorem{remark}[theorem]{Remark}
\newtheorem{example}[theorem]{Example}

\newcommand{\Q}{\mathbb Q}
\newcommand{\C}{\mathbb C}
\newcommand{\R}{\mathbb R}

\newcommand{\D}{\mathbb D}

\newcommand{\Z}{\mathbb Z}
\newcommand{\N}{\mathbb N}

\newcommand{\cN}{\mathcal N}
\newcommand{\cM}{\mathcal M}
\newcommand{\cO}{\mathcal O}

\def\Dom{\textnormal{Dom}}

\def\Im{\textnormal{Im}}

\def\hra#1{\overset{#1}{\lhook\joinrel\longrightarrow}}
\def\lra#1{\overset{#1}{\longrightarrow}}

\newlist{HF}{enumerate}{1}
\setlist[HF]{label=(HF\arabic*)}

\newlist{PB}{enumerate}{1}
\setlist[PB]{label=(PB\arabic*)}

\newlist{CZ}{enumerate}{1}
\setlist[CZ]{label=(CZ\arabic*)}

\newlist{S1HF}{enumerate}{1}
\setlist[S1HF]{label=(S1HF\arabic*)}

\numberwithin{equation}{section}

\title{Floer Cohomology, Multiplicity and the Log Canonical Threshold.}
\author{Mark McLean}

\begin{document}

\begin{abstract}
Let $f$ be a polynomial over the complex numbers with an isolated singularity at $0$.
We show that the multiplicity and the log canonical threshold of $f$ at $0$
are invariants of the link of $f$ viewed as a contact submanifold of the sphere.

This is done by first constructing a spectral sequence converging to the fixed point Floer cohomology of any iterate of the Milnor monodromy map whose $E^1$ page is explicitly described in terms of a log resolution of $f$.
This spectral sequence is a generalization of a formula by A'Campo.
By looking at this spectral sequence, we get a purely Floer theoretic description of the multiplicity and log canonical threshold of $f$.
\end{abstract}

\maketitle

\bibliographystyle{halpha}

\tableofcontents

\section{Introduction}

Let $f : \C^{n+1} \to \C$ be a polynomial with an isolated singular point at $0$ where $n \geq 1$. Let $S_\epsilon \subset \C^{n+1}$ be the sphere of radius $\epsilon$ centered at $0$.
The {\it link} of $f$ at $0$ is the submanifold $L_f \equiv f^{-1}(0) \cap S_\epsilon \subset S_\epsilon$
where $\epsilon>0$ is sufficiently small.
One can ask the following question: what is the relationship between the link of $f$ and various algebraic properties of $f$?
For instance
Zariski in \cite{zariski:openquestions}
asked whether the multiplicity of $f$ at $0$
depends only on the embedding $L_f \subset S_\epsilon$.
Another important invariant is the {\it log canonical threshold} (see \cite{atiyahresolutiondistributions}, \cite{mustata:logcanonicalthreshold} or Definition \ref{defn:logresolution}).
Again one can ask if this is an invariant of $L_f \subset S_\epsilon$
(see \cite[Section 1.6]{budursingularityinvariants}).
We will answer weaker versions of these questions.
If $\epsilon>0$ is small enough, it turns out that $L_f$ is naturally a contact submanifold of $S_\epsilon$
(See \cite{Varchenko:isolatedsingularities}).
If $g : \C^{n+1} \lra{} \C$ is another polynomial with isolated singularity at $0$ then we say that $f$ and $g$ have {\it embedded contactomorphic links} if there is a contactomorphism $\Phi : S_\epsilon \lra{} S_\epsilon$ sending $L_f$ to $L_g$.
Varchenko showed in \cite{Varchenko:isolatedsingularities}
that if there is a holomorphic change of coordinates sending $f$ to $g$ then they have embedded contactomorphic links.
One of the goals of this paper is to prove the following theorem.
\begin{theorem} \label{theorem:lctmultiplicitycontactinvariance}
Suppose that $f, g : \C^{n+1} \lra{} \C$ are two polynomials with isolated singular points at $0$
with embedded contactomorphic links.
Then the multiplicity and the log canonical threshold
of $f$ and $g$ are equal.
\end{theorem}
We will prove this theorem by finding formulas for the multiplicity and log canonical threshold in terms of a sequence Floer cohomology groups.
The key technical result of this paper proving the above theorem will be a natural generalization of a formula by A'Campo \cite{acampo:monodromyzeta}.

For all $\epsilon > 0$ small enough,
there is a smooth fibration
$$\arg(f) : S_\epsilon - f^{-1}(0) \lra{} \R/2\pi\Z, \quad \arg(f)(z) \equiv \arg(f(z))$$
called the {\it Milnor fibration} associated to $f$ (see \cite[Chapter 4]{milnorsingular}).
A fiber $M_f \equiv \arg(f)^{-1}(0)$ is called the {\it Milnor fiber} of $f$.
By choosing an appropriate connection on this fibration, there is a natural compactly supported diffeomorphism
$\phi : M_f \lra{} M_f$ given by parallel transporting
around the circle $\R/2\pi\Z$ called the {\it Milnor monodromy map}. 
The
 {\it Lefschetz number $\Lambda(\phi^m)$} of $\phi^m$ is defined to be
 $$\Lambda(\phi^m) \equiv \sum_{j=0}^\infty (-1)^j\text{Tr}(\phi^m_* : H_j(M_f;\Z) \lra{} H_j(M_f;\Z)),$$
and this is
an invariant of the embedding $L_f \subset S_\epsilon$
 for each $m > 0$.
A'Campo in \cite{acampo:monodromyzeta} computed these numbers in the following way. 
Let $\pi : Y \lra{} \C^{n+1}$ be a log resolution of the pair $(\C^{n+1},f^{-1}(0))$ at $0$.
Let $(E_j)_{j \in \check{S}}$ be the prime exceptional divisors of this resolution and define $E_{\star_S}$
to be the proper transform
$\overline{\pi^{-1}(f^{-1}(0)-0)}$
of $f^{-1}(0)$.
Let $S \equiv \check{S} \sqcup \{\star_S\}$.
%
%
Define $E_j^0 \equiv E_j - \cup_{i \in S-\{j\}} E_i$ for all $j \in S$
and define $S_m \equiv \{i \in \check{S} \ : \text{ord}_f(E_i) \ \text{divides} \  m\}$ for all $m > 0$.
A'Campo showed that
\begin{equation} \label{eqn:acamposformula}
\Lambda(\phi^m) = \sum_{i \in  S_m} \text{ord}_f(E_i) \chi(E_i^o), \quad \forall m>0.
\end{equation}

The key technical result of this paper (Theorem \ref{theorem:mainspectralsequence})
is a spectral sequence
converging to a group whose Euler characteristic is naturally equal to the left hand side of
(\ref{eqn:acamposformula}) multiplied by $(-1)^n$ and so that the Euler characteristic of the $E^1$ page is naturally equal to the right had side multiplied by $(-1)^n$. We will now explain this result.

For $\epsilon>0$ small enough, the Milnor fiber $M_f$ is naturally a symplectic manifold and $\phi$ can be made to be a compactly supported symplectomorphism 
(see Section \ref{section liouville domains etc}).
For any compactly supported symplectomorphism $\psi$ satisfying some additional properties, one can assign a group $HF^*(\psi,+)$
called the {\it Floer cohomology} group of $\psi$ (see \cite[Section 4]{Seidel:morevanishing} or Section \ref{section fixed point floer cohomology definition} of this paper). The Euler characteristic of this group is $(-1)^n$ multiplied by the Lefschetz number of $\psi$
(See property \ref{HF property Euler characteristic} in Section \ref{section fixed point floer cohomology definition} of this paper).
As a result, we have a sequence of groups $HF^*(\phi^m,+)$ whose Euler characteristic is $(-1)^n\Lambda(\phi^m)$ for all $m > 0$.
All of these groups are invariants of the link of $f$ up to embedded contactomorphism (See Lemma \ref{lemma:invarianceforlinks}).
%
The log resolution $\pi : Y \lra{} \C^{n+1}$ is called a {\it multiplicity $m$ separating resolution} if $\text{ord}_f(E_i) + \text{ord}_f(E_j) > m$ for all $i,j \in S$ satisfying $i \neq j$ and $E_i \cap E_j \neq \emptyset$.

\begin{theorem} \label{theorem:mainspectralsequence}
Suppose that $\pi : Y \lra{} \C^{n+1}$ is a multiplicity $m$ separating resolution
for some $m \in \N_{>0}$.
Let $(w_i)_{i \in \check{S}}$ be positive integers so that $-\sum_{i \in \check{S}} w_i E_i$ is ample.
Let $a_i$ be the discrepancy of $E_i$ (see Definition \ref{defn:logresolution}) and define $k_i \equiv \frac{m}{\text{ord}_f(E_i)}$ for all $i \in S_m$ .
Then there is a cohomological spectral sequence converging to $HF^*(\phi^m,+)$ with $E^1$ page
$$E_1^{p,q} = 
\bigoplus_{\{ i \in S_m \ : \ k_i w_i  = -p\}} H_{n - (p+q) - 2k_i(a_i+1) }(\widetilde{E}^o_i;\Z) 
$$
where $\widetilde{E}^o_i$ is an $m_i$-fold cover of $E^o_i$ for all $i \in S_m$.
The cover $\widetilde{E}^o_i$ is constructed as follows:
Let $U_i$ be a neighborhood of $E_i^o$
inside $Y - \cup_{j \in S - i}E_i$ which deformation retracts on to $E_i$, let $\iota_i : U_i - E_i^o \lra{} U_i$ be the natural inclusion map and define
$$f_i : U_i - E_i^o \lra{} \C^*, \ f_i(x) \equiv f(\pi(x)).$$
Then $\widetilde{E}_i^o$ is disjoint union of connected covers corresponding to normal subgroup
$$G_i := (\iota_i)_*(\ker((f_i)_*)) \subset \pi_1(U_i) = \pi_1(E_i^o)$$
and the number of such covers is $m_i$ divided by the index of $G_i$ in $\pi_1(E^o_i)$.

\end{theorem}

The covers $\widetilde{E}^o_i$ are described in an explicit algebraic way in \cite[Section 2.3]{denefloeserlefschetz}.
Intuitively, we should think of $\widetilde{E}^o_i$ in the following way:
if $U_i$ was a `nice' tubular neighborhood of $E_i^o$ and we had a `nice' projection map $Q_i : U_i \lra{} E_i^o$
then $Q_i : (\pi \circ f)^{-1}(\epsilon) \cap U_i \lra{} E_i^o$ would be a covering map on to its image homotopic to $\widetilde{E}^o_i$ for $\epsilon>0$ small enough (see the proof of Lemma \ref{lemma:indexdividesmi}).

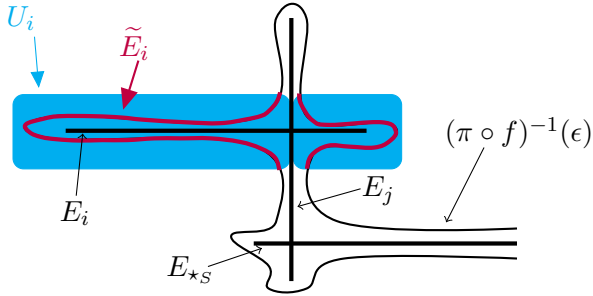
\begin{tikzpicture}

	\draw[cyan, fill=cyan, rounded corners]  (1.4664,-0.0188) rectangle (0.026,-1.0072);
	
	\draw[cyan, fill=cyan, rounded corners]  (-3.7,-0.0188) rectangle (-0.0175,-1.0072);

	\draw[ultra thick] (0,1) -- (0,-2.5);

	\draw[ultra thick] (1,-0.5) -- (-3,-0.5);

	\draw[ultra thick] (-0.5,-2) -- (3,-2);

	\draw[thick]  plot[smooth, tension=.7] coordinates {(3,-1.8) (0.6,-1.7) (0.3,-0.8) (1.2,-0.7) (1.3,-0.4) (0.2,-0.3) (0.2,1) (-0.2,1) (-0.4,-0.3) (-3.1,-0.3) (-3.2,-0.6) (-0.5,-0.7) (-0.2,-1.7) (-0.8,-1.9) (-0.3,-2.3) (-0.3,-2.6) (0.2,-2.6) (0.7,-2.2) (3,-2.2)};

	\draw[purple, ultra thick]  plot[smooth, tension=.7] coordinates {(-0.1489,0.0114) (-0.1952,-0.1157) (-0.2992,-0.2371) (-0.3975,-0.3007) (-0.6229,-0.3527) (-0.9466,-0.399) (-1.4033,-0.3874) (-1.9698,-0.3527) (-2.2588,-0.3238) (-2.6229,-0.3065) (-3.0796,-0.3065) (-3.2992,-0.3412) (-3.5305,-0.4163) (-3.4785,-0.5204) (-3.3224,-0.5839) (-3.12,-0.6128) (-2.8252,-0.6302) (-2.5132,-0.6244) (-2.1836,-0.6244) (-1.8195,-0.6013) (-1.4726,-0.6013) (-1.0738,-0.6013) (-0.6865,-0.6475) (-0.4091,-0.7342) (-0.2183,-0.8556) (-0.1605,-1.0232)};

	\draw[purple, ultra thick]  plot[smooth, tension=.7] coordinates {(0.1054,-0.0059) (0.1285,-0.1446) (0.2846,-0.3643) (0.5041,-0.4221) (1.2037,-0.3874) (1.3886,-0.503) (1.3019,-0.6417) (1.1054,-0.7227) (0.7701,-0.7111) (0.4869,-0.6996) (0.2672,-0.8267) (0.2094,-1.0059)};

	\draw [->](2.562,-0.7227) -- (2.0591,-1.7169);

	\node at (3.0291,-0.5261) {$(\pi \circ f)^{-1}(\epsilon)$};

	\draw [->](-1.1085,-2.3065) -- (-0.3917,-2.0637);

	\node at (-1.3397,-2.347) {$E_{\star_S}$};

	\draw [->](0.9204,-1.2718) -- (0.0823,-1.5088);

	\node at (1.1575,-1.295) {$E_{j}$};

	\draw [->](-2.8599,-1.3643) -- (-2.7732,-0.5377);

	\node at (-2.8773,-1.5955) {$E_{i}$};

	\draw [purple,-triangle 45,thick](-2.0333,0.4334) -- (-2.2357,-0.2198);

	\node[purple] at (-2.0854,0.6588) {$\widetilde{E}_{i}$};

	\draw [cyan,-triangle 45](-3.4148,0.7513) -- (-3.357,0.0692);

	\node[cyan] at (-3.594,0.9999) {$U_i$};

\end{tikzpicture}

By Lemma \ref{lemma:existenceofmultiplicitymseparating} combined with \cite{hironaka:resolution},
a resolution satisfying the properties stated in the above Theorem exists for each $m>0$.
By looking at this spectral sequence above, one gets the following corollary:
\begin{corollary} \label{corollary:spectralsequencenumbers}
For each $m > 0$, define $\nu_m \equiv \sup\{\alpha : HF^\alpha(\phi^m,+) \neq 0 \}$
and define
$\mu_m \equiv \inf \{k_i (a_i + 1) \ : \ i \in S_m \}$ where $k_i$ and $a_i$ are defined as in Theorem \ref{theorem:mainspectralsequence} above.
Then $$\nu_m = n-2\mu_m \quad \forall \ m > 0.$$
In particular $HF^*(\phi^m,+)$ vanishes if and only if $\mu_m = \infty$.
Also the numbers $\mu_m$ are invariants of the link up to embedded contactomorphism since
the groups $HF^*(\phi^m,+)$ are.
\end{corollary}
We will prove this corollary in Section \ref{section:proofofmainresult}.
We have an immediate corollary of this corollary proving a conjecture of Seidel \cite{seidel:singularitieslecture} regarding the multiplicity of a singularity.
\begin{corollary} \label{corollary:multiplicityandlct}
The multiplicity of $f$ is the smallest $m > 0$
so that $HF^*(\phi^m,+) \neq 0$.
The log canonical threshold of $f$ at $0$ is
$$\text{lct}_0(f) = \liminf_{m \to \infty} \left(\inf \left\{ \ -\frac{\alpha}{2m} \ : \  HF^\alpha(\phi^m,+) \neq 0 \ \text{or} \ -\frac{\alpha}{2m} = 1 \right\}\right).$$
\end{corollary}
Since $HF^*(\phi^m,+)$ is an invariant of the link of $f$ up to embedded contactomorphism by Lemma \ref{lemma:invarianceforlinks} we get that Theorem \ref{theorem:lctmultiplicitycontactinvariance}
follows immediately from Corollary \ref{corollary:multiplicityandlct}.

For each $m > 0$, Denef and Loeser in \cite{denefloeserlefschetz} constructed natural spaces $\chi_{m,1}$ whose Euler characteristic is $\Lambda(\phi^m)$.
Therefore it is natural to ask, what is the relationship between these spaces and the groups $HF^*(\phi^m,+)$, if any? Such a question was considered by Seidel (see \cite[Remark 2.7]{denefloeserlefschetz}).
It might be interesting to see if there is a similar spectral sequence converging to $H_*(\chi_{m,1};\Z)$ since these spaces admit a natural stratification induced by the strata of the log resolution $\pi$.
A possible proof would
exploit the spectral sequence
\cite[Formula (3)]{petersen2016spectral}
combined with \cite[Lemma 2.2]{denefloeserlefschetz} (see also the calculations in the proof of \cite[Lemma 2.5]{denefloeserlefschetz}).

\subsection{Sketch of the proof of Theorem \ref{theorem:mainspectralsequence}}

We will now state one of the key properties of the group $HF^*(\psi,+)$ that will be used in this proof.
This property is stated precisely in \ref{item:spectralsequenceproperty} Section
\ref{section fixed point floer cohomology definition} and proven in Appendix C.

{\it Spectral Sequence Property}: Suppose that the set of fixed points of $\psi$
is a disjoint union of connected codimension $0$ submanifolds
$B_1,\cdots,B_l$ with boundary and corners and suppose that $\psi$ behaves in a particular way near the boundary of $B_i$ for each $i$.
Then there is a grading
$CZ(B_i) \in \Z$ for each $B_i$ and
there is a specific function $\iota : \{1,\cdots,l\} \lra{} \N$ 
so that there is a spectral sequence converging to $HF^*(\psi,+)$
with $E^1$ page equal to
$$E_1^{p,q} = 
\bigoplus_{ \{i \in \{1,\cdots,l\} \ : \ \iota(i) = p\} } H_{n-(p+q)-CZ(\phi,B_i)}(B_i;\Z) 
$$

The spectral sequence above is an example of a {\it Morse-Bott spectral sequence} (See \cite[Corollary 2]{bottstableclassical} and \cite[Section 6.4]{hutchingslecturenotesonfloer} for other similar examples).
Therefore in order to prove Theorem \ref{theorem:mainspectralsequence} it would be sufficient for us to deform the monodromy symplectomorphism $\phi^m$ so that the set of fixed points is a union
of codimension $0$ submanifolds homotopic to
$\widetilde{E}^o_{i_p}$ for each $i \in \{1,\cdots,l\}$.
The problem is that we cannot quite do this, but we can construct a new symplectomorphism with the required fixed point sets without changing $HF^*(\phi^m,+)$.
Also, Theorem \ref{theorem:mainspectralsequence} really requires a specific ordering of the submanifolds
$\widetilde{E}^o_{i_p}$
corresponding to the sequence of positive integers $(w_j)_{j \in S}$ but we will ignore this detail here, as the main applications of this paper do not need such an ordering. We will now explain how to modify $\phi^m$ without changing $HF^*(\phi^m,+)$ so that it has this fixed point property.
This is done in Section \ref{section constructing a nice contact open book}.

We have a natural symplectic form $\omega_Y$ on $Y$
coming from the ample divisor $-\sum_{i \in S} w_i E_i$.
This symplectic form gives us a natural Ehresmann connection on $\pi^* f$ away form $(\pi^*f)^{-1}(0)$
and hence gives us a monodromy map.
First of all, we deform $\omega_Y$ so that it behaves well with respect to $\pi^* f$
(See Sections \ref{section some preliminary definitions} and \ref{section trivialization line bundles}).
The key idea is that since $\pi^* f$ locally looks like $\prod_{i = 1}^m z_i^m$, we can deform $\omega_Y$ so that it basically looks like the standard symplectic form in these local charts (with a few modifications).
The next step is to show that the corresponding monodromy map $\psi$ satisfies $HF^*(\psi^m,+) = HF^*(\phi^m,+)$ (See sections \ref{section links of divisors and open books}, \ref{section making the monodromy nice} and Appendix A).
Here we are using the fact that these Floer cohomology groups are invariants of the mapping tori of $\phi^m$ and $\psi^m$
respectively 
along with an additional contact structure on these tori and some additional data.
Finally we need to compute the fixed points of the monodromy map so that we can apply our spectral sequence property (See section \ref{section good dynamical properties}).

\subsection{Plan of the paper}

In Section \ref{section multiplicities and discrepancies} we construct algebraic invariants of $(\C^{n+1},f^{-1}(0))$ which will be used to tell us the smallest non-vanishing degree of $HF^*(\phi^m,+)$ for each $m$.
These invariants are constructed by looking at the multiplicities and discrepancies of the prime exceptional divisors $(E_i)_{i \in S}$ of a resolution.

In Section \ref{section liouville domains etc} we give some basic definitions of the main objects in symplectic and contact geometry that will be used in this paper. These include Liouville domains, (abstract) open books, contact mapping cylinders and gradings.
In Section \ref{section fixed point floer cohomology definition} we give a definition Floer cohomology of a symplectomorphism. We also state the three main properties \ref{HF property Euler characteristic} - \ref{item:spectralsequenceproperty} of the Floer cohomology group $HF^*(\psi,+)$ that will be needed for this paper. These properties will be proven in Appendix B and C.

Section \ref{section constructing a nice contact open book} is the largest section of the paper. This section is used to construct a monodromy symplectomorphism nice enough so that we can use the properties from Section \ref{section fixed point floer cohomology definition} to prove  Theorem \ref{theorem:mainspectralsequence}.
This section heavily relies on results and notation from \cite{McLeanTehraniZinger:smoothing}.
Section \ref{section:proofofmainresult}
contains a proof of Theorem \ref{theorem:mainspectralsequence}
and Corollary \ref{corollary:spectralsequencenumbers}.

Appendix A is deals with gradings. It enables us to compute the quantities $CZ(B_i)$ stated in the spectral sequence property in the sketch of the proof of Theorem \ref{theorem:mainspectralsequence} earlier.
Appendix B proves that the groups $HF^*(\phi^m,+)$ only depend on the link $L_f \subset S_\epsilon$ as a contact submanifold. This relies heavily on results of \cite{McLean:monodromy}.
In Appendix C we prove the spectral sequence property of $HF^*(\psi,+)$ described above.

\subsection{Conventions}
If $(M,\omega)$ is a symplectic manifold and $\theta$ is a $1$-form
then its {\it $\omega$-dual} $X^\omega_\theta$ is the unique vector field satisfying $\omega(X^\omega_\theta,Y) = \theta(Y)$ for all vectors $Y$.
Sometimes we just write $X_\theta$ instead of $X^\omega_\theta$ if it is clear from the context that the symplectic form we are using is $\omega$.
For a smooth function $H : M \lra{} \R$, we define $X_H \equiv X_{-dH}$.
The time $t$ flow of $X_{-dH}$ will be denoted by $\phi^H_t : M \lra{} M$ (this is called the {\it time $t$ Hamiltonian flow of $H$}).

Also if $f : B' \lra{} B$ is a smooth map and $\pi : V \lra{} B$ is a vector bundle then
we will write elements of the pullback bundle $f^*(V)$ as pairs $(b',v) \in B' \times V$
satisfying $f(b') = \pi(v)$.
For any fiber bundle $\pi : E \lra{} B$ and any subsets $N \subset E, \ C \subset B$, we define $N|_C \equiv N \cap \pi^{-1}(C)$.
To avoid cluttered notation, we will not distinguish between an element
of a set and a subset of size $1$ when the context is clear (e.g. $i$ will quite often mean $\{i\}$).
We also write $\Dom(f)$ and $\Im(f)$ for the domain and image of a map $f$.
For any set $I$, we define
$\N_{>0}^I$ to be the set
of tuples $(k_i)_{i \in I}$
where $k_i \in \N_{>0}$.

\bigskip

{\bf Acknowledgments}: Many thanks to Mircea~Musta\c{t}\v{a} for suggesting the connection with log canonical threshold. Also many thanks to Paul Seidel for answering some of my questions. This paper is supported by the NSF grant DMS-1508207.

\section{Multiplicities and Discrepancies of Exceptional Divisors.} \label{section multiplicities and discrepancies}

In this section we will introduce some of the basic tools that are needed from algebraic geometry.
We will define the multiplicity and log canonical threshold of an isolated hypersurface singularity as well as some more general invariants.
We will also explain how to compute these invariants in terms of certain resolutions called multiplicity $m$ separating resolutions and show how such computational techniques do not depend on the choice of resolution.

Let $f : \C^{n+1} \to \C$ be a polynomial with an isolated singular point at $0$.

\begin{defn} \label{defn:logresolution}
A {\it log resolution at $0$} of the pair $(\C^{n+1},f^{-1}(0))$ is a proper holomorphic map $\pi : Y \to \C^{n+1}$ from a complex manifold $Y$ so that
there is some open set $U \subset \C^{n+1}$ containing $0$ satisfying:
\begin{enumerate}
\item $\pi^{-1}(f^{-1}(0)) \cap U$ is a finite union of smooth transversally intersecting hypersurfaces $(E_i)_{i \in S}$.
We will call such divisors {\it resolution divisors}.
Each $E_j$ satisfying $\pi(E_j) = \{0\}$ is called a {\it prime exceptional divisor}. We require that the prime exceptional divisors are connected.
We also require that there is a unique element $\star_S \in S$ where $E_{\star_S} = \overline{\pi^{-1}(f^{-1}(0) - 0)}$
($E_{\star_S}$ need not be connected).
We call $E_{\star_S}$ the {\it proper transform of $f^{-1}(0)$}.
\item $\pi|_{\pi^{-1}(U \setminus \{0\})} : \pi^{-1}(U \setminus \{0\}) \to U \setminus \{0\}$
is a biholomorpihsm.
\end{enumerate} 

Since the only singularities in this paper will be at $0 \in \C^{n+1}$,
we will just call a log resolution at $0$ of $f$ a {\it log resolution} of $(\C^{n+1},f^{-1}(0))$.

The {\it multiplicity of $f$ along $E_j$},
denoted by $\text{ord}_f(E_j)$, is the order of $\pi^* f$ along $E_j$.
In other words, choose some local coordinate chart $z_1,\cdots,z_n$ centered at a generic point of $E_j$
so that $E_j = \{z_1 = 0\}$
and define $\text{ord}_f(E) \equiv k$ where $k \in \Z$ satisfies $\pi^* f = g z_1^k$ in this coordinate system for some holomorphic function $g$ satisfying $g(0) \neq 0$.

The {\it discrepancy} of $E_j$ denoted by $a(E_j)$ is calculated as follows:
choose local holomorphic coordinates $y_1,\cdots,y_n$ on $Y$ centered at a point on $E_j$
and holomorphic coordinates $x_1,\cdots,x_n$ on $\C^{n+1}$ centered at $0$.
Then $a(E_j)$ is the order of the Jacobian determinant  of $f$ along $E_j$ expressed in these coordinates.
This quantity does not depend on the choices of such holomorphic coordinates.
The {\it multiplicity} of $f$ at $0$ is $\min \{\text{ord}_f(E_j) \ : \ j \in S - \star_S \}$
and the {\it log canonical threshold} is $\min \left\{\frac{a(E_j) + 1}{\text{ord}_f(E_j)} \ : \ j \in S \right\}$.

Throughout this paper, we will define $E_I$ to be $\cap_{j \in I} E_j$
for each $I \subset S$. If $I$ is the empty set then $E_I$ is the entire manifold $Y$.
\end{defn}

\begin{defn} \label{defn:minimalmultiplicitydiscrepancy}
Let $\pi : Y \to \C^{n+1}$ be a log resolution of $(\C^{n+1},f^{-1}(0))$ as above.
For each $m \in \N_{>0}$, we define the {\it minimal multiplicity $m$ discrepancy} to be:
\[\text{md}_m(\pi,f) \equiv \text{inf} \left\{ \sum_{j \in I} k_j a(E_j)  : \ I \subset S, \ I \neq \star_S, \ (k_j)_{j \in I} \in \N_{>0}^I, \  E_I \neq \emptyset, \ \sum_{j \in I} k_j \text{ord}_f(E_j) = m \right\}.\]
Our convention here is that infimum of the empty set is $\infty$.
Later on in Lemma \ref{lemma:multipicitymresolutionindepcedence} we will show that
$\text{md}_m(\pi,f)$ does not depend on $\pi$ and hence we can define $\text{md}_m(f) \equiv \text{md}_m(\pi,f)$ for some choice of log resolution $\pi$.

A morphism $\pi : Y \to \C^{n+1}$ is a {\it multiplicity $m$ separating resolution} if it is a log resolution of $(\C^{n+1},f^{-1}(0))$
such that for any two resolution divisors $E$ and $F$ of $\pi$ satisfying $E \cap F \neq \emptyset$,
the sum of the multiplicities of $f$ along $E$ and $F$ is greater than $m$.
\end{defn}

Multiplicity $m$ separating resolutions make it much easier for us to compute the minimal multiplicity $m$ discrepancy.

\begin{lemma} \label{lemma:multiplicitymresolutiondiscrepancycalculation}
	If $\pi : Y \to \C^{n+1}$ is a multiplicity $m$ separating resolution of $(\C^{n+1},f^{-1}(0))$
	and $(E_j)_{j \in S}$ are its resolution divisors then
	\[\text{md}_m(\pi,f) = \text{inf} \left\{ ka(E_j) \ : \ k \in \N_{>0}, \ j \in S-\star_S, \ k\text{ord}_f(E_j) = m\right\}.\]
\end{lemma}
\proof
This follows from the fact that if
$\sum_{j \in I} k_j \text{ord}_f(E_j) = m$ and $E_I \neq \emptyset$ for some $I \subset S - \star_S$ and $(k_j)_{j \in I} \in \N_{>0}^I$ 
then $|I| = 1$.
\qed

\begin{lemma} \label{lemma:existenceofmultiplicitymseparating}
If we have any log resolution then we can blow such a resolution up along strata of $\cup_{i \in S} E_i$ inside $\pi^{-1}(0)$ so that it becomes a multiplicity $m$ separating resolution.
\end{lemma}
\proof
Let $\pi : Y \lra{} \C^{n+1}$ be a resolution with resolution divisors $(E_j)_{j \in S}$.
Define
$$a_Y \equiv \text{min}\left\{ \sum_{j \in I} k_j \text{ord}_f(E_j) \ : \ I \subset S, \ |I| = 2, \ E_I \neq \emptyset, \ (k_j)_{j \in I} \in \N_{>0}^I\right\}.$$
Let $b_Y$ be the number of elements in the set
$$B_Y \equiv \left\{I \subset S \ : \ |I| = 2, \ E_I \neq \emptyset, \ \sum_{j \in I} k_j \text{ord}_f(E_j) = a_Y \right\}.$$
Since $b_Y \geq 1$, choose $I \in B_Y$.
Let $Y'$ be the blowup of $Y$ along $E_I$.
Then $a_Y - b_Y$ is strictly smaller than $a_{Y'} - b_{Y'}$.
Hence by induction we can blow up $Y$ along subsets of the form $E_I$ until we get a log resolution $\pi'' : Y'' \lra{} \C$
of $(\C^{n+1},f^{-1}(0))$
so that $a_{Y''} - b_{Y''} \geq m$.
Since $b_{Y''} \geq 1$, we get that $a_{Y''} > m$.
Hence $\pi''$ is a multiplicity $m$ separating resolution.
\qed

\begin{lemma} \label{lemma:mdmafterblowup}
	Let $\pi : Y \to \C^{n+1}$ be a log resolution of
	 $(\C^{n+1},f^{-1}(0))$
	 and $I \subset S$ a subset satisfying $|I| \geq 2$.
	Let $\check{\pi} : \check{Y} \to \C^{n+1}$ be the log resolution of
		 $(\C^{n+1},f^{-1}(0))$ obtained by blowing up $Y$
		 along $E_I$.
	Then $$md_m(\check{\pi},f) = md_m(\pi,f).$$
\end{lemma}
\proof
Let $(E_j)_{j \in S}$ be the resolutions divisors of $\pi$.
Let $\check{E}_j$ be the proper transform of $E_j$ in $\check{Y}$
for all $j \in S$.
Then 
\begin{equation} \label{eqn:propertransformdiscrepancy}
a(\check{E}_j) = a(E_j) \ \text{and} \ \text{ord}_f(\check{E}_j) = \text{ord}_f(E_j).
\end{equation}

Let $E$ be the exceptional divisor of the blowdown map
$\check{Y} \lra{} Y$.
Then by looking at a local model of the blowdown
map and using the chain rule we get
\begin{equation} \label{eqn:exceptionaldiscrepancy}
a(E) = |I|-1 + \sum_{j \in I} a(E_j), \quad \text{ord}_f(E) = \sum_{j \in I} \text{ord}_f(E_j).
\end{equation}
Suppose for some $\check{I} \subset S$ satisfying $E_{\check{I}} \neq \emptyset$,
some $k \in \N_{\geq 0}$ and $(k_j)_{j \in \check{I}} \in \N_{>0}^{\check{I}}$
we have
$$k\text{ord}_f(E) + \sum_{j \in \check{I}} k_j \text{ord}_f(\check{E}_j)  = m.$$
Then by Equations (\ref{eqn:propertransformdiscrepancy}) and (\ref{eqn:exceptionaldiscrepancy})
 we have
\begin{equation} \label{eqn:discrepancyinequality}
ka(E) + \sum_{j \in \check{I}} k_j a(\check{E}_j)  
= k(|I|-1) + \sum_{j \in I-\check{I}} ka(E_j) + \sum_{j \in I \cap I'} (k+k_j)a(E_j) + \sum_{j \in \check{I}-I} k_j a(E_j).
\end{equation}
Also Equations 
(\ref{eqn:propertransformdiscrepancy}) and (\ref{eqn:exceptionaldiscrepancy})
tell us
that
\begin{eqnarray} \label{eqn:multiplicityidentity}
k\text{ord}_f(E) + \sum_{j \in \check{I}} k_j \text{ord}_f(\check{E}_j)  = \nonumber \\
\sum_{j \in I-\check{I}} k\text{ord}_f(E_j) + \sum_{j \in I \cap \check{I}} (k+k_j)\text{ord}_f(E_j) + \sum_{j \in \check{I}-I} k_j \text{ord}_f(E_j).
\end{eqnarray}
Equations 
(\ref{eqn:discrepancyinequality}) and
(\ref{eqn:multiplicityidentity})
tell us that
$md_m(\check{\pi},f) = md_m(\pi,f)$.
\qed

\begin{lemma} \label{lemma:multipicitymresolutionindepcedence}
The minimal multiplicity $m$ discrepancy does not depend on the choice of log resolution $\pi : Y \to \C^{n+1}$ of $(\C^{n+1},f^{-1}(0))$.
\end{lemma}
\proof
Let $\pi : Y \to \C^{n+1}$ and $\check{\pi} : \check{Y} \to \C^{n+1}$ be two such resolutions.
Lemma \ref{lemma:mdmafterblowup} tells us that blowing up along strata does not change the minimal multiplicity $m$ discrepancy.
Hence by Lemma \ref{lemma:existenceofmultiplicitymseparating}, we can assume that
$\pi$ and $\check{\pi}$ are multiplicity $m$ separating resolutions.

Since $\pi$ and $\check{\pi}$ are birational morphisms, we have that
there is a birational morphism $\Phi : Y \dashrightarrow \check{Y}$ so that $\pi = \check{\pi} \circ \Phi$.
Let $(E_j)_{j \in S}$ be the resolution divisors of $\pi$ and $(\check{E}_j)_{j \in \check{S}}$ the resolution divisors of $\check{\pi}$.
Suppose that $\text{ord}_f(E_j)$ divides $m$ for some $j \in S - \star_S$.
Let $\check{I} \subset \check{S}$ be the largest subset satisfying $\Phi(E_j) \subset \check{E}_{\check{I}}$.

Since $\Phi$ is well defined outside a subvariety of codimension $\geq 2$, we have a point $p \in E_j$ and holomorphic charts
$y_1,\cdots,y_n$ in $Y$ and $x_1,\cdots,x_n$ in $\check{Y}$ centered at $p$
and $\Phi(p)$ respectively
where $\Phi$ is well defined.
We will also assume that $E_j = \{y_1 = 0\}$
and $\check{E}_k = \{x_{\alpha(k)} = 0\}$ for all $k \in \check{I}$ and for some $\alpha : \check{I} \lra{} \{1,\cdots,n\}$.
Let $J$ be the Jacobian of $\pi(y_1,\cdots,y_n)$,
$\check{J}$ the Jacobian of $\check{\pi}(x_1,\cdots,x_n)$
and $J_\Phi$ the Jacobian of $\Phi(y_1,\cdots,y_n)$.
Then 
\begin{equation} \label{eqn:orderequation}
\text{ord}_f(E_j) = \sum_{k \in \check{I}}\text{ord}_{x_{\alpha(k)} \circ \Phi}(E_k) \text{ord}_f(\check{E}_k)
\end{equation}
and $$\text{ord}_{J}(E_j) = \text{ord}_{J_\Phi}(E_j) + \sum_{k \in \check{I}}\text{ord}_{x_{\alpha(k)} \circ \Phi}(E_k) \text{ord}_{\check{J}}(\check{E}_k).$$
By (\ref{eqn:orderequation})
combined with the fact that $\check{\pi}$ is a multiplicity $m$ separating resolution,
we have that $\check{I} = \{k\}$
for some $k \in \check{S} - \star_{\check{S}}$.
Hence $\text{ord}_f(E_j) = \kappa \text{ord}_f(\check{E}_k)$ and $a(E_j) \geq \kappa a(\check{E}_k)$ where $\kappa = \text{ord}_{x_{\alpha(k)} \circ \Phi}(E_j)$.
Therefore by Lemma \ref{lemma:multiplicitymresolutiondiscrepancycalculation},
$\text{md}_m(\pi,f) \geq \text{md}_m(\check{\pi},f)$.
Similarly $\text{md}_m(\check{\pi},f) \geq \text{md}_m(\pi,f)$ and hence $\text{md}_m(\pi,f) = \text{md}_m(\check{\pi},f)$.
\qed

\section{Liouville domains, Symplectomorphisms and Open Books} \label{section liouville domains etc}

In this section we give basic definitions of Liouville domains
and graded symplectomorphisms and open books.
We will also explain the correspondence between open book decompositions and graded symplectomorphisms of Liouville domains.
All of the material here is contained in \cite{Giroux:openbooks}
with the exception of gradings which is contained
in \cite{Seidel:graded}.
For more details on open book decompositions see \cite{MaxDornerThesis}.

\begin{defn} \label{defn:liouvilledomain}
An {\it exact symplectic manifold} is a pair $(M,\theta_M)$ where $M$ is a manifold and $\theta_M$ is a $1$-form so that $\omega_M \equiv d\theta_M$ is symplectic.
A {\it Liouville domain} is an exact symplectic manifold $(M,\theta_M)$ where $M$ is a compact manifold with boundary
and the $\omega_M$-dual $X_{\theta_M}$ of $\theta_M$ points outwards along $\partial M$.
The $1$-form $\theta_M$ is called the {\it Liouville form}.
The {\it contact boundary} of $M$ is the pair $(\partial M,\alpha_M)$ where
$\alpha_M \equiv \theta_M|_{\partial M}.$
Here $\alpha_M$ is a contact form.
Since $X_{\theta_M}$ points outwards along $\partial M$, we get that the backwards flow
$$(\phi_t : M \hookrightarrow M)_{t \in (-\infty,0]}$$
of $X_{\theta_M}$
exists for all time $t$.
By considering the smooth embeddings
$\phi_{\ln(r_M)}|_{\partial M}$ where $r_M \in (0,1]$,
we can construct a {\it standard collar neighborhood} $(0,1] \times \partial M \subset M$ of $\partial M$ where
$$\theta_M|_{(0,1] \times \partial M} = r_M \alpha_M.$$
Here $r_M$ is the coordinate given by the natural projection
$r_M : (0,1] \times \partial M \twoheadrightarrow (0,1]$ and is called the {\it cylindrical coordinate} on $M$.


An {\it exact symplectomorphism} $\phi : M \lra{} M$ is a diffeomorphism so that $\phi^* \theta_M  = \theta_M + dF_\phi$
for some smooth function $F_\phi : M \to \R$.
Technically we want to think of this as a pair $(\phi,F_\phi)$ but we will suppress $F_\phi$ from the notation and just write $\phi$.
The {\it support} of such an exact symplectomorphism is the region
$$\{ x \in M \ : \ \phi(x) \neq x \quad \text{or} \quad dF_\phi(x) \neq 0\}.$$
\end{defn}

We now need to define graded symplectomorphisms
as in \cite{Seidel:graded}. This is needed so that
we can define their Floer cohomology groups in the next section.

\begin{defn} \label{defn:gradedsymplectomorphism}
We define $(\R^{2n},\Omega_{std})$ to be the standard symplectic vector space.
Let $Sp(2n)$ be the space of linear symplectomorphisms
of $(\R^{2n},\Omega_{std})$
and $\widetilde{Sp}(2n)$ its universal cover.
Let $\pi : E \lra{} V$ be a symplectic vector bundle
with symplectic form $\Omega_E$ whose fibers have dimension $2n$.
Sometimes we will write $(E,\Omega_E)$ or just $E$ for such a symplectic vector bundle when the context is clear.
Define the {\it symplectic frame bundle}
$Fr(E)$ to be an $Sp(2n)$ bundle whose fiber
at $x \in V$ is the 
space of linear symplectomorphisms
from $(\R^{2n},\Omega_{std})$ to $(\pi^{-1}(x),\Omega_E|_x)$.
A {\it grading} on a symplectic vector bundle
$\pi : E \lra{} V$
is an $\widetilde{Sp}(2n)$ bundle
$\widetilde{Fr}(E) \lra{} V$
together with a choice of isomorphism of $\widetilde{Sp}(2n)$ bundles
\begin{equation} \label{defn:grading}
\iota : \widetilde{Fr}(E) \times_{\widetilde{Sp}(2n)} Sp(2n) \cong Fr(E).
\end{equation}
This is just a choice of reduction of the structure group of $E$ from $Sp(2n)$ to $\widetilde{Sp}(2n)$.
A symplectic vector bundle with a choice of grading is called a {\it graded symplectic vector bundle}.
Suppose that $\check{\pi} : \check{E} \lra{} \check{V}$ is a symplectic vector bundle and
$\widetilde{\beta} : \check{E} \lra{} E$ is
a bundle morphism covering a smooth map $\beta : \check{V} \lra{} V$ so that $\widetilde{\beta}$ restricted to each fiber
is a linear symplectomorphism.
Let
$$\widehat{\beta} : \check{E} \lra{} \beta^*(E), \quad \widehat{\beta}(\check{e}) = (\check{\pi}(\check{e}),\widetilde{\beta}(\check{e}))$$
be the natural isomorphism between this bundle and the pullback bundle.
Then
$\check{E}$ has a natural grading
$$\beta^* (\widetilde{Fr}(E)) \lra{} \check{V},$$
$$\check{\iota} : \beta^*(\widetilde{Fr}(E)) \times_{\widetilde{Sp}(2n)} Sp(2n) = 
\beta^* \left(\widetilde{Fr}(E) \times_{\widetilde{Sp}(2n)} Sp(2n)\right) \lra{} Fr(\check{E}),$$
$$ \check{\iota}(\check{v},w) \equiv \widehat{\beta}^{-1}((\check{v},\iota(w))), \ \forall \ \check{v} \in \check{V}, \ w \in \widetilde{Fr}(E) \times_{\widetilde{Sp}(2n)} Sp(2n)$$
called the {\it grading on $E$ pulled back by $\widetilde{\beta}$}.
\end{defn}

The simplest example of a grading
is the natural grading on the trivial bundle
$\check{V} \times \C^k \lra{} \check{V}$ given by
$$\widetilde{Fr}(\check{V} \times \C^k) \equiv \check{V} \times \widetilde{Sp}(2n),$$
$$\iota : \widetilde{Fr}(\check{V} \times \C^k) \times_{\widetilde{Sp}(2n)} Sp(2n) =
\check{V} \times (\widetilde{Sp}(2n) \times_{\widetilde{Sp}(2n)} Sp(2n)) = \check{V} \times Sp(2n) \to Fr(\check{V} \times \C^k).$$
We will call such a grading the {\it trivial grading}.

\begin{remark}
In this paper we are only interested in gradings up to isotopy and so we will sometimes
regard isotopic gradings as the same grading.
If we have a smooth connected family of symplectic vector bundles then
if one of them has a grading then all of them have a natural choice of grading up to isotopy.
For more information about this see Appendix A.	
\end{remark}
\begin{defn}
Suppose that
$$\pi_1 : E_1 \lra{} V_1, \quad \pi_2 : E_2 \lra{} V_2$$
are graded symplectic vector bundles
and $F : E_1 \lra{} E_2$ is a symplectic vector bundle
isomorphism covering a diffeomorphism $V_1 \lra{} V_2$.
Then a {\it grading} on $F$ is an
$\widetilde{Sp}(2n)$ bundle isomorphism
$$\widetilde{F} : \widetilde{Fr}(E_1) \lra{} \widetilde{Fr}(E_2)$$
covering the $Sp(2n)$ bundle morphism
$Fr(E_1) \to Fr(E_2)$
induced by $F$.

Let $(X,\omega_X)$ be a $2n$ dimensional symplectic manifold.
A {\it grading}
on $(X,\omega)$ is a choice of grading on the symplectic vector bundle $TX$.
A symplectic manifold with a choice of grading is called
a {\it graded symplectic manifold}

A {\it grading} of a symplectomorpism 
$\phi$ between two graded symplectic manifolds
$(X_1,\omega_{X_1})$, $(X_2,\omega_{X_2})$
is a choice of grading for the symplectic bundle isomorphism
$d\phi : TX_1 \lra{} TX_2$.
A {\it graded symplectomorphism} 
is a symplectomorphism with a choice of grading.
\end{defn}

We also wish to have a notion of grading for contact manifolds.

\begin{defn}
Recall that a {\it cooriented contact manifold} $(C,\xi_C)$ is a manifold $C$ of dimension $2n-1$ with a cooriented hyperplane distribution
$\xi_C$ with the property that there is a $1$-form $\alpha_C$ whose kernal is $\xi_C$ respecting this coorientation
and so that $\alpha_C \wedge (d\alpha_C)^{n-1}$ is a volume form on $C$.
The $1$-form $\alpha_C$ is called a {\it contact form compatible with $\xi_C$}.
A {\it coorientation preserving contactomorphism} between two cooriented contact manifolds is a diffeomorphism preserving their respective hyperplane distributions
and coorientations.
A {\it contact submanifold} $B \subset C$
is a submanifold so that $(B,\xi_B \equiv \xi_C \cap TB)$ is a cooriented contact manifold with the induced coorientation from $\xi_C$.

A {\it grading} on a contact manifold $(C,\xi_C)$
consists of a choice of contact form $\alpha_C$ compatible
with $\xi_C$ and a choice of grading on the symplectic
vector bundle $(\xi_C,d\alpha_C|_{\xi_C})$.
A cooriented contact manifold with a choice of grading is called
a {\it graded contact manifold.}
Since the space of contact forms compatible with $\xi_C$ is contractible,
we get an induced grading on $(\xi_C, d\alpha|_{\xi_C})$ for any other contact form $\alpha$
compatible with $\xi_C$ which is unique up to isotopy.
Hence from now on we will regard this as the same grading.

A {\it grading} of a coorientation preserving contactomorphism $\phi$ between two graded contact manifolds
$(C_1,\xi_1),(C_2,\xi_2)$
consists of a grading on the symplectic vector bundle isomorpism
$d\phi|_{\xi_{C_1}} : \xi_{C_1} \lra{} \xi_{C_2}$
where the symplectic forms on $\xi_{C_1}$ and $\xi_{C_2}$ come from
contact forms $\alpha_{C_1}$ and $\alpha_{C_2}$  compatible with $\xi_{C_1}$ and $\xi_{C_2}$ 
respectively satisfying $\alpha_{C_1} = \phi^* \alpha_{C_2}$ with induced gradings.
A {\it graded contactomorphism} is a coorientation preserving contactomorphism contactomorphism with a choice of grading. 
\end{defn}

In this paper, we will only deal cooriented contact manifolds and coorientation preserving contactomorphisms. So from now on, a {\it contact manifold} is a cooriented contact manifold and a {\it contactomorphism} is a coorientation preserving contactomorphism.

\begin{defn}
Suppose that $(C,\xi_C)$ is a contact manifold and $B$ is a contact submanifold.
The {\it normal bundle of $B$} is a symplectic vector bundle
$$ \pi_{\cN_C B} : \cN_C B \equiv (TC|_B)/TB \twoheadrightarrow B$$
with symplectic form defined as follows:
Choose a compatible contact form $\alpha_C$ on $(C,\xi_C)$
and define $$T^\perp B \equiv \{ v \in \xi_C|_x \ : \ x \in B, \quad d\alpha_C(v,w) = 0, \quad \forall w \in \xi_C|_x \cap TB|_x.\}$$
Then $T^\perp B$ is a symplectic vector bundle with symplectic form $d\alpha_C|_{T^\perp B}$.
The symplectic structure on $\cN_C B$ is the pushforward of the above symplectic form
under the natural bundle isomorphism $T^\perp B \to \cN_C B$.
\end{defn}

Since the space of compatible contact forms is contractible,
we have that the choice of symplectic form on the $\cN_C B$ is unique up to symplectic bundle isomorphism
and any two choices of such isomorphism are homotopic.
As a result, we will refer to this bundle as {\it the} normal bundle
as we are only concerned with isomorphisms of such bundles up to homotopy.

\begin{defn} \label{defn:contactpair}
A {\it contact pair with normal bundle data} $(B \subset C,\xi_C,\Phi_B)$ consists of a contact manifold $(C,\xi_C)$ where $B$ is a
codimension $2$ contact submanifold along with a symplectic trivialization $$\Phi_B : \cN_C B \to B \times \C$$ of its normal bundle.
A {\it contactomorphism} between two such triples
\begin{equation} \label{eqn:pairofpairs}
(B_1 \subset C_1,\xi_{C_1},\Phi_{B_1}), \quad (B_2 \subset C_2,\xi_{C_2},\Phi_{B_2})
\end{equation}
is a  contactomorphism $\Psi : C_1 \to C_2$ sending $B_1$ to $B_2$
so that the composition
$$\cN_{C_1} B_1 \overset{d\Psi|_{B_1}}{\xrightarrow{\hspace*{0.6cm}}} \cN_{C_2} B_2 \stackrel{\Phi_{B_2}}{\lra{}} B_2 \times \C  \overset{(\Psi|_{B_1})^{-1} \times \text{id}_\C}{\xrightarrow{\hspace*{1.5cm}}}
 B_1 \times \C$$
is homotopic through symplectic bundle trivializations to $\Phi_{B_1}$.
If there exists such a contactomorphism between the pairs
as in Equation (\ref{eqn:pairofpairs}) then we say that they are {\it contactomorphic}.

A contact pair with normal bundle data $(B \subset C,\xi_C,\Phi_B)$ is {\it graded}
if there is a choice of grading on $C - B$.
A {\it graded contactomorphism} between two graded contact pairs with normal bundle data as in Equation (\ref{eqn:pairofpairs})
consists of a contactomorphism $\Psi$
between these contact pairs
and a choice of grading of the contactomorphism $\Psi|_{C_1 - B_1} : C_1 - B_1 \lra{} C_2 - B_2$.
If a graded contactomorphism exists between two
graded contact pairs with normal bundle data
then we say that they are {\it graded contactomorhic}.
\end{defn}

The main example of a contact pair with normal bundle data comes from singularity theory.
\begin{example} \label{example:contactpairoff}
Fix $n > 0$.
Let $f : \C^{n+1} \lra{} \C$ be a holomorphic function with an isolated singularity at $0$
and let $J_0 : T\C^{n+1} \lra{} T\C^{n+1}$ be the standard complex structure on $\C^{n+1}$.
Let $S_\epsilon \subset \C^{n+1}$ be the sphere of radius $\epsilon>0$
and let $\xi_{S_\epsilon} = TS_\epsilon \cap J_0 TS_\epsilon$ be the standard contact structure on $S_\epsilon$.
Define $L_f \equiv f^{-1}(0) \cap S_\epsilon$.
Then a result by Varchenko in \cite{Varchenko:isolatedsingularities}
tells us that for all sufficiently small $\epsilon > 0$,
$L_f \subset S_\epsilon$ is a contact submanifold called the {\it link} of $f$ at $0$.
Also $df$ gives us an induced map $\overline{df} : \cN_{S_\epsilon} L_f \lra{} \C$
and hence a trivialization
$$\Phi_f \equiv (id_{L_f}, \overline{df}) : \cN_{S_\epsilon} L_f \lra{} L_f \times \C.$$
The contact pair with normal bundle data $(L_f \subset S_\epsilon, \xi_{S_\epsilon}, \Phi_f)$
is called the {\it contact pair associated to $f$}.

We also need a grading on this contact pair.
It turns out, by the discussion in Definition
\ref{defn:gradingtrivializationcorrespondence},
that every contact manifold with trivial first and second homology group has a unique
grading up to homotopy.
This means that $\xi_{S_\epsilon}$ has a grading
giving us an induced grading
on the contact pair associated to $f$.
We will call this the {\it standard grading}.
\end{example}

We will now define open book decompositions and also graded open book decompositions.
Let $\D \subset \C$ be the unit disk with polar coordinates $(r,\vartheta)$.

\begin{defn} \label{defn:openbook}
	An {\it open book} is a pair $(C, \pi)$ where 
	
	\begin{itemize}
		\item $C$ is a smooth manifold,
		\item $\pi : C - B \lra{} \R / \Z$ is a smooth fibration where
		$B$ is a codimension $2$ submanifold and
		\item there is a tubular neighborhood $B \times \D \subset C$ of $B = B \times \{0\}$ in $C$
		so that $\pi$ satisfies:
		$$\pi|_{B \times (\D - 0)} : B \times (\D - 0) \lra{} \R / \Z, \quad \pi(x,(r,\theta)) = 2\pi \theta.$$
	\end{itemize}
	The submanifold $B$ is called the {\it binding} of the open book
	and a {\it page} of the open book is the closure of a fiber of $\pi$
	which is a submanifold with boundary equal to $B$.
\end{defn}

We now want open books to be compatible with contact pairs.

\begin{defn}
A contact pair with normal bundle data $(B \subset C,\xi_C,\Phi_B)$
is {\it supported by an open book} $(C, \pi)$
\begin{enumerate}
\item \label{item:openbookliouvilledomaincondition}
if there is a contact form $\alpha_C$ compatible with $\xi_C$
so that $d\alpha_C|_{\pi^{-1}(t)}$ is a symplectic form for all $t \in \R / \Z$,
\item \label{item:openbookorientationcondition}
the trivialization of $\cN_C B$ induced by
the choice of tubular neighborhood from Definition \ref{defn:openbook}
is homotopic through orientation preserving bundle trivializations
to $\Phi_B$
(this does not depend on the choice of such a tubular neighborhood).
\end{enumerate}
We will write $(C, \xi_C, \pi)$ for a contact pair supported by an open book and we will call it a
{\it contact open book}.
Note that $B$ and $\Phi_B$ are not included in the notation as 
$B = C - \Dom(\pi)$ and $\Phi_B$
is determined by the open book
due to the fact that the space of orientation preserving trivializations of $\cN_C B$
is weakly homotopic to the space of symplectic trivializations of $\cN_C B$.
The contact pair
$(B \subset C, \xi_C,\Phi_B)$ is called the {\it contact pair associated to $(C,\xi_C,\pi)$.}
If the contact pair associated to $(C,\xi_C,\pi)$ is graded then we call this
a {\it graded contact open book.}

An {\it isotopy} between two contact open books
$(C_1, \xi_{C_1}, \pi_1)$,
$(C_2, \xi_{C_2}, \pi_2)$
is a contactomorphism $\Phi : C_1 \lra{} C_2$
between the respective contact pairs with normal bundle data
together with a smooth family of maps $(\check{\pi}_t : \Dom(\pi_1) \lra{} \R / \Z)_{t \in [1,2]}$
joining $\pi_1$ and $\pi_2 \circ \Phi$ so that
$(C_1, \xi_{C_1}, \check{\pi}_t)$ is a contact open book for all $t \in [1,2]$.
Such an isotopy is {\it graded} if we have a smooth family of graded contact open books
and $\Phi$ is a graded contactomorphism.
\end{defn}

The main example of a contact open book will come from singularity theory.
\begin{example} \label{example:milnoropenbook}
Let $f : \C^{n+1} \lra{} \C$, $n > 1$ be a holomorphic function with an isolated singularity at $0$
and let $(L_f \subset S_\epsilon, \xi_{S_\epsilon}, \Phi_f)$
be the contact pair associated to $f$ as in Example \ref{example:contactpairoff}.
Let $arg(f) : \C^{n+1} - f^{-1}(0) \lra{} \R / 2\pi\Z$ be the argument of $f$.
Then by \cite[Proposition 3.4]{CaublePopescuNemethi:milnoropen},
we have that
$(S_\epsilon, \xi_{S_\epsilon}, \frac{1}{2\pi}\arg(f)|_{S_\epsilon})$
is a contact open book
for all $\epsilon > 0$ small enough.
This open book is supported by the contact pair associated to $f$ from Example \ref{example:contactpairoff}
and it has a grading coming from the standard grading.
This is a graded contact open book
called the {\it Milnor open book of $f$}.
\end{example}

\begin{defn}
An {\it abstract contact open book} consists of a triple $(M,\theta_M,\phi)$
where $(M,\theta_M)$ is a Liouville domain and $\phi$ is an exact symplectomorphism
with support in the interior of $M$.
A {\it graded abstract contact open book} is an abstract contact open book $(M,\theta_M,\phi)$
with a choice of grading on $(M,d\theta_M)$ and a graded symplectomorphism $\phi$.

An {\it isotopy} between abstract contact open books $(M_1,\theta_{M_1},\phi_1)$, $(M_2,\theta_{M_2},\phi_2)$
consists of a diffeomorphism $\Phi : M_1 \lra{} M_2$,
a smooth family of $1$-forms $(\theta_t)_{t \in [0,1]}$ joining $\theta_{M_1}$ and $\Phi^* \theta_{M_2}$
and a smooth family of diffeomorphisms $(\check{\phi}_t)_{t \in [1,2]}$ joining $\phi_1$ and $\Phi^{-1} \circ \phi_2 \circ \Phi$
so that $(M_1,\theta_t,\check{\phi}_t)$ are all abstract contact open books
and the support of $\check{\phi}_t$ is contained inside a fixed compact subset of the interior of $M$.
If both of our abstract contact open books are graded then
such an isotopy is a {\it graded isotopy} if all the abstract contact open books
$(M_1,\theta_t,\check{\phi}_t)$ are graded so that these gradings smoothly
depend on $t \in [0,1]$
 and the grading on
$(M_0,\theta_0,\check{\phi}_0)$ coincides with the grading on $(M_1,\theta_{M_1},\phi_1)$
and the grading on $(M_1,\theta_1,\check{\phi}_1)$
coincides with the grading on $(M_2,\theta_{M_2},\phi_2)$ pulled back by $\Phi$.
\end{defn}

From an abstract contact open book $(M,\theta_M,\phi)$,
we wish to construct a contact open book.
This construction is referred to as a {\it generalized Thurston-Winkelnkemper construction}
in \cite[Section 2.2.1]{MaxDornerThesis}.
To do this we need the following definition.
\begin{defn} \label{defn:mappingtorusofphi}
	Let $(M,\theta_M,\phi)$ be an abstract contact open book.
	Let $F_\phi : M \lra{} \R$ be the smooth function with support
	in the interior of $M$ satisfying $\phi^* \theta_M = \theta_M + dF_\phi$.
	Let $\rho : [0,1] \lra{} [0,1]$ be a smooth function equal to
	$0$ near $0$ and $1$ near $1$.

	The {\it mapping torus} of $(M,\theta_M,\phi)$ is a smooth map
	$$\pi_{T_\phi} : T_\phi \lra{} \R / \Z, \quad
	T_\phi \equiv M \times [0,1] / \sim$$
	together with a contact form $\alpha_{T_\phi}$ on $T_\phi$
	where
	\begin{itemize}
	\item $\sim$ identifies $(x,1)$ with $(\phi(x),0)$,
	\item $\alpha_{T_\phi} \equiv \theta_M + d(\rho(t)F_\phi) + C dt$
	where $C>0$ is large enough to ensure that $\alpha_{T_\phi}$ is a contact form
	and
	\item
	$\pi_{T_\phi}(x,t) \equiv t \quad \forall (x,t) \in T_\phi = M \times [0,1] / \sim.$
	\end{itemize}
	
	For $\delta>0$ small enough, we have that
	the subset
	$$(1-\delta,1] \times \partial M \subset (0,1] \times \partial M$$
	of the collar neighborhood of $\partial M$
	is disjoint from the support of $\phi$.
	This means that there is a natural embedding:
	$$(1-\delta,1] \times \partial M \times \R / \Z \subset T_\phi$$
	which we will call the {\it standard collar neighborhood of $\partial T_\phi$}.
	Note that $\alpha_{T_\phi}$ is equal to
	$r_M \alpha_M + C dt$ in the standard collar neighborhood where $r_M$ (resp. $t$) is the natural projection map to $(1-\delta,1]$ (resp. $\R / \Z$) and $\alpha_M = \theta_M|_{\partial M}$.
	
	If $(M,\theta_M,\phi)$ is a graded abstract contact open book
	then we get a natural grading on $(T_\phi,\ker(\alpha_{T_\phi}))$
	as follows:
	Since the kernal of $d\alpha_{T_\phi}$
	is transverse to the fibers of $T_\phi$,
	$\xi_{T_\phi} \equiv \ker(\alpha_{T_\phi})$
	is isotopic through hyperplane distributions $H_t, \ t \in [0,1]$ to the vertical tangent space $T^{\text{ver}} T_\phi \equiv \ker(D\pi_\phi)$ of $\pi_{T_\phi}$
	with the property that $d\alpha_{T_\phi}|_{H_t}$
	is non-degenerate for all $t \in [0,1]$.
	Therefore it is sufficient to construct a grading
	for the symplectic vector bundle $(T^{\text{ver}} T_\phi,d\alpha_{T_\phi})$.
	Consider the symplectic vector bundle
	$(pr^* TM,pr^* (d\theta_M))$ where $$pr : M \times [0,1] \lra{} M$$
	is the natural projection map.
	We have that the symplectic vector bundle $(T^{\text{ver}} T_\phi,d\alpha_{T_\phi})$
	is isomorphic to the symplectic vector bundle on
	$(pr^*TM/\sim,pr^* (d\theta_M))$
	where $\sim$ identifies $pr^* TM|_{M \times \{1\}}$ with $pr^* TM|_{M \times \{0\}}$
	using the map $$D\phi : TM = pr^* TM|_{M \times \{1\}} \lra{} TM = pr^* TM|_{M \times \{0\}}.$$
	Since $(TM,d\theta_M)$ has a natural grading,
	we get that $(pr^* TM,pr^* (d\theta_M))$ has an induced grading
	$$\widetilde{Fr}(pr^* TM) \times_{\widetilde{Sp}(2n)} Sp(2n) \cong Fr(pr^* TM)$$	
	pulled back via $pr$.
	The map $D\phi$ gives us an induced map
	$$Fr(pr^* TM)|_{M \times \{1\}} \lra{} Fr(pr^* TM)|_{M \times \{0\}}$$
	and since $\phi$ is a graded symplectomorphism, the map above lifts to a map:
	$$\widetilde{Fr}(pr^* TM)|_{M \times \{1\}} \lra{} \widetilde{Fr}(pr^* TM)|_{M \times \{0\}}.$$
	Hence by using the above gluing map, we get a grading on the quotient
	$(pr^* TM / \sim,pr^* (d\theta_M))$
	and therefore a grading on
	$(T^{\text{ver}} T_\phi,d\alpha_{T_\phi})$.
	In turn this gives us a grading on the contact manifold
	$(T_\phi, \ker(\alpha_{T_\phi}))$.
	We will call this the {\it induced grading on} $T_\phi$.
\end{defn}

We will now construct a contact open book from an abstract contact open book.
Let $\D(\rho) \subset \C$ be the closed disk of radius $\rho>0$
with polar coordinates $(r,\vartheta)$.

\begin{defn} \label{defn:openbookfromabstractone}
	Let $(M,\theta_M,\phi)$ be an abstract contact open book decomposition
	and let $$\pi_{T_\phi} : T_\phi \lra{} \R / \Z$$
	be the associated mapping torus with contact form $\alpha_\phi$ and
	standard collar neighborhood
	$$(1-\delta,1] \times \partial M \times \R / \Z \subset T_\phi$$
	as in Definition \ref{defn:mappingtorusofphi}.
	Define
	$C_\phi \equiv (\partial M \times \D(\delta)) \sqcup T_\phi / \sim$
	where
	$\sim$ identifies
	$(x,(r,\vartheta)) \in \partial M \times \D(\delta)$
	with
	$$(1-r,x,\frac{\vartheta}{2\pi}) \in
	(1-\delta,1] \times \partial M \times \R / \Z \subset T_\phi.$$
	
	We define $B_\phi \equiv \partial M \times \{0\} \subset \partial M \times \D(\delta) \subset C_\phi$
	and
	$$\pi_{C_\phi} : C_\phi - B_\phi = T_\phi - \partial T_\phi \lra{} \R / \Z, \quad
	\pi_{C_\phi} = \pi_{T_\phi}|_{T_\phi - \partial T_\phi} .$$

	Let
	\begin{equation} \label{eqn:alphaequation}
	h_1,h_2 : [0,\delta) \lra{} \R
	\end{equation}
	be a pair of smooth functions so that
	\begin{itemize}
	\item $h_1'(r) < 0$, $h_2'(r) \geq 0$ for all $r > 0$,
	\item $h_1(r) = 1-r^2$ and $h_2(r) = \frac{1}{2} r^2$ for small $r$ and
	\item $h_1(r) = 1-r$ and $h_2(r) = C/2\pi$ for $r \in [\frac{\delta}{2},\delta)$.
	\end{itemize}
	
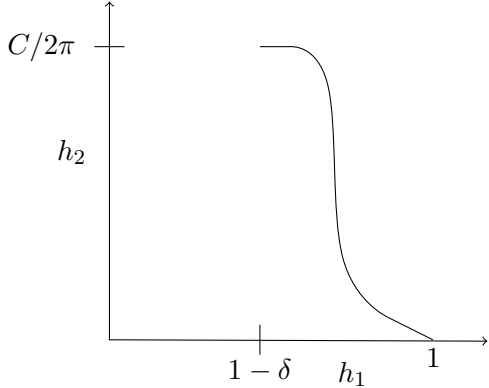
\begin{tikzpicture}

\draw[<->] (3.5238,-1.0181) -- (-1.5,-1) -- (-1.5,3.5);

\node at (-2,1.5) {$h_2$};
\node at (1.732,-1.4537) {$h_1$};

\node at (2.8046,-1.2387) {$1$};

\node at (0.5,-1.4) {$1-\delta$};

\draw (0.5,-1.2) -- (0.5,-0.8);

\draw (0.9,2.9) node (v1) {} -- (0.5,2.9);

\node at (-2.4,2.9) {$C/2\pi$};

\draw  plot[smooth, tension=.7] coordinates {(v1) (1.4,2.4) (1.6,0.1) (2.2,-0.7)};

\draw (2.2,-0.7) -- (2.8,-1);
\draw (-1.7,2.9) -- (-1.3,2.9);

\end{tikzpicture}
	
	The above conditions ensure that
	$$\alpha_{C_\phi} \equiv
	\left\{
	\begin{array}{cl}
	 h_1(r) \alpha_M +
	h_2(r) d\vartheta & \text{inside} \ \partial M \times \D(\delta/2) \\
	\alpha_{T_\phi} & \text{inside} \ T_\phi - \partial M \times \D(\delta/2).
	\end{array}
	\right.
	$$
	is a contact form whose restriction to $B_\phi$
	is also a contact form and that
	the restriction of $d\alpha_{C_\phi}$ to $\pi^{-1}(t)$
	is symplectic for all $t \in \R/ \Z$.
	Define $\xi_{C_\phi} \equiv \ker(\alpha_{C_\phi}).$
		The tubular neighborhood $\partial M \times \D(\delta/2)$
	of $B_\phi$ gives us a trivialization $\Phi_{B_\phi}$
	of its normal bundle and hence
 we get a contact pair with normal bundle data
	$(B_\phi \subset C_\phi,\xi_{C_\phi},\Phi_{B_\phi})$
	which we will called the {\it contact pair associated to $(M,\theta_M,\phi)$}.
	This contact pair with normal bundle data
	is supported by the open book
	$(C_\phi,\pi_{C_\phi})$.
	Hence
	$$OBD(M,\theta_M,\phi) \equiv (C_\phi,\xi_{C_\phi},\pi_{C_\phi})$$
	is a contact open book which we call
	the {\it open book associated to
	the abstract contact open book $(M,\theta_M,\phi)$}.

	Now suppose that $(M,\theta_M,\phi)$ is a graded abstract contact open book.
	Then since the contact manifold $(C_\phi - B_\phi, \xi_{C_\phi}|_{C_\phi - B_\phi})$
	is isotopic through contact manifolds to $T_\phi$,
	we get that the induced grading on $T_\phi$
	gives us a grading on $(C_\phi - B_\phi, \xi_{C_\phi}|_{C_\phi - B_\phi})$.
	Hence we have a relative grading on $(B_\phi \subset C_\phi, \xi_{C_\phi},\Phi_\phi)$ which we will call the
	{\it induced grading}.
\end{defn}

It is fairly straightforward to show that if two abstract contact open books are (graded) isotopic
then their respective contact open book decompositions are (graded) isotopic.
Hence we get a map
$$OBD : \left\{\begin{array}{c}
\text{(graded) abstract contact} \\
\text{open books}
\end{array}\right\} / \text{isotopy} \lra{} \left\{\begin{array}{c} \text{(graded)  open books }\end{array}\right\} / \text{isotopy}.$$

The Theorem below is a result of Giroux in \cite{Giroux:openbooks}.
\begin{theorem} \label{Theorem:obdisomorphism}
The above map $OBD$ is an bijection.
\end{theorem}
A detailed construction of the inverse of $OBD$ is contained in the proof of \cite[Theorem 3.1.22]{MaxDornerThesis}.
As a result of this Theorem, we have the following definition:
\begin{defn}
	The {\it monodromy map} of a (graded) contact open book
	$(C, \xi_C, \pi)$
	is defined to be the (graded) contactomorphism
	$\phi : M \lra{} M$
	where $(M,\theta_M,\phi)$ is the abstract contact open book
	$OBD^{-1}\left((C, \xi_C, \pi)\right)$.
\end{defn}

Technically this monodromy map is only defined up to isotopy, and so
the monodromy map is really just a choice of representative in this isotopy class.


\section{Fixed Point Floer Cohomology Definition} \label{section fixed point floer cohomology definition}


\begin{defn} \label{defn:definitionfixedpoints}
Let $(M,\theta_M)$ be a Liouville domain.
An almost complex structure $J$ on $M$ is {\it cylindrical near $\partial M$}
if it is compatible with the symplectic form $d\theta_M$ (I.e. $d\theta_M( \cdot, J(\cdot))$ is a Riemannian metric)
and if $dr_M \circ J = -\alpha_M$ near $\partial M$ inside the standard collar neighborhood $(0,1] \times \partial M$.

An exact symplectomorphism $\phi : M \to M$ is {\it non-degenerate}
if for every fixed point $p$ of $\phi$ the linearization of $\phi$ at $p$ as no eigenvalue equal to $1$.
It has {\it small positive slope} if it is equal to the time $1$
Hamiltonian flow of $\delta r_M$ near $\partial M$
where $\delta>0$ is smaller than the period of the smallest
periodic Reeb orbit of $\alpha_M$ (this means that it corresponds to the time $\delta$ Reeb flow near $\partial M$).
If $\phi$ is an exact symplectomorphism, then a {\it small positive slope perturbation} $\check{\phi}$ of $\phi$
is an exact symplectomorphism $\check{\phi}$ equal to the composition of $\phi$ with a $C^\infty$ small Hamiltonian symplectomorphism of small positive slope.
The {\it action} of a fixed point $p$ is $-F_\phi(p)$
where $F_\phi$ is a function satisfying $\phi^* \theta_M = \theta_M + dF_\phi$.
The action depends on a choice of $F_\phi$ which has to be fixed when $\phi$ is defined although usually $F_\phi$ is chosen so that it is zero near $\partial M$ (if possible).
All of the symplectomorphisms coming from isolated hypersurface singularities will have such a unique $F_\phi$.
An {\it isolated family of fixed points}
is a path connected compact subset $B \subset M$ consisting of fixed points of $\phi$ of the same action
and for which there is a neighborhood $N \supset B$ where $N \setminus B$ has no fixed points.
Such an isolated family of fixed points is called a {\it codimension $0$ family of fixed points}
if in addition there is an autonomous Hamiltonian $H : N \to (-\infty,0]$ so that $H^{-1}(0) = B$ is a connected codimension $0$ submanifold with boundary and corners,
the time $t$ flow of $X_H$ is well defined for all $t \in \R$
and $\phi|_N : N \to N$ is equal to the time $1$ flow of $H$.
The {\it action} of an isolated family of fixed points $B \subset M$ is the action of any point $p \in B$.
\end{defn}

Before we define Floer cohomology, we need some definitions so that we can give it a grading.
For any path of symplectic matrices $(A_t)_{t \in [a,b]}$
we can assign an index $CZ(A_t)$
called its {\it Conley-Zehnder index}.
The Conley-Zehnder index of certain paths of symplectic matrices $A_t$ was originally
defined in \cite{ConleyZehnder:Index}. It was defined for a general path of symplectic matrices
in \cite{RobbinSalamon:maslov} and also in \cite{Gutt:GeneralizedConleyZehnder}.
We will not define it here as we will only use the following properties (see \cite[Proposition 8]{Gutt:GeneralizedConleyZehnder}):
\begin{CZ}
\item \label{item:cznormalization}
 $CZ((e^{it})_{t \in [0,2\pi]}) = 2$.
 \item \label{item:signproperty}
  $(-1)^{n-CZ((A_t)_{t \in [0,1]})} = \text{sign det}_\R(\text{id} - A_1)$ for any path of symplectic matricies $(A_t)_{t \in [0,1]}$.
\item \label{item:czadditive}
$CZ(A_t \oplus B_t) = CZ(A_t) + CZ(B_t)$.
\item  \label{item:catenationczproperty}
The Conley-Zehnder index of the catenation of two paths is the sum of their Conley-Zehnder indices.
\item \label{item:homotopyinvariance}
If $A_t$ and $B_t$ are two paths of symplectic matrices which are homotopic relative to their endpoints
then they have the same Conley-Zehnder index.
Also such an index only depends on the path up to orientation
preserving reparameterization.
\end{CZ}

\begin{defn}
Let $(M,\theta_M,\phi)$ be a graded abstract contact open book.
The {\it Conley-Zehnder index} $CZ(p)$ of a fixed point $p$ of the graded symplectomorphism $\phi$
is defined as follows:
Since $(TM,d\theta_M)$ is a graded symplectic vector bundle,
we have an associated
$\widetilde{Sp}(2n)$ bundle
$$\widetilde{Fr}(TM) \lra{} M$$
together with a choice of isomorphism of $\widetilde{Sp}(2n)$ bundles
$$\iota : \widetilde{Fr}(TM) \times_{\widetilde{Sp}(2n)} Sp(2n) \cong Fr(TM).$$
%
Now choose an identification of $\widetilde{Sp}(2n)$ torsors
\begin{equation} \label{eqn:identificationofcovers}
\widetilde{Sp}(2n) = \widetilde{Fr}(TM)|_p.
\end{equation}
The symplectomorphism $\phi$ has a choice of grading giving
us a natural map:
$$\widetilde{\phi} : \widetilde{Fr}(TM)|_p
\lra{} \widetilde{Fr}(TM)|_p$$
and hence by Equation (\ref{eqn:identificationofcovers}),
a map
$$\widetilde{\phi} : \widetilde{Sp}(2n) \lra{} \widetilde{Sp}(2n).$$
Since $\widetilde{Sp}(2n)$ is the universal cover of $Sp(2n)$,
its elements correspond to paths of symplectic matrices starting
from the identity up to homotopy fixing their endpoints and so let
$\beta$ be the path corresponding to 
$\widetilde{\phi}(id) \in \widetilde{Sp}(2n)$.
We define $CZ(\phi,p) \equiv CZ(\beta)$.

If we have an isolated family of fixed points $B \subset X$,
then we define its {\it Conley-Zehnder index} $CZ(\phi,B)$ to be the Conley-Zehnder index
of some $b \in B$. Since $B$ is path connected,
this does not depend on the choice of $b \in B$ by property \ref{item:homotopyinvariance} above.
\end{defn}

In Appendix A, we also have a different way of computing the Conley-Zehnder index by looking at the mapping torus $T_\phi$ of $\phi$. This will be useful in the proof of Theorem \ref{theorem dynamical properties of monodromy}.

\begin{defn}
Let $(M,\theta_M,\phi)$ be an abstract contact open book.
Let $(J_t)_{t \in [0,1]}$ be a smooth family of almost complex structures with the property that
$\phi^* J_0 = J_1$.
A {\it Floer trajectory} of $(\phi,J_t)_{t \in [0,1]})$ joining $p_-,p_+ \in M$ is a smooth map $u : \R \times [0,1] \to M$
so that $\partial_s u + J_t \partial_t u = 0$ where $(s,t)$ parameterizes $\R \times \R/\Z$, $u(s,0) = \phi(u(s,1))$
and so that $\lim_{s \to \pm\infty} u(s,t) = p_{\pm}$ for all $t \in [0,1]$.
We write ${\cM}(\phi,J_t,p_-,p_+)$ for the set of such Floer trajectories
and define $\overline{\cM}(\phi,J_t,p_-,p_+) \equiv {\cM}(\phi,J_t,p_-,p_+)/\R$,
where $\R$ acts by translation in the $s$ coordinate.
\end{defn}

Let $(M,\theta_M,\phi)$ be a graded abstract contact open book.
We will now give a sketch of the definition
of the {\it Floer cohomology} group $HF^*(\phi,+)$
(see \cite{Seidel:morevanishing}).
Let $\check{\phi}$ be a small positive slope perturbation of $\phi$.
%
This can be done so that $\check{\phi}$ is $C^\infty$ close to $\phi$ and so that the fixed points of $\check{\phi}$ are non-degenerate (see \cite[Theorem 9.1]{SZ:morsetheory} in the case where $\check{\phi}$ is Hamiltonian. The general case is similar \cite[Page 586]{DostoglouSalamon:instantonscurves}).
We can also ensure that $\check{\phi}$ is a graded symplectomorphism
due to the fact that it is isotopic to $\phi$ through symplectomorphisms.

We now choose a $C^\infty$ generic family of cylindrical almost
complex structures $(J_t)_{t \in [0,1]}$
satisfying $\phi^* J_0 = J_1$.
The genericity property then tells us that
$\overline{\cM}(\check{\phi},J_t,p_-,p_+)$ is a compact oriented manifold of dimension $0$
for all fixed points $p_-,p_+$ of $\phi$ satisfying $CZ(p_-) - CZ(p_+) = 1$ (\cite[Theorem 3.2]{DostoglouSalamon:instantonscurves}).
We define $\#^\pm\overline{\cM}(\check{\phi},J_t,p_-,p_+)$  to be the signed count of elements of
$\overline{\cM}(\check{\phi},J_t,p_-,p_+)$.
Let $CF^*(\check{\phi})$
be the free abelian group generated by fixed points of $\phi$ and graded by the Conley-Zehnder index {\it taken with negative sign}.
The differential $\partial_{\check{\phi},(J_t)_{t \in [0,1]}}$ on $CF^*(\check{\phi})$ is a $\Z$-linear map
satisfying $\partial_{\check{\phi},(J_t)_{t \in [0,1]}}(p_+) = \sum_{p_-} \#^\pm\overline{\cM}(J_t,p_-,p_+) p_-$ for all fixed points $p_+$ of $\check{\phi}$ where the sum is over all fixed points $p_-$ satisfying $(-CZ(p_-)) - (-CZ(p_+)) = 1$.
Because $(J_t)_{t \in [0,1]}$ is $C^\infty$ generic,
one can show that $\partial_{\check{\phi},(J_t)_{t \in [0,1]}}^2=0$
(\cite[Theorem 3.3(1)]{DostoglouSalamon:instantonscurves})
and we define the resulting homology group to be $HF^*(\check{\phi},(J_t)_{t \in [0,1]})$.
We define $HF^*(\phi,+) \equiv HF^*(\check{\phi},(J_t)_{t \in [0,1]})$.
This does not depend on the choice of perturbation $\check{\phi}$
or on the choice of almost complex structure $(J_t)_{t \in [0,1]}$
(\cite[Theorem 3.3(2)]{DostoglouSalamon:instantonscurves}).
Our conventions then tell us that if $\phi : M \lra{} M$ is the identity map with the trivial grading and $\text{dim}(M) = n$
then $HF^k(\phi,+) = H^{n+k}(M;\Z)$.

We will {\it only} use the following properties of these Floer cohomology groups:
\begin{HF}
\item \label{HF property Euler characteristic}
For a graded abstract contact open book
$(M,\theta_M,\phi)$,
the Lefschetz number $\Lambda(\phi)$ of $\phi$
is equal to the Euler characteristic of $HF^*(\phi,+)$
multiplied by $(-1)^n$
(follows from \ref{item:signproperty}).
\item \label{item:HFOBD}
Suppose that $(M_1,\theta_{M_1},\phi_1)$,
$(M_2,\theta_{M_2},\phi_2)$ are graded abstract
contact open books
so that
the graded contact pairs associated to them are graded contactomorphic.
Then $HF^*(\phi_0,+) = HF^*(\phi_1,+)$ (See the Appendix A).
\item  \label{item:spectralsequenceproperty}
Let $(M,\theta_M,\phi)$ be a graded abstract contact open book where $\text{dim}(M) = 2n$.
Suppose that the set of fixed points of
a small positive slope perturbation
$\check{\phi}$ of $\phi$ is a disjoint union of
codimension $0$ families of fixed points $B_1,\cdots,B_l$ and let $\iota : \{1,\cdots,l\} \lra{} \N$
be a function where
\begin{itemize}
\item $\iota(i) = \iota(j)$ if and only if the action of $B_i$ equals the action of $B_j$ and
\item $\iota(i) < \iota(j)$ if the action of $B_i$ is less than the action of $B_j$.
\end{itemize}
Then there is a cohomological spectral sequence converging to $HF^*(\phi,+)$ whose $E_1$ page is equal to
$$E_1^{p,q} = 
\bigoplus_{ \{i \in \{1,\cdots,l\} \ : \ \iota(i) = p\} }
H_{n-(p+q)-CZ(\phi,B_j)}(B_p;\Z)
$$
(see Appendix C).
\end{HF}

\section{Constructing a Well Behaved Contact Open Book.} \label{section constructing a nice contact open book}

The aim of this section is to modify the Milnor monodromy map so that the set of fixed points is a union of codimension $0$ families of fixed points so that we can apply \ref{item:spectralsequenceproperty} above.
We will do this by constructing such a nice symplectomorphism whose mapping torus is isotopic to the mapping torus of the Milnor monodromy map.

\subsection{Some Preliminary Definitions} \label{section some preliminary definitions}

The aim of this section is to construct a symplectic form on the resolution which behaves well with respect to the resolution divisors.
To do this we need a purely symplectic notion of smooth normal crossing divisor.
We will introduce some notation from
\cite{McLeanTehraniZinger:smoothing}.

If $V \subset X$ is a submanifold of a manifold $X$ then we will use the notation
\begin{equation}\label{eqn:normalbundle}
\pi_{\cN_X V} : {\cN}_X V \equiv \frac{TX|_V}{TV} \twoheadrightarrow V
\end{equation}
for the normal bundle of $V$.
If $(V_i)_{i \in S}$ is a finite collection of submanifolds of $X$
and $I \subset S$, let
\[V_I \equiv \bigcap_{i \in I} V_i. \]
Also by convention we define $V_\emptyset \equiv X$.
The collection $(V_i)_{i \in S}$ intersects {\it transversally} if
for every subset $I \subset S$ and every $x \in V_I$,
\[\text{codim}_{T_xX}\left(\bigcap_{i \in I} T_xV_i\right)
= \sum_{i \in I} \text{codim}_{T_xX} (T_x V_i).\]
If $V \subset X$ is a submanifold and $X$ is oriented
then an orientation on $V$ induces an orientation
on ${\cN}_X V$ and conversely
an orientation on ${\cN}_X V$ induces an orientation
on $V$ using Equation
(\ref{eqn:normalbundle})
(if $V$ is odd dimensional this depends on a sign convention
but we will not need this since the manifolds that we will be dealing with are even dimensional).
If $X$ is oriented and
 $(V_i)_{i \in S}$
is a collection of oriented
transversally intersecting submanifolds,
then the submanifold $V_I$ has a natural orientation
called the {\it intersection orientation}
since ${\cN}_X V_I = \bigoplus_{i \in I} {\cN}_X V_i$
is oriented.

\begin{defn}
Suppose that $(X,\omega)$ is a symplectic manifold.
Then $(V_i)_{i \in S}$ is called a
{\it symplectic crossings divisor} or {\it SC divisor}
if $(V_i)_{i \in S}$ are transversally intersecting codimension $2$
symplectic submanifolds
so that $V_I$ is also symplectic
and so that the symplectic orientation on $V_I$
agrees with its corresponding intersecting orientation
for all $I \subset S$.
We will also assume that $S$ is a finite set.
\end{defn}

We now want to define SC divisors with particularly nice neighborhoods.
We call $$\pi : (L,\rho,\nabla) \to V$$ a {\it Hermitian line bundle}
if $\pi : L \to V$ is a complex line bundle,
$\rho$ is a Hermitian metric and $\nabla$ is a $\rho$-compatible
Hermitian connection on $L$.
We define $\rho^{\R}$ and $\rho^{{\it i}\R}$ to be the real and complex parts of the Hermitian metric $\rho$.
We will also use the notation $\rho$ to denote square of the norm function on $L$.
If we view $L$ as an oriented real vector bundle 
then we can recover the complex structure
${\bf i}_\rho$ from the metric $\rho^{\R}$ using the fact that
for all $x \in V$ and $W \in L|_x - 0$,
${\bf i}_\rho(W)$ is the unique vector making $W, {\bf i}_\rho(W)$ into an oriented orthonormal basis of $L|_x$.
Hence we can define a {\it Hermitian structure} $(\rho,\nabla)$
on any oriented $2$-dimensional real vector bundle $L \twoheadrightarrow V$
to be a pair $(\rho,\nabla)$ making $(L,{\bf i}_\rho)$
into a Hermitian line bundle.

For any such triple $(L,\rho,\nabla)$ we have an associated
Hermitian connection $1$-form $\alpha_{\rho,\nabla} \in \Omega^1(L - V)$. This is the pullback of the associated principal $U(1)$-connection on the unit circle bundle of $(L,\rho)$ (see \cite{boothbywangcontact} or \cite[Appendix A]{McLeanTehraniZinger:smoothing}).
If $(L_i,\rho_i, \nabla^{(i)})_{i \in I}$
is a finite collection of Hermitian line bundles over a symplectic manifold
$(V,\omega)$ and $pr_{I;i} : \bigoplus_{i \in I} L_i \to L_i$
is the natural projection map then we define
\begin{equation} \label{equation:regularizationsymplecticform}
\omega_{(\rho_i,\nabla^{(i)})_{i \in I}} \equiv \pi^* \omega + \frac{1}{2}\sum_{i \in I} pr_{I;i}^* d(\rho \alpha_{\rho_i,\nabla^{(i)}}).
\end{equation}

This is a symplectic form in some small neighborhood of the zero section.
Given a $2$-dimensional symplectic vector bundle $L \twoheadrightarrow V$
with symplectic form $\Omega$,
an {\it $\Omega$-compatible Hermitian structure} on $L$ is a Hermitian structure $(\rho,\nabla)$
where
the complex structure ${\bf i}_\rho$ is compatible with $\Omega$.

Suppose that $\Psi : \check{V} \to V$ is a diffeomorphism
and suppose
 $\pi : (L_i,\rho_i, \nabla^{(i)})_{i \in I} \to V$,
$\check{\pi} : (\check{L}_i,\check{\rho}_i, \check{\nabla}^{(i)})_{i \in I} \to \check{V}$ are two
collections of Hermitian line bundles
then a {\it product Hermitian isomorphism}
is a vector bundle isomorphism
$$\widetilde{\Psi} : \bigoplus_{i \in I} \check{L}_i \to \bigoplus_{i \in I} L_i$$
covering $\Psi$
and respecting the direct sum decomposition
so that the induced map
$\widetilde{\Psi} : (\check{L}_i,\check{\rho}_i,\check{\nabla}^{(i)}) \to (L_i,\rho_i,\nabla^{(i)})$
is an isomorphism of Hermitian line bundles for all $i \in \check{I}$.

\begin{defn} \label{defn:regularization}
	Let $V \subset X$ be a submanifold of a manifold $X$.
	A {\it regularization} is a diffeomorphism
	$\Psi : \check{\cN} \to X$ from a neighborhod
	$\check{\cN} \subset {\cN}_X V$ of the zero section
	onto its image so that $\Psi(x) = x$ for all $x \in V$
	and so that the map
	$$d_x \Psi : \cN_X V|_x \lra{} \cN_X V|_x, \quad d_x\Psi(v) \equiv Q\left(D\Psi\left(\left. \frac{d}{dt}(tv) \right|_{t = 0} \right)\right)
	$$
	is the identity map where
	$Q : TX|_V \lra{} \cN_X V$ is the natural quotient map.

\end{defn}

We also need a notion of regularization compatible with
the symplectic form. Because of this we first need to talk
about connections induced by closed $2$-forms.
Recall that an {\it Ehresmann connection} on a smooth submersion 
$\pi : E \lra{} B$ is a distribution $H \subset TE$ so that
$D\pi|_{H_x} : H_x \lra{} T_{\pi(x)}B$ is an isomorphism for all $x \in E$.
\begin{defn}
	Let $\pi : E \lra{} B$
	be a smooth submersion
	and let $\Omega$ be a $2$-form on $E$
	whose restriction to each fiber is non-degenerate.
	Then the {\it symplectic connection associated to $\Omega$}
	is an Ehresmann connection $H \subset TE$
	where $H$ is the set of vectors which are $\Omega$-orthogonal
	to the fibers of $\pi$.
	In other words,
	$$H = \{ V \in T_xE \ : \ x \in E, \quad \Omega(V,W) = 0 \quad \forall W \in T^{\text{ver}}_x E\}$$
	where $T^{\text{ver}}_x E \equiv \ker(D\pi)$ is the {\it vertical tangent bundle}.
\end{defn}

If $(X,\omega)$ is a symplectic manifold and $V \subset X$ is a symplectic submanifold
then ${\cN}_X V$ is a symplectic vector bundle on $V$.
We write $\omega|_{\cN_X V}$ for the symplectic form on this vector bundle
and $\omega|_L$ for the restriction of $\omega|_{\cN_X V}$ to $L$
where $L$ is any subbundle $L \subset \cN_X V$.
The following definition differs in notation from
\cite[Definition 2.8]{McLeanTehraniZinger:normalcrossings}
but it defines the same object.

\begin{defn} \label{defn:omegaregularizationforonedivisor}
	Let $(X,\omega)$ be a symplectic manifold, $V$ a symplectic submanifold and let
\begin{equation} \label{eqn:splittingequation}
	 \cN_X V \equiv \bigoplus_{i \in I} L_i
\end{equation}	
	be a splitting into $2$-dimsional subbundles so that $\omega|_{L_i}$ is non-degenerate for all $i \in I$.
	A tuple $((\rho_i)_{i \in I},\Psi)$ is called an {\it $\omega$-regularization for $V$ in $X$}
	if $\Psi$ is a regularization for $V$ in $X$ and $\rho_i$ is a map
	$\rho_i : \Im(\Psi) \lra{} [0,\infty)$
	so that there is an $\omega|_{L_i}$-compatible Hermitian structure
	$(\widetilde{\rho}_i,\nabla^{(i)})$ on $L_i$
	satisfying 
	\begin{itemize}
		\item 	$\widetilde{\rho}_i|_{\Dom(\Psi)} = \rho_i \circ \Psi$ for all $i \in I$ where $\widetilde{\rho}_i$ is (by abuse of notation) the pullback of $\widetilde{\rho}_i$ by the natural projection map $\oplus_{j \in I} L_j \lra{} L_i$,
		\item 	$\nabla^{(i)}$ restricted to $\Dom(\Psi) \cap L_i$
			coincides with the symplectic connection associated to $\omega|_{L_i}$
			for all $i \in I$	and
	\begin{equation} \label{equation symplectic form} 
		\llap{\textbullet\hspace{127pt}}
		\Psi^* \omega = \omega_{(\widetilde{\rho}_i,\nabla^{(i)})_{i \in I}}|_{\Dom(\Psi)}
			\end{equation}
	\end{itemize}
	for each $i \in I$.

	The splitting (\ref{eqn:splittingequation}) is called the {\it associated splitting}
	and the $\omega|_{L_i}$-compatible Hermitian structure
	$(\widetilde{\rho}_i,\nabla^{(i)})$
	is called the {\it associated Hermitian structure on $L_i$}.
	This Hermitian structure is uniquely determined by $\rho_i$ and $\Psi$.
	Also since $\widetilde{\rho}_i$ gives us a complex structure on
	$L_i$ for each $i \in I$, we get a natural complex structure on $\cN_X V$
	which we will call the {\it complex structure associated to the $\omega$-regularization $((\rho_i)_{i \in I},\Psi)$}.
\end{defn}

We wish to extend Definitions \ref{defn:regularization} and \ref{defn:omegaregularizationforonedivisor} to transverse collections of submanifolds
and SC divisors respectively.
To do this we need some preliminary notation.
Let $X$ be a manifold and $(V_i)_{i \in S}$ transversely intersecting submanifolds.
%
We have a canonical identification
\begin{equation}  \label{eqn:normalbundleidentification}
\cN_X V_I = \pi_{I;I'}^* (\cN_{V_{I'}}V_I)
\end{equation}
for each  $I' \subset I$
where $$ \pi_{I;I'} : \cN_{V_{I-I'}} V_I \to V_I$$
is the natural projection map.
Note that (\ref{eqn:normalbundleidentification}) is not an identification of bundles
since the base of the left hand bundle is $V_I$ whereas the base of the right hand bundle is $\cN_{V_{I-I'}} V_I$.

\begin{defn} \label{defn:systemofregularizations}
Let $X$ be a manifold and $(V_i)_{i \in S}$ a transverse collection of submanifolds. A {\it system of regularizations} is a tuple
$(\Psi_I)_{I \subset S}$, where $\Psi_I$ is a regularization for $V_I$ so that
\begin{equation} \label{equation:domainimagecompatibility}
\Psi_I(\cN_{V_{I'}}|_{V_I} \cap \Dom(\Psi_I)) = V_{I'} \cap \Im(\Psi_I)
\end{equation}
for all $I' \subset I \subset S$.
\end{defn}


Define
$$\iota : \pi_{I;I'}^* (\cN_{V_{I'}} V_I) \hra{} T\pi_{I;I'}^* (\cN_{V_{I'}} V_I) \stackrel{(\ref{eqn:normalbundleidentification})}{=} T\cN_X V_I, \quad \iota(x,v) \equiv \frac{d}{dt}\left.\left(x,tv\right)\right|_{t = 0}.$$
Using the inclusion map $\iota$ above, define
\begin{eqnarray}
\mathfrak{D}\Psi_{I;I'} : \pi_{I;I'}^* (\cN_{V_{I'}}V_I)|_{\Psi_I^{-1}(V_{I'})} \lra{} \cN_X (V_{I'} \cap \Im(\Psi_I)), \quad
\mathfrak{D}\Psi_{I;I'}(w) \equiv Q(D\Psi_I(\iota(w))) \label{defn:infinitessimalnormalbundleidentification}
\end{eqnarray}
where $Q : TX|_{V_{I'} \cap \Im(\Psi_I)} \lra{} \cN_X(V_{I'} \cap \Im(\Psi_I))$ is the natural projection map.

Using the equality (\ref{eqn:normalbundleidentification}), $\mathfrak{D}\Psi_{I;I'}$ identifies the normal bundle of $\cN_{V_{I-I'}}V_I$ inside $\cN_X V_I$ near $0$ with the normal bundle of $V_{I'}$ near $V_I$
using the derivative of the regularization $\Psi_I$.
The map $\mathfrak{D}\Psi_{I;I'}$ is also a bundle isomorphism covering the diffeomorphism $$\Psi_I|_{\Psi_I^{-1}(V_{I'})} : \Psi_I^{-1}(V_{I'}) \lra{} V_{I'} \cap \Im(\Psi_I).$$
The definition below tells us that $\Psi_{I'}$ and $\Psi_I$ should
be equal under the identification (\ref{defn:infinitessimalnormalbundleidentification}).

\begin{defn} \label{defn:scregularization}
Let $X$ be a manifold and $(V_i)_{i \in S}$ a transverse collection of submanifolds of $X$.
Then a {\it regularization for $(V_i)_{i \in S}$} is a system of regularizations $(\Psi_I)_{I \subset S}$
for $(V_i)_{i \in S}$ so that:
\begin{equation} \label{equation:chartcompatibility}
\mathfrak{D}\Psi_{I;I'}(\Dom(\Psi_I)) = \Dom(\Psi_{I'})|_{V_{I'} \cap \Im(\Psi_I)},\quad \Psi_I = \Psi_{I'} \circ \mathfrak{D} \Psi_{I;I'}|_{\Dom(\Psi_I)}.
\end{equation}
\end{defn}

The following definition differs from \cite[Definition 2.11]{McLeanTehraniZinger:normalcrossings} for the same reasons that
Definition \ref{defn:omegaregularizationforonedivisor} differs
from \cite[Definition 2.8]{McLeanTehraniZinger:normalcrossings}.
Apart from that, this definition is exactly the same. This structure also appears in \cite[Lemma 5.14]{McLean:affinegrowth}
although the regularization maps have particular domains
and it is defined in a slightly different way.
\begin{defn} \label{defn:omegaregularization}
Let $(X,\omega)$ be a symplectic manifold and $(V_i)_{i \in S}$ an SC divisor.
An {\it $\omega$-regularization} is a pair of tuples
$$((\rho_i)_{i \in S}, (\Psi_I)_{I \subset S})$$
where
\begin{enumerate}
	\item $(\Psi_I)_{I \subset S}$ is a regularization for $(V_i)_{i \in S}$
		as in Definition \ref{defn:scregularization} and $$\rho_i : \bigcup_{i \in I \subset S} \Im(\Psi_I) \lra{} [0,\infty)$$ is a smooth map, 
	\item $((\rho_i|_{\Im(\Psi_I)})_{i \in I},\Psi_I)$ is an $\omega$-regularization
	for $V_I$ in $X$ for each $I \subset S$ as in Definition
	\ref{defn:omegaregularizationforonedivisor}
	and 
	\item the maps
	$\mathfrak{D}\Psi_{I;I'}$ from (\ref{defn:infinitessimalnormalbundleidentification})
	are product Hermitian isomorphisms for all $I' \subset I \subset S$ with respect to the natural splittings
	$$\pi_{I;I'}^* (\cN_{V_{I'}}V_I)|_{\Psi_I^{-1}(V_{I'})} = \bigoplus_{i \in I'} \pi_{I;I'}^*(\cN_X V_i|_{V_I})|_{\Psi_I^{-1}(V_{I'})},$$
	$$\cN_X(V_{I'} \cap \Im(\Psi_I)) = \bigoplus_{i \in I'} \cN_X V_i|_{V_{I'} \cap \Im(\Psi_I)}.$$

\end{enumerate}
\end{defn}

We are only interested in regularizations near $(V_i)_{i \in S}$
and so we want a notion of equivalence to reflect this.
\begin{defn} \label{defn:germequivalenctregulariztions}
Two $\omega$-regularizations
$$((\rho_i)_{i \in S}, (\Psi_I)_{I \subset S}), \quad 
((\check{\rho}_i)_{i \in S}, (\check{\Psi}_I)_{I \subset S})
$$
for $(V_i)_{i \in S}$
are {\it germ equivalent}
if there is an open set
$$U_I \subset \Dom(\Psi_I) \cap \Dom(\check{\Psi}_I)$$
containing $V_I$
so that
$\Psi_I|_{U_I} = \check{\Psi}_I|_{U_I}$
and
$\rho_i|_{\Psi_I(U_I)} = \check{\rho}_i|_{\Psi_I(U_I)}$ for each $i \in I \subset S$.
\end{defn}

A real codimension $2$ submanifold with an oriented normal bundle
should be thought of as the differential geometric analogue of a smooth divisor in algebraic geometry.
We wish to construct complex line bundles from such submanifolds in the same way that line bundles are constructed from Cartier divisors in algebraic geometry.
The following line bundle associated to a codimension
$2$ submanifold $V$ of a manifold $X$
will depend on a choice of regularization $\Psi : \check{\cN} \to X$ of $V$
and a complex structure ${\bf i}$ on $\cN_X V$
and it will come with a canonical section $s_V : V \to \cO_X(V)$ whose zero set is $V$.
We define:
\begin{equation} \label{eqn:linebundle}
\cO_X(V) = \left( (\pi_{\cN_X V}^* \cN_X(V))|_{\Dom(\Psi)} \sqcup (X - V) \times \C \right) / \sim, 
\end{equation}
\begin{equation}
(\pi_{\cN_X V}^* \cN_X(V))|_{\Dom(\Psi_I)} \ni (v,cv) \sim (\Psi(v),c) \in (X - V) \times \C, \quad \forall \ v \in \cN_X V - V, \ c \in \C \nonumber.
\end{equation}
The corresponding fibration is defined in the following natural way
$$\pi_{\cO_X(V)} : \cO_X(V) \lra{} X,$$
$$\pi_{\cO_X(V)}(v,w) \equiv \Psi(v), \quad \forall (v,w) \in \pi_{\cN_XV}^*(\cN_XV)$$
$$\pi_{\cO_X(V)}(x,c) \equiv x, \quad \forall (x,c) \in (X - V) \times \C.$$

This is a line bundle satisfying the following important canonical identities:
\begin{equation} \label{equation:restriction}
\cO_X(V)|_V = \cN_X(V), \quad \cO_X(V)|_{X - V} = (X - V) \times \C.
\end{equation}
We will call $\Psi$ the {\it associated regularization}.
The {\it canonical section} $s_V : X \to \cO_X(V)$ of this line bundle is defined as follows:
\begin{equation} \label{equation:section}
s_V(x) \equiv 
\left\{
\begin{array}{lll}
(\Psi^{-1}(x),\Psi^{-1}(x)) & \in (\pi_{\cN_X V}^* \cN_X(V)|_{\Dom(\Psi_I)} & \text{if} \ x \in \Im(\Psi_I) \\
(x,1) & \in (X - V) \times \C & \text{if} \ x \in (X - V).
  \end{array}
  \right.
\end{equation}
We also define $\cO_X(0) \equiv \cO_X \equiv X \times \C$ to be the trivial bundle. This is also $\cO_X(\emptyset)$ where $\emptyset$ is the empty submanifold.

\subsection{Trivializing Line Bundles} \label{section trivialization line bundles}

In the previous section,
we constructed a line bundle from
any codimension $2$ submanifold with oriented boundary.
As a result, we can construct line bundles from any SC divisor.
In this subsection we show that if such a line bundle is trivial and if our SC divisor admits a regularization,
then our line bundle admits a trivialization which behaves well with respect to this regularization.
This trivialization will be used later to construct a map from a neighborhood of our SC divisor to $\C$
with nice parallel transport maps away from the singularities.

Let $(X,\omega)$ be a symplectic manifold and $(V_i)_{i \in S}$
an SC divisor on $X$ admitting an $\omega$-regularization
$${\mathcal R} \equiv ((\rho_i)_{i \in S}, (\Psi_I)_{I \subset S})$$ as in Definition \ref{defn:omegaregularization}.
Let $(m_i)_{i \in S}$ be natural numbers indexed by $S$.
For all $i \in S$, let $\cO_X(V_i)$ be the line bundle with associated regularization $\Psi_i$
and complex structure
associated to the $\omega$-regularization $(\rho_i,\Psi_i)$.
Recall that these have natural sections $s_{V_i} : X \to \cO_X(V_i)$ as in Equation (\ref{equation:section}).
Define $\cO_X(\sum_i m_i V_i) \equiv \otimes_{i \in S} \cO_X(V_i)^{\otimes m_i}$
and let 
\begin{equation} \label{eqn:canonicalsection}
s_{(m_i)_{i \in S}} \equiv \otimes_{i \in S} s_{V_i}^{\otimes m_i}
\end{equation}
be the canonical section of this bundle.
Using the identity (\ref{equation:restriction})
 we have
\begin{equation}
\cO_X(V_i)|_{V_I} = 
\cN_X V_i|_{V_I}, \quad \forall \ i \in I \subset S, \nonumber
\end{equation}
and hence we get the following maps:
\begin{eqnarray}
	\Pi_{V_i;I} : \cN_X V_I = \bigoplus_{j \in I} \cN_X V_j|_{V_I} \lra{} \cO_X(V_i)|_{V_I}, \nonumber \\
	(v_j)_{j \in I} \lra{} \left\{\begin{array}{ll}
	v_i & \text{if} \ i \in I \\
	s_{V_i}(x) & \text{if} \ i \notin I \end{array}\right., \quad \forall \ (v_j)_{j \in I} \in \cN_X V_I|_x, \ x \in V_I. \nonumber
\end{eqnarray}
Therefore we get a (not necessarily fiberwise linear) map:
\begin{eqnarray} \label{equation:projection}
	\Pi_{(m_i)_{i \in S},I} : \cN_X V_I  \lra{} \cO_X(\sum_i m_i V_i)|_{V_I}, \nonumber\\
	\Pi_{(m_i)_{i \in S},I}(v) \equiv \bigotimes_{i \in S} \Pi_{V_i;I}(v)^{\otimes m_i}
\quad \forall \ v \in \cN_X V_I.
\end{eqnarray}

One can think of the above map as a section of $\cO_X(\sum_i m_i V_i)|_{V_I}$ along with non-trivial infinitesimal
information in the normal direction of $V_I$.

Below is a definition of a trivialization of $\cO(\sum_i m_i V_i)$ with the property that locally around each point of $V_I$, 
the canonical section $s_{(m_i)_{i \in S}}$ of $\cO_X(\sum_i m_i V_i)$
looks approximately like the map $(z_1,\cdots,z_n) \to \prod_{i \in I} (z_ia(|z_i|)/|z_i|)^{m_i}$ in some local chart $z_1,\cdots,z_n$
where $I \subset \{1,\cdots,n\}$ and $a : \R \lra{} \R$ has the following graph:

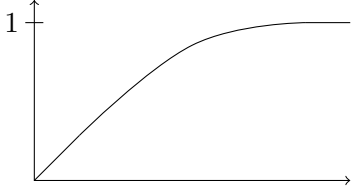
\begin{tikzpicture}[scale = 0.6]
\draw[<->] (6.5,-3) -- (-0.5,-3) node (v1) {} -- (-0.5,1);
\draw (-0.5,-3) -- (0.5,-2) node (v2) {};
\draw  plot[smooth, tension=.7] coordinates {(v2) (3,0) (5.5,0.5)};
\draw (5.5,0.5) -- (6.5,0.5);
\draw (-0.3,0.5) -- (-0.7,0.5);
\node at (-1,0.5) {$1$};
\end{tikzpicture}

We need to use the above function $a$ to ensure that we have good dynamical properties
(see Section \ref{section good dynamical properties}).
See \cite[Definition 3.8]{McLeanTehraniZinger:smoothing} for a related definition.

\begin{defn} \label{defn:compatibletrivialization}
For each $r > 0$,
we define the {\it radius $r$ tube of $V_I$} to be the set
\begin{equation} \label{eqn:radiusrtube}
T_{r,I} \equiv \cap_{i \in I} \{x \in \Im(\Psi_I) \ : \ \rho_i(x) \leq r\}
\end{equation}
over $V_I$.
Let $B \subset X$ be any set.
The {\it tube radius of ${\mathcal R}$ along $B$}
is the largest radius $r$ tube of $V_I$ along $B$ that can `fit'
inside the image of $\Psi_I$ for each $I \subset S$.
More precisely, it is the supremum
of all $r \geq 0$ with the property that $T_{r,I} \cap (\Im(\Psi_I)|_x)$
is a compact subset of $X$ for all $x \in B \cap V_I$ and $I \subset S$.

Let $U \subset X$ be an open set.
Now suppose that the tube radius of ${\mathcal R}$ along $U$
is positive and let $R>0$ be any constant smaller than the tube radius and also smaller than $1$.
We let
$a_R : [0,\infty) \lra{} [0,\infty)$
be a smooth function satisfying:
\begin{enumerate}
	\item $a'_R(x) > 0$ for $x \in [0,3R/4)$,
	\item $a_R(x) = x$ for $x \leq R/4$,
	\item $a_R(x) = 1$ for $x \geq 3R/4$.
\end{enumerate}


A bundle trivialization:
$$\Phi \equiv (\pi,\Phi_2) : \cO_X(\sum_i m_i V_i) \lra{} X \times \C$$
is {\it radius $R$ compatible with ${\mathcal R}$ along $U \cap V_I$}
if
\begin{eqnarray}
\Phi_2(s_{(m_i)_{i \in S}}(x))
=
\Phi_2(\Pi_{(m_i)_{i \in S},I}(\Psi_I^{-1}(x)))
\prod_{i \in I} \left(\frac{\sqrt{a_R(\rho_i(x))}}{\sqrt{\rho_i(x)}}\right)^{m_i}, \label{eqn:phiequalsdefinition}
\\
\forall x \in T_{R,I} \cap (\Im(\Psi_I)|_{V_I \cap U - \cup_{i \in S-I} T_{\frac{3R}{4},i}}). \nonumber
\end{eqnarray}
and where the norm of the linear map:
\begin{equation} \label{eqn:normoflinearmap}
\Phi_2|_{\otimes_{i \in I} (\cN_X V_i|_x)^{\otimes m_i}} : (\otimes_{i \in I} \cN_X V_i|_x)^{\otimes m_i} \lra{} \C
\end{equation}
is equal to $1$
for all $x \in V_I \cap U - \cup_{i \in S-I} T_{\frac{3R}{4},i}$
using the identification (\ref{equation:restriction}).

We say that $\Phi$ is {\it radius $R$ compatible with ${\mathcal R}$ along $U$}
if it is radius $R$ compatible with ${\mathcal R}$ along $U \cap V_I$
for each $I \subset S$.
It is {\it compatible with ${\mathcal R}$ along $U$}
if it is radius $R$ compatible with ${\mathcal R}$ along $U$
for some $R$ smaller than the tube radius of ${\mathcal R}$.

\end{defn}

One should think of Equation (\ref{eqn:phiequalsdefinition})
as saying that the trivialization $\Phi$ identifies the canonical section of $\cO_X(\sum_i m_i V_i)$
with the `infinitesimal' section $\Pi_{(m_i)_{i \in S},I}$ multiplied
by a particular factor near $V_I$.
This particular factor that we are multiplying by is actually equal to $1$
if we are very near $V_I$.
Note also that Equation (\ref{eqn:phiequalsdefinition})
tells us that the norm of
$\Phi_2(s_{(m_i)_{i \in S}}(x))$ only depends on
$(\rho_i(x))_{i \in I}$ and
if $\rho_i(x) \geq 3R/4$ for some $i \in I$ then this norm does not depend on
$\rho_i(x)$ for all $x \in \Im(\Psi_I) - \cup_{i \in S - I} T_{3R/4,i}$.
As a result, the above definition is consistent with the stated norm property
of the linear map (\ref{eqn:normoflinearmap}).

\begin{lemma} \label{lemma:nicetrivialization}
Suppose $\cO_X (\sum_i m_i V_i)$ admits a trivialization $\Phi$
and let $U \subset X$ be a relatively compact open set.
Then there is a trivialization of $ \cO_X(\sum_i m_i V_i)$
compatible with ${\mathcal R}$ along $U$ 
which is homotopic to $\Phi$ through trivializations of $\cO_X (\sum_i m_i V_i)$.
\end{lemma}
\proof
Just as in the proof of \cite[Definition 3.9]{McLeanTehraniZinger:smoothing} we will proceed by induction on the strata of $\cup_j V_j$.
Before we do this though we will need to construct certain natural maps that identify
$\Pi_{(m_i)_{i \in S};I}$ with $s_{(m_i)_{i \in S}}$ near $V_I$.
Define $$N_{i,I} \equiv \Dom(\Psi_i)|_{V_i \cap \Im(\Psi_{I \cup i})} \stackrel{(\ref{equation:chartcompatibility})}{=} {\mathfrak D} \Psi_{I \cup i;i}(\Dom(\Psi_{I \cup i}))$$
$$
W_{i;I} \equiv \Psi_I(\Dom(\Psi_I)|_{V_I - V_i})
$$
$$
D_{i;I} \equiv W_{i;I}  \cup \Psi_i(N_{i;I})
$$
 for all $i \in S, \ I \subset S$.
By Equation (\ref{eqn:linebundle}), we have the natural identification:
\begin{equation}  \label{eqn:neardivisoridentity1}
	\cO_X(V_i)|_{\Psi_i(N_{i;I})} =
	\pi_{\cN_X V_i}^* \cN_X(V_i)|_{N_{i;I}}
\end{equation}
and by (\ref{equation:domainimagecompatibility}) and (\ref{equation:restriction}) we have the identity
\begin{equation} \label{eqn:trivializationidentity}
	\cO_X(V_i)|_{W_{i;I}} =
	W_{i:I} \times \C
\end{equation}
for all $i \in S$ and $I \subset S$.
By using the natural projections
$$Pr_{I;I'} : \cN_X V_I \lra{} \cN_X V_{I'}|_{V_I}, \quad \forall \ I' \subset I \subset S$$
we have a map:
\begin{equation}
	\widehat{\Pi}_{V_i;I} : \cO_X(V_i)|_{D_{I;i}}
	\lra{} \cO_X(V_i)|_{V_I} \nonumber
\end{equation}
for all $i \in S$ and $I \subset S$
whose restriction to $\Psi_i(N_{i;I})$ is defined by the equation
\begin{eqnarray}
	\widehat{\Pi}_{V_i;I}(v,w) \equiv (Pr_{I \cup i;i}(\mathfrak{D}\Psi_{I \cup i;i}^{-1}(v)),Pr_{I \cup i;i}(\mathfrak{D}\Psi_{I \cup i;i}^{-1}(w))) \nonumber \\
	\quad \forall (v,w) \in 
	(\pi_{\cN_X V_i}^* \cN_X(V_i))|_{N_{i;I}} \nonumber
\end{eqnarray}
by using the identity (\ref{eqn:neardivisoridentity1})
and whose restriction to $W_{i;I}$ is defined by
\begin{equation}
\widehat{\Pi}_{V_i;I}(x,c) \equiv (\pi_{\cN_X V_I}(\Psi_I^{-1}(x)),c) \in (V_I- V_i) \times \C, \quad \forall \ (x,c) \in 
W_{i:I} \times \C \nonumber
\end{equation}
by using the identity (\ref{eqn:trivializationidentity}).
One should think of this map as a way of canonically
identifying the bundle $\cO_X(V_i)$ near $V_I$
with its pullback along the natural projection map
from $D_{i;I}$ to $V_I$ induced by $\pi_{\cN_X V_I} \circ \Psi_I^{-1}$.
By (\ref{eqn:linebundle}) and (\ref{equation:section}),
\begin{equation} \label{eqn:liftofinfinitessimalsection}
	\Pi_{V_i;I}|_{\Psi_I^{-1}(D_{i;I})} = \widehat{\Pi}_{V_i;I} \circ s_{V_i} \circ \Psi_I \quad \forall \ i \in S, \ I \subset S.
\end{equation}
Also by (\ref{equation:chartcompatibility})
and the fact that ${\mathfrak D}\Psi_{I;I'}$
is a product Hermitian isomorphism for all $I' \subset I \subset S$,
\begin{equation} \label{eqn:projectioncomposition} 
	\widehat{\Pi}_{V_i;I} \circ \widehat{\Pi}_{V_i;I'}|_{D_{I;i}} = \widehat{\Pi}_{V_i;I} \quad \forall \ i \in S, \ I' \subset I \subset S.
\end{equation}

We can define similar maps for the line bundle $\cO_X(\sum_i m_i V_i)$ in the following way
	\begin{eqnarray}
	\widehat{\Pi}_{(m_i)_{i \in S};I} : \cO_X(\sum_i m_i V_i)|_{\cap_{i \in S} D_{i;I}}
	\lra{} \cO_X(\sum_i m_i V_i)|_{V_I} \nonumber \\
	\widehat{\Pi}_{(m_i)_{i \in S};I}(\otimes_{i \in S, j \in \{1,\cdots,m_i\}} v_{i,j}) =
	\otimes_{i \in S, j \in \{1,\cdots,m_i\}} \widehat{\Pi}_{V_i;I}(v_{i,j}) \quad \forall \ I \subset S. \nonumber
\end{eqnarray}
Equations (\ref{eqn:liftofinfinitessimalsection})
and
(\ref{eqn:projectioncomposition})
give us the following equations:
\begin{eqnarray}
	\Pi_{(m_i)_{i \in S};I}|_{\Psi_I^{-1}(\cap_{i \in S} D_{i;I})} = \widehat{\Pi}_{(m_i)_{i \in S};I} \circ s_{(m_i)_{i \in S}}  \circ \Psi_I|_{\cap_{i \in S} D_{i;I}} \quad \forall \ I \subset S
	\label{eqn:liftofinfinitessimalsectionmi} \\
	\widehat{\Pi}_{(m_i)_{i \in S};I} \circ \widehat{\Pi}_{(m_i)_{i \in S};I'}|_{\cap_{i \in S} D_{i;I}} = \widehat{\Pi}_{(m_i)_{i \in S};I} \quad \forall \ I' \subset I \subset S. \label{eqn:projectioncompositionmi} 
\end{eqnarray}

Using these equations, we will now prove our lemma by induction on the set of subsets of $S$.
Let $\preceq$ be a total order on the set of subsets of $S$ with the property that if $|I'| < |I|$ then $I \preceq I'$.
We write $I \prec I'$ when $I \preceq I'$ and $I \neq I'$.
Since $U$ is relatively compact, the tube radius of ${\mathcal R}$ along $U$ is positive and
hence we can choose any constant $R>0$ smaller than this tube radius.

Suppose that there is some $I^* \subset S$ and a trivialization
$$\Phi^{\prec} \equiv (\pi,\Phi_2^\prec) : \cO_X(\sum_i m_i V_i) \lra{} X \times \C$$
which is radius $R$ compatible with ${\mathcal R}$ along $U \cap V_I$
for all $I \prec I^*$ and
which is isotopic to $\Phi$ through trivializations of $\cO_X(\sum_i m_i V_i)$.
%
%
We now wish to modify the trivialization $\Phi^{\prec}$ so that these properties hold for all $I \preceq I^*$.
Let $T_{r,I}$ be the radius $r$ tube of $V_I$ as in Equation (\ref{eqn:radiusrtube}) and
define $L_r^\prec \equiv \cup_{I \prec I^*} T_{r,I}$.


First of all, let
$$g : V_{I^*} - L_{\frac{3R}{4}}^\prec \lra{} (0,\infty)$$
be a smooth function whose value at $x \in V_{I^*} - L_{\frac{3R}{4}}^\prec$ is equal to the norm of the linear map
\begin{equation} \label{eqn:normalongistar}
\Phi^\prec_2|_{\otimes_{i \in I^*} (\cN_X V_i|_x)^{\otimes m_i}} : (\otimes_{i \in I^*} \cN_X V_i|_x)^{\otimes m_i} \lra{} \C. \nonumber
\end{equation}
Since the map $\Pi_{V_i;I}$ restricted to each fiber of $Pr_{I;I-i}$ is an isometry for all $i \in I \subset S$,
we get that $g(x) = 1$ for all $x \in
V_{I^*} \cap (L_R^\prec - L_{\frac{3R}{4}}^\prec)$ by our induction hypothesis.
Combining this with the fact that $a_R(s) = 1$ for all $s \geq 3R/4$,
we can choose a smooth function 
$f : X \lra{} (0,\infty)$
which is equal to $1$ in the region
$L_R^\prec$
and equal to $(\pi_{\cN_X V_{I^*}} \circ \Psi_{I^*}^{-1})^*\left(\frac{1}{g}\right)$ inside
$T_{R,I^*} \cap \Psi_{I^*}(\Dom(\Psi_{I^*})|_{V_{I^*} - L_{\frac{3R}{4}}^\prec}).$
This implies that the norm of the linear map:
\begin{equation} \label{eqn:normalongistar}
f(x) \Phi^\prec_2|_{\otimes_{i \in I} (\cN_X V_i|_x)^{\otimes m_i}} : (\otimes_{i \in I} \cN_X V_i|_x)^{\otimes m_i} \lra{} \C
\end{equation}
is $1$ for all $I \preceq I^*$.
Hence Equation (\ref{eqn:normoflinearmap}) holds for $f\Phi_2^\prec$ for all $I \preceq I^*$.

We now wish to modify $f\Phi^\prec$
so that
Equation (\ref{eqn:phiequalsdefinition}) for this new trivialization holds as well.
We let 
\begin{eqnarray}
\Phi^= \equiv (\pi,\Phi_2^=) : \cO_X(\sum_i m_i V_i)|_{\Im(\Psi_{I^*})} \lra{} X \times \C,  \nonumber\\
\Phi^=(v) \equiv f(x) \Phi^\prec(\widehat{\Pi}_{(m_i)_{i \in S};I^*}(v))
\prod_{i \in I^*} \left(\frac{\sqrt{a_R(\rho_i(x))}}{\sqrt{\rho_i(x)}}\right)^{m_i}, \label{eqn:phiequaldefinition}\\
\forall \ v \in \cO_X(\sum_i m_i V_i)|_x, \quad x \in \Im(\Psi_{I^*}). \nonumber
\end{eqnarray}
be a smooth trivialization.
Equation (\ref{eqn:projectioncompositionmi})
combined with the fact that $\Phi_2^\prec$
is radius $R$ compatible with ${\mathcal R}$
along $U \cap V_I$ for all $I \prec I^*$
implies that
$\Phi_2^\prec$ is radius $R$ compatible with ${\mathcal R}$
along $U \cap (L_R^\prec - L_{\frac{3R}{4}}^\prec) \cap V_{I^*}$.
Hence by Equation (\ref{eqn:phiequaldefinition}),
we have
\begin{equation}
\Phi_2^{\prec}(v) = \Phi_2^=(v), \quad \forall \ v \in T_{R,I^*}
\cap (L_R^\prec - L_{\frac{3R}{4}}^\prec).
\nonumber
\end{equation}
Combining this with the fact that
$T_{R,I^*}$
deformation retracts on to
$V_{I^*} \cup (T_{R,I^*} \cap L_R^\prec)$,
we can construct a smooth trivialization
$$\Phi^{\preceq} \equiv (\pi,\Phi_2^\prec) : \cO_X(\sum_i m_i V_i) \lra{} X \times \C$$
homotopic to $\Phi^{\prec}$ so that
\begin{equation} \label{eqn:phipreceqequation}
\Phi^{\preceq}|_{L_R^\prec} = \Phi^{\prec}|_{L_R^\prec} \ \text{and} \ \Phi^{\preceq}|_{T_{R,I^*}} = \Phi^=|_{T_{R,I^*}}.
\end{equation}
Equations (\ref{eqn:liftofinfinitessimalsectionmi}),
(\ref{eqn:phiequaldefinition}) and
(\ref{eqn:phipreceqequation})
tell us that
$\Phi^{\preceq}$ is radius $R$ compatible with ${\mathcal R}$
along $U \cap V_{I^*}$.
Also since $\Phi^{\preceq} = \Phi^{\prec}$
along $L_R^\prec$, we get that
$\Phi^{\preceq}$ is radius $R$ compatible with ${\mathcal R}$
along $U \cap V_I$ for all $I \prec I^*$.
Hence $\Phi^{\preceq}$ is radius $R$ compatible with ${\mathcal R}$
along $U \cap V_I$ for all $I \preceq I^*$.

Because the norm of the linear map
(\ref{eqn:normalongistar}) is $1$ for all $I \preceq I^*$, we have
by
Equations
(\ref{eqn:phiequaldefinition}) and
(\ref{eqn:phipreceqequation})
that the norm of the linear map
$$f \Phi^\preceq_2|_{\otimes_{i \in I} (\cN_X V_i|_x)^{\otimes m_i}} : (\otimes_{i \in I} \cN_X V_i|_x)^{\otimes m_i} \lra{} \C$$
is equal to $1$ for all $I \preceq I^*$.
Hence we are done by induction.

\qed

\subsection{Links of Divisors and Open Books}
\label{section links of divisors and open books}

In this subsection, we first give a purely symplectic definition of a divisor which looks like the resolution divisors of a log resolution of an isolated hypersurface singularity.
We then construct the `link' of this resolution divisor, which corresponds to the embedded link of our isolated singularity.
Finally we construct an open book decomposition of this resolution divisor corresponding the Milnor open book of our hypersurface singularity.

We have the following definition from \cite{McLean:affinegrowth}.
\begin{defn} \label{defn:wrappingnumber}
Let $(X,\omega)$ be a symplectic manifold of dimension $2n$ and let
$\theta \in \Omega^1(X - K)$ satisfy $d\theta = \omega|_{X - K}$
for some compact $K \subset X$.
Suppose that $K$ admits an open neighborhood $U$
which deformation retracts onto $K$.
Let $\rho : X \lra{} [0,1]$ be a smooth function equal to $0$ along $K$ and equal to $1$ outside a compact subset of $U$.
Then the {\it dual} 
$$c(\omega,\theta) \in H_{2n-2}(K;\R) = H_{2n-2}(U;\R)$$
of $(\omega,\theta)$ is defined to be the Lefschetz dual
of $(\omega - d(\rho\theta))|_{U} \in H^2_c(U)$.

Now suppose that $K = \cup_{i\in S} V_i$ where $(V_i)_{i \in S}$
is an SC divisor and each $V_i$ is connected and compact.
Then a Mayor-Vietoris argument tells us that
$H_{2n-2}(\cup_{i\in S} V_i;\R)$ is freely generated by
the fundamental classes $[V_i]$ and hence there are unique real
numbers $(w_i)_{i \in S}$ so that
$c(\omega,\theta) = -\sum_i w_i [V_i]$.
The {\it wrapping number} of $\theta$ around $V_j$ is defined to be
$w(\theta,V_j) \equiv w_j.$
\end{defn}
The wrapping number does not depend on the choice of neighborhood $U$
or bump function $\rho$.
We can calculate the wrapping number $w(\theta,V_j)$
in the following way
(See the discussion after Definition 5.4 in \cite{McLean:isolated}):
Let $D \subset X$ be a small symplectically embedded
disk with polar coordinates
$(r,\vartheta)$
which intersects $V_i$ positively once at $0 \in D$ and does not intersect
$V_j$ for all $j \in S - i$.
Then the wrapping number $w_i$ is the unique number so that $(w_i/2\pi )d\vartheta $ is cohomologous to
$$\theta|_{D - 0} - \frac{1}{2} r^2 d\vartheta$$
inside   $H^1(D - 0;\R)$.
The computation of the wrapping number using the disk $D$ above enables us to define wrapping numbers in the case when each $V_i$ is properly embedded but not necessarily compact.
We will need this broader definition of wrapping number in the proof of Lemma \ref{lemma:locallink} below.

The next definition is supposed to be a way of describing a log resolution
of a pair $(\C^{n+1},f^{-1}(0))$ (as in Definition \ref{defn:logresolution}) in a symplectic way. The key motivating
example for such a definition is given in Example \ref{defn:modelresolutionmotivatingexample} below.

\begin{defn}
A {\it resolution divisor} is a pair $(X,(V_i)_{i \in S})$
where $(V_i)_{i \in S}$ are transversally intersecting properly embedded
codimension $2$ submanifolds of a manifold $X$ so that
there is a unique element $\star_S \in S$
with the property
that $V_{\star_S}$ is non-compact and
$V_i$ is compact for all $i \in S- \star_S$.
We also require that $\cup_i V_i$ is connected and that $V_i$ is connected for each $i \in S - \star_S$ (although $V_{\star_S}$ is allowed to be disconnected).

A {\it model resolution} is a triple $(\cO_X(\sum_{i \in S} m_i V_i), \Phi,\theta)$
where $X$ is a manifold, $(V_i)_{i \in S}$ are codimension $2$
submanifolds, $\theta \in \Omega^1(X - \cup_{i \in S - \star_S} V_i)$
and $\Phi$ is a trivialization of
$$\cO_X({\textstyle \sum}_{i \in S} m_i V_i) \equiv \otimes_{i \in S} \cO_X(V_i)^{\otimes m_i}$$ where
$(m_i)_{i \in S}$ are positive integers satisfying:
\begin{enumerate}
	\item $(X,(V_i)_{i \in S})$ is a resolution divisor as above,
	\item  $d\theta = \omega|_{X - \cup_{i \in S - \star_S}V_i}$ for some symplectic form $\omega$ on $X$,
	\item $(V_i)_{i \in S}$ is an SC divisor with respect to $\omega$,
	\item $m_{\star_S} = 1$
	and the wrapping number $w(\theta,V_i)$ is positive for all $i \in S - \star_S$.
	
	\bigskip 
	The form $\omega$ is called the
	{\it symplectic form associated to 
		$(\cO_X(\sum_{i \in S} m_i V_i),
		\Phi,\theta)$.}
	A {\it grading} on this model resolution is a grading
	on $(X - \cup_{i \in S - \star_S} V_i,\omega)$.
\end{enumerate}

\end{defn}


\begin{defn}
Two model resolutions $$Y \equiv (\cO_X(\sum_{i \in S} m_i V_i),
\Phi,\theta), \quad
\widehat{Y} \equiv (\cO_{\widehat{X}}(\sum_{i \in \widehat{S}} \widehat{m}_i \widehat{V}_i),
\widehat{\Phi},\widehat{\theta})$$
are {\it isotopic} if there is 
\begin{itemize}
	\item a bundle isomorphism
	$\widetilde{\Psi} : \cO_X(\sum_{i \in S} m_i V_i) \lra{}
	\cO_{\widehat{X}}(\sum_{i \in \widehat{S}} \widehat{m}_i \widehat{V}_i)$ covering a diffeomorphism
	$\Psi : X \lra{} \widehat{X}$,
	\item a bijection $\iota : S \lra{} \widehat{S}$
	sending $\star_S$ to $\star_{\widehat{S}}$,
	\item a smooth family of $1$-forms $(\theta_t \in \Omega^1(X - \cup_{i \in S - \star_S} V_i))_{t \in [0,1]}$ joining $\theta$ and $\Psi^* \widehat{\theta}$ and
 trivializations $(\Phi_t)_{t \in [0,1]}$ of
	$\cO_X(\sum_{i \in S} m_i V_i)$ joining $\Phi$
	and $\widehat{\Phi} \circ \widetilde{\Psi} \circ (\Psi^{-1} \times \text{id}_{\C})$
\end{itemize}
 so that $m_i = \widehat{m}_{\iota(i)}$ and $\Psi(V_i) = \widehat{V}_{\iota(i)}$ for all $i \in S$ and
 $Y_t \equiv (\cO_X(\sum_{i \in S} m_i V_i),
 \Phi_t,\theta_t)$
 is a model resolution for all $t \in [0,1]$.
 
 These model resolutions are {\it graded isotopic}
 if they are isotopic as above with the additional property
 that the model resolutions
 $Y_t$
 all admit gradings that smoothly depend on $t \in [0,1]$
 and where the grading on
 $Y_0$
coincides with the grading on 
$Y$
and the grading on
 $Y_1$
coincides with the grading on 
$\widehat{Y}$
under the identification $\Psi$.
\end{defn}

We will now give an example of a model resolution.
\begin{example} \label{defn:modelresolutionmotivatingexample}
Let $f : \C^{n+1} \lra{} \C$ be a polynomial with an isolated singularity
at $0$ and let
$U \subset \C^{n+1}$ be an open set containing $0$
so that $f|_{U - 0}$ is regular.
Let $\pi : Y \lra{} \C^{n+1}$ be a log resolution of $(\C^{n+1},f^{-1}(0))$
obtained by a sequence of blowups along smooth loci
and define $X \equiv \pi^{-1}(U)$ (such a resolution exists by \cite{hironaka:resolution}).
Let $(E_j)_{j \in S}$ be the resolution divisors of
this resolution and let $E_{\star_S} \subset X$
be the proper transform of $f^{-1}(0) \cap U$.
Because such a resolution is obtained by a sequence of blowups along smooth loci,
there are positive integers $(w_j)_{j \in S - \star_S}$
so that $A \equiv -\sum_{i \in S - \star_S} w_i E_i$ is ample on $X$.
By the divisor line bundle correspondence,
let $L \lra{} X$ be the corresponding ample line bundle
with a meromorphic section satisfying
$(s) = A$.
Choose a Hermitian metric $\|\cdot\|$ on $L$
whose curvature form is a positive $1$-$1$ form.
Define $\theta = -d^c \ln(\|s\|)$.
The $2$-form $d\theta$ extends uniquely to a K\"{a}hler form $\omega$
on $X$.
The wrapping number of $\theta$ around $E_j$ is $w_j$ for all $j \in S - \star_S$.

We also let $m_i \in \N_{>0}$ be the multiplicity of $f$ along $E_i$
for all $i \in S$.
Since $\sum_{i \in S} m_i E_i$ is the divisor defined by $f \circ \pi|_X$
we get, by the divisor line bundle correspondence, that the holomorphic line bundle
$\cO_X(\sum_{i \in S} m_i E_i)$ has an induced
trivialization $\Phi : \cO_X(\sum_{i \in S} m_i E_i) \cong X \times \C$
so that the section $s$
corresponding to the holomorphic function
$f \circ \pi|_X$ satisfies $pr_2 \circ \Phi \circ s = f \circ \pi|_X$ where $pr_2 : X \times \C \lra{} \C$ is the natural projection map.
Then $(\cO_X(\sum_{i \in S} m_i E_i),
\Phi,\theta)$
is a model resolution
called a {\it model resolution associated to $f$}.

Such a resolution also has a grading
as follows:
Let $\check{X} \equiv X - \cup_{i \in S-\star_S} E_i$.
Since
$$\pi|_{\check{X}} : \check{X} \lra{} U - 0 \subset \C^{n+1}$$
is a biholomorphism we get a canonical holomorphic trivialization
$$\Phi : T\check{X} \lra{} \check{X} \times \C^{n+1}$$
as a unitary vector bundle coming from the trivialization on $T\C^{n+1}$.
The grading on $T\check{X}$ is equal to the trivial grading on $\check{X} \times \C^{n+1}$ pulled back by $\Phi$.
We will call this the {\it standard grading}.
\end{example}

We now wish to associate a pair of contact manifolds with normal bundle data to a model resolution. In the case of Example \ref{defn:modelresolutionmotivatingexample}, this is the
link of our singularity (as defined in the introduction) with some additional data.

\begin{defn} \label{defn:comptiblewithvstars}
Let $(X,(V_i)_{i \in S})$ be a resolution divisor.
A tuple of regularizations $$(\Psi_i)_{i \in S - \star_S}$$
is {\it compatible with $V_{\star_S}$}
if $\Psi_i$ is a regularization of $V_i$ for each $i \in S - \star_S$
and $$\Psi_i(\Dom(\Psi_i)|_{V_{\star_S} \cap V_i}) \subset V_{\star_S} \  \text{for all} \ i \in S - \star_S.$$
In other words,  $\Psi_i$ restricted to
$\Dom(\Psi_i)|_{V_{\star_S} \cap V_i}$
is a regularization of $V_i \cap V_{\star_S}$ inside $V_{\star_S}$ for all $i \in S - \star_S$.
\end{defn}

\begin{defn}
	Let $\Psi,\check{\Psi}$ be regularizations of a submanifold $V$
	of a manifold $X$.
	Then a smooth family of regularizations $\Psi_t$ of $V$
	{\it connects $\Psi$ and $\check{\Psi}$} if
	$\Psi$ is germ equivalent to $\Psi_0$
	and $\check{\Psi}$ is germ equivalent to $\Psi_1$.
\end{defn}

\begin{lemma} \label{lemma:connectedregularizations}
Let $(X,(V_i)_{i \in S})$ be a resolution divisor.
For any two tuples of regularizations
$(\Psi_i)_{i \in S - \star_S}$, $(\check{\Psi}_i)_{i \in S - \star_S}$
 compatible with $V_{\star_S}$,
there is a smooth family
of such regularizations
$$(\Psi^t_i)_{i \in S - \star_S}, \quad t \in [0,1]$$
compatible with $V_{\star_S}$
which connects
$(\Psi_i)_{i \in S - \star_S}$
and $(\check{\Psi}_i)_{i \in S - \star_S}$.
\end{lemma}
\proof
Choose a metric making $V_{\star_S}$
into a totally geodesic submanifold.
Define $T^r X \subset TX$ to be the set of tangent vectors of length at most $r$.
Fix $i \in S$.
Choose a relatively compact neighborhood $W_i$ of $V_i$ in $X$
and let $\delta>0$ small enough so that the exponential
map restricted to $T^\delta_w X$ is a diffeomorphism onto its image for all $w$ in $W_i$.
Let $\widetilde{W}_i \subset \Psi_i^{-1}(W_i) \cap (\check{\Psi}_i)^{-1}(W_i)$
be a small enough neighborhood of $V_i$
so that the distance $d(v)$ between
$\Psi_i(v)$ and $\check{\Psi}_i(v)$ is less than $\delta$
for all $v \in \widetilde{W}_i$.
Now let $\gamma_v : [0,d(v)] \lra{} X$
be the unique geodesic of length $<\delta$ joining
$\Psi_i(v)$ and $\check{\Psi}_i(v)$
for all $v \in \widetilde{W}_i$.
Define
$$\widetilde{\Psi}^t_i : \widetilde{W}_i \lra{} X, \quad \widetilde{\Psi}^t_i(v) \equiv \gamma_v(td(v)) \quad \forall \ t \in [0,1].$$
Since $d_v \Psi$ and $d_v \check{\Psi}$
from Definition \ref{defn:regularization}
are the both identity map,
we get that $d_v \widetilde{\Psi}^t_i$ is also the identity map
for all $v \in V_i$.
Hence there is a neighborhood $\widetilde{W}'_i \subset \widetilde{W}_i$
of $V_i$
so that
$\Psi^t_i \equiv \widetilde{\Psi}^t_i|_{\widetilde{W}'}$
is a diffeomorphism onto its image for all $t \in [0,1]$.
Hence $(\Psi^t_i)_{t \in [0,1]}$ is a smooth family of regularizations
compatible with $V_{\star_S}$
which connects $\Psi_i$ and $\check{\Psi}_i$.
Therefore $(\Psi^t_i)_{i \in S - \star_S}, \ t \in [0,1]$ connects
$(\Psi_i)_{i \in S - \star_S}$
and $(\check{\Psi}_i)_{i \in S - \star_S}$.
\qed

\begin{defn} \label{defninition function compatible with}
Let $X$ be a smooth manifold and let
$(V_i)_{i \in \check{S}}$ be transversely intersecting compact codimension $2$ submanifolds of $X$.
A smooth function $$f : X  - \cup_{i \in \check{S}} V_i \lra{} \R$$
is {\it compatible with} $(V_i)_{i \in \check{S}}$ if there is
\begin{itemize}
	\item a regularization $\Psi_i$ of $V_i$,
	\item  a real number $b_i > 0$ and
	\item a
	smooth function $q_i : \Dom(\Psi_i) \lra{} [0,1]$
	equal to the square of some norm on $\cN_X V_i$ near $V_i$, equal to $1$ outside a compact subset of
	$\Dom(\Psi_i)$
	and non-zero on $\Dom(\Psi_i) - V_i$
\end{itemize}
for each $i \in S$ and a smooth function $\tau : X \lra{} \R$
so that
$$f = \sum_{i \in \check{S}} b_i  \log(q_i \circ \Psi_i^{-1}) + \tau$$
where $\log(q_i \circ \Psi_i^{-1})$
is defined to be $0$ outside $\Im(\Psi_i)$ for each $i \in S$. We will call the regularizations $(\Psi_i)_{i \in S-\star_S}$
{\it associated regularizations of $f$}.

Now suppose that we have an additional smooth submanifold $V_{\star_S}$ of $X$ so that
$(V_i)_{i \in S}$ becomes a resolution divisor where $S = \check{S} \sqcup \{\star_S\}$.
We say that $f$ is {\it compatible with $(V_i)_{i \in S}$} if it is compatible with $(V_i)_{i \in S -\star_S}$ as above with the additional property that
the associated regularizations of $f$ are compatible with $V_{\star_S}$.
As a consequence of this we have that $f|_{V_{\star_S}}$ is compatible with $(V_i \cap V_{\star_S})_{i \in S - \star_S}$.

We say that $f$ is {\it strongly compatible}
with $(V_i)_{i \in S}$ if in addition $\tau = 0$.

\end{defn}

\begin{lemma} \label{lemma:compatiblefunctionsareconnected}
	Let $(X,(V_i)_{i \in S})$ be a resolution divisor
	and let
	$$f, g : X  - \cup_{i \in S - \star_S} V_i \lra{} \R$$
	be a pair of smooth functions compatible with $(V_i)_{i \in S}$.
	Then there is a smooth family of functions
	$$f_t : X  - \cup_{i \in S - \star_S} V_i \lra{} \R, \quad t \in [0,1]$$
	compatible with  $(V_i)_{i \in S}$
	so that $f_0 = f$ and $f_1 = g$.
\end{lemma}
\proof
For all  $i \in S - \star_S$, there are regularizations
$\Psi_i,\check{\Psi}_i$ of $V_i$,
 smooth functions
$$q_i : \Dom(\Psi_i) \lra{} [0,1], \quad \check{q}_i : \Dom(\check{\Psi}_i) \lra{} [0,1]$$
equal to the square of a norm near $V_i$,
equal to $1$ outside a compact set and
non-zero outside $V_i$,
real numbers $b_i,\check{b}_i > 0$ and smooth functions
$\tau,\check{\tau} : X \lra{} \R$ so that
$$f = \sum_{i \in S -\star_S} b_i  \log(q_i \circ \Psi_i^{-1}) + \tau \quad \text{and} \quad
g = \sum_{i \in S -\star_S} \check{b}_i  \log(\check{q}_i \circ {\check{\Psi}}_i^{-1}) + \check{\tau}.$$
First of all, we can smoothly deform $f$ and $g$ through
smooth functions compatible with $(V_i)_{i \in S}$
by changing $q_i,\check{q}_i$, $b_i$, $\check{b}_i$ and $\tau$ and $\check{\tau}$
so that $q_i = \check{q}_i$, $b_i = \check{b}_i$ and $\tau = \check{\tau} = 0$.
Hence we can assume that
$$f = \sum_{i \in S -\star_S} b_i  \log(q_i \circ \Psi_i^{-1}) \quad \text{and} \quad
g = \sum_{i \in S -\star_S} b_i  \log(q_i \circ {\check{\Psi}}_i^{-1}).$$
Lemma \ref{lemma:connectedregularizations}
tells us that there is a smooth family of regulations $(\Psi_i^t)_{i \in S - \star_S}, \ t \in [0,1]$ compatible with
$V_{\star_S}$ connecting $(\Psi_i)_{i \in S - \star_S}$
and $(\check{\Psi}_i)_{i \in S - \star_S}$ and hence we get a smooth family
of functions
$$f_t = \sum_{i \in S -\star_S} b_i  \log(q_i \circ (\Psi_i^t)^{-1})$$
compatible with $(V_i)_{i \in S}$ (after possibly shrinking the region
on which $q_i$ is not equal to $1$ so that it fits inside $\cap_{t \in [0,1]} \Dom(\Psi_t)$).
This is a smooth family of functions compatible with $(V_i)_{i \in S}$
joining $f$ and $g$.
\qed


\begin{lemma} \label{defn:contacthypersurfacegrading}
Let $(X,\omega)$ be a symplectic manifold with a choice of grading
and $C \subset X$ a contact hypersurface with a contact form
$\alpha$ compatible with the contact structure
satisfying $d\alpha = \omega|_C$.
Then $(C,\ker(\alpha))$ has a natural induced choice of grading.
\end{lemma}
We will call such a grading the {\it induced grading on $C$}.

\proof
Let $X_\alpha$ be a smooth section of the bundle $TX|_C \lra{} C$ equal to the $\omega$-dual of $\alpha$.
Since $d\alpha = \omega|_C$ and $\alpha$ is a contact form,
we get that $X_\alpha$ is transverse to $C$.
Let $R$ be the Reeb vector field of $\alpha$
and define $\xi_C \equiv \ker(\alpha)$.
Let $H \subset TX|_C$ be the $2$-dimensional symplectic vector subbundle
spanned by $X_\alpha$ and $R$.
Then we have the following direct sum decomposition
of symplectic vector bundles
$$(TM,\omega)|_C \cong (\xi_C,d\alpha) \oplus (H,\omega|_H).$$
Since $X_\alpha,R$ is a symplectic basis for $H$ at each point of $C$,
we have a natural symplectic trivialization
of $(H,\omega|_H)$ sending $X_\alpha, R$
to the standard symplectic basis vectors on $\C$.
Hence we have a natural isomorphism:
\begin{equation} \label{eqn:bundlesplittinycontact}
(TM,\omega)|_C \cong (\xi_C,d\alpha) \oplus (\C,\omega_{std}).
\end{equation}
Now choose an almost complex structure $J$ on $M$
compatible with $\omega$ so that its restriction to
$TM|_C$ is equal to $J_C \oplus i$
with respect to the splitting 
(\ref{eqn:bundlesplittinycontact})
where $J_C$ is an almost complex structure on $\xi_C$
compatible with $d\alpha|_{\xi_C}$
and $i$ is the standard complex structure on $\C$.

Now we will use the natural correspondence between
gradings and trivializations of the canonical bundle
as stated in Appendix A.
Let $\Phi : \kappa_J \lra{} X \times \C$
be the choice of trivialization of
the canonical bundle of $(TM,J)$
associated to the grading on $(X,\omega)$
(See Definition \ref{defn:gradingtrivializationcorrespondence}
in Appendix A).
Since $J|_C = J_C \oplus i$
under the splitting (\ref{eqn:bundlesplittinycontact}),
we get a natural trivialization $\Phi_C : \kappa_{J_C} \lra{} C \times \C$ induced from the trivialization $\Phi|_C$.
The induced grading on $(C,\xi_C)$
is then the grading associated to the trivialization $\Phi_C$ as in Definition \ref{defn:gradingtrivializationcorrespondence}.
\qed

\bigskip

We wish to use functions compatible with a resolution divisor to construct it's `link'. The following proposition tells us how to at least start doing this.

\begin{prop} \label{proposition:linkprop}
Let $(\cO_X(\sum_{i \in S} m_i V_i), \Phi,\theta)$
be a model resolution.
Define $K \equiv \cup_{i \in S - \star_S} V_i$.
Let $f : X - K \lra{} \R$
be a smooth function compatible with $(V_i)_{i \in S}$.
Then there is a smooth function $g : X - K \lra{} \R$ and an open neighborhood $U$ of $K$ so that
$df(X^\omega_{\theta + dg})|_U > 0$
and
$df_\star(X^{\omega_\star}_{\theta_\star + dg_\star})|_{U \cap V_{\star_S}} > 0$
where $\omega$ is the symplectic form associated to our model resolution,
$\omega_\star \equiv \omega|_{V_{\star_S}}$,
$f_\star \equiv f|_{V_{\star_S} - K}$ and
$g_\star \equiv g|_{V_{\star_S} - K}$.
\end{prop}

The proof of this proposition
is almost exactly the same as the proof of \cite[Proposition 5.8]{McLean:isolated}.
The only difference is that we have to take into account the additional submanifold $V_{\star_S}$.
For the sake of completeness we have produced the proof below.
We also have a parameterized version of the proposition above.

\begin{prop} \label{proposition:linkpropparameterized}
	Let ${\mathcal M}_t \equiv (\cO_X(\sum_{i \in S} m_i V_i), \Phi,\theta_t), \ t \in [0,1]$
	be a smooth family of model resolutions.
	Define $K \equiv \cup_{i \in S - \star_S} V_i$.
	Let $f_t : X - K \lra{} \R,\ t \in [0,1]$
	be a smooth family of functions compatible with $(V_i)_{i \in S}$.
	Then there is a smooth family of functions $g_t : X - K \lra{} \R, \ t \in [0,1]$ and an open neighborhood $U$ of $K$ so that
	$df_t(X^{\omega_t}_{\theta + dg_t})|_U > 0$
	and
	$df_{\star,t}(X^{\omega_{\star,t}}_{\theta_\star + dg_{\star,t}})|_{U \cap V_{\star_S}} > 0$
	where $\omega_t$ is the symplectic form associated to ${\mathcal M}_t$,
	$\omega_{\star,t} \equiv \omega_t|_{V_{\star_S}}$,
	$f_{\star,t} \equiv f_t|_{V_{\star_S} - K}$ and
	$g_{\star,t} \equiv g_t|_{V_{\star_S} - K}$ for all $t \in [0,1]$.
\end{prop}

The proof of this proposition is almost exactly the same as the proof of Proposition \ref{proposition:linkprop} except that all variables are now parameterized by $t$.
Therefore for notational simplicity we will just prove Proposition \ref{proposition:linkprop}.
Before we prove this we need a few preliminary technical lemmas.
The following lemma is very similar to \cite[Lemma 5.12]{McLean:isolated}.
This lemma should be thought of as a local version of Proposition \ref{proposition:linkprop}.

\begin{lemma} \label{lemma:locallink}
Let $(\cO_X(\sum_{i \in S} m_i V_i), \Phi,\theta)$
be a model resolution and fix a metric $\|\cdot\|$ on $X$.  
Define $K \equiv \cup_{i \in S - \star_S} V_i$.
Fix $I \subset S$ and let $\omega$ be the symplectic form associated to our model resolution.
Let $U \subset K$ be an open set with the property that $\overline{U} \cap V_{I'}$ is contained inside a contractible Darboux chart of $V_{I'}$ for all $I' \subset S$
and so that $\overline{U} \cap V_i = \emptyset$ for all $i \in S - I$.

Then there is a smooth function $g : X - K \lra{} \R$ so that for any function $f : X - K \lra{} \R$ compatible with $(V_i)_{i \in S}$,
we have
\begin{enumerate}
\item \label{item:estimate1}
$df(X^\omega_{\theta + dg})|_{W_f} > c_f \| \theta + dg\| \|df\||_{W_f}$ for some constant $c_f > 0$ and some small neighborhood $W_f$ of $\overline{U} \cap V_I$,
\item \label{item:estimate2}
$df_\star(X^{\omega_\star}_{\theta_\star + dg_\star})|_{W_{f,\star}} > c_f \| \theta_\star + dg_{p,\star}\| \|df_\star\||_{W_{f,\star}}$ where $f_\star \equiv f|_{V_{\star_S} - K}$, $W_{f,\star} \equiv W_{f,\star} \cap V_{\star_S}$,  $\theta_\star \equiv \theta|_{V_{\star_S} - K}$ and $g_\star \equiv g|_{V_{\star_S} - K}$ and
\item \label{item:estimate3}
$a_1 \|df\| < \|\theta + dg\| < a_2 \|df\|$ inside $W_f$
for some constants $a_1,a_2>0$.
\end{enumerate}
\end{lemma}
\proof
Define $\widehat{I} \equiv I - \star_S$ and let $n$ be the dimension of $X$ divided by $2$.
Since $\overline{U} \cap V_i$ is contained inside a contractible Darboux chart we have,
by a Moser argument, symplectic coordinates
$x^i_1,y^i_1,\cdots,x^i_n,y^i_n$
defined on some neighborhood $W_i$ of $\overline{U} \cap V_i$ in $X$
so that $V_i \cap W_i = \{x^i_1 = y^i_1 = 0\}$ for each $i \in \widehat{I}$.
We can also choose $W_i$ so that $W_i \cap V_j = \emptyset$ for all $j \in S - I$ and so that $\overline{W}_i$ is a compact contractible codimension $0$ submanifold of $X$ with boundary so that $\partial \overline{W}_i$, $(V_j)_{j \in S}$ are transversely intersecting.
Define $W'_i \equiv W_i - V_i$
and let $P_i : \widetilde{W}'_i \lra{} W'_i$ be the cover corresponding to the subgroup of $\pi_1(W'_i)$
generated by loops wrapping around $V_i$ near $V_i$
for each $i \in \widehat{I}$.
In other words, the cover corresponding to the image of $\pi_1(T_i - V_i)$ in $\pi_1(W_i')$ where $T_i$ is a small tubular neighborhood of $V_i \cap W_i$ in $W_i$.
Let $r_i : W'_i \lra{} \R, \ \vartheta_i : W'_i \lra{} \R/2\pi\Z$
be functions satisfying $x^i_1 = r_i \cos(\vartheta_i)$, $y^i_1 = r_i \sin(\vartheta_i)$.
Define $\rho_i \equiv \frac{1}{2}(r_i)^2$, $\widetilde{\rho}_i \equiv P_i^* \rho_i$,
$\widetilde{y}^i_j \equiv P_i^* y^i_j$
and $\widetilde{x}^i_j \equiv P_i^* x^i_j$ for all $i \in \widehat{I}$ and $j \in \{1,\cdots,n\}$.
Let $\widetilde{\vartheta_i} : \widetilde{W}'_i \lra{} \R$ be smooth function whose value mod $2\pi$ is equal to $P_i^* \vartheta_i$.
Also let $\widetilde{\omega} \equiv P_i^* (\omega|_{W'_i})$ and $\widetilde{\theta} \equiv P_i^*(\theta|_{W'_i})$.

Then $\widetilde{\omega} = d(\widetilde{\rho}_i) \wedge d(\widetilde{\vartheta}^i) + \sum_{j = 2}^n d(\widetilde{x}^i_j) \wedge d(\widetilde{y}^i_j)$.
Therefore there is a natural symplectic embedding of $\iota_i : \widetilde{W}'_i \hookrightarrow \C^{n+1}$
so that the standard symplectic coordinates $x_1,y_1,\cdots,x_n,y_n$ in $\C^n$ restricted to $\widetilde{W}'_i$
are $\widetilde{\rho}_i,\widetilde{\vartheta}_i,\widetilde{x}^i_2,\widetilde{y}^i_2,\cdots,\widetilde{x}^i_n,\widetilde{y}^i_n$ respectively.
%
Let $\widehat{W}'_i \subset \C^{n+1}$ be the closure of $P_i^{-1}(W'_i)$ inside $\C^{n+1}$ and let
$\widehat{V}'_{I-i}$ be the closure of $P_i^{-1}(V_{I - i})$
inside $\C^{n+1}$.
Then $\widehat{W}'_i$ is a codimension $0$ submanifold of $\C^{n+1}$ with boundary and corners
and $\widehat{V}'_{I-i}$
is a codimension $2(|I|-1)$ submanifold of $\widehat{W}'_i$ with boundary and corners where one part of the boundary is
$\check{V}_{I-i}  \equiv \{x_1 = 0\} \cap \widehat{V}'_{I-i}$.
Let $\overline{W}_i$ be the closure of $W_i$ in $X$.
The map $P_i$ extends to a map
$\overline{P}_i : \widehat{W}'_i \lra{} \overline{W}_i$
whose fibers over $\overline{W}_i \cap V_i$ are one dimensional.
Also $\check{V}_{I-i}$ is equal to $\overline{P}_i^{-1}(V_I)$ and $\widehat{V}'_{I-i} = \overline{P}_i^{-1}(V_{I-i})$.

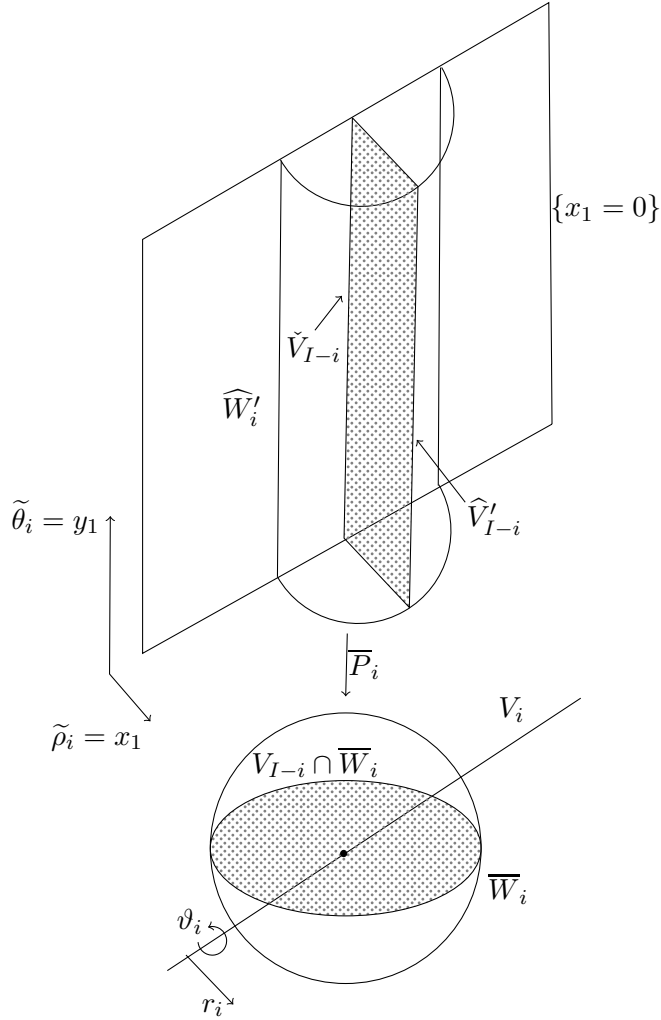
\bigskip

\bigskip

\begin{tikzpicture}[scale = 0.9]

\usetikzlibrary{patterns}

\draw  (0,-0.5) node (v1) {} ellipse (2 and 2);

\draw (-2.6349,-2.3095) -- (3.4721,1.7186);

\draw  (v1) ellipse (0 and 0);

\draw[pattern=crosshatch dots,pattern color=gray] (0,-0.5) ellipse (2 and 1);

\draw [->](0.0234,2.6741) -- (0,1.746);crosshatch dots

\draw (3.0476,5.7698) {} -- (-3,2.3809) -- (-3,8.5) -- (3,12) -- (3.0476,5.7698);

\draw (-0.9524,9.6635) -- (-1.0083,3.5122);

\draw (1.4011,11.0402) -- (1.3673,4.8615);

\draw (-0.9645,9.6783) arc (-150.1722:29.7449:1.3723);
\draw (-1.0083,3.5122) arc (-150.1722:29.7449:1.3723);

\draw[pattern=crosshatch dots,pattern color=gray] (0.099,10.2974) node (v3) {} -- (1.0673,9.2815) -- (0.9403,3.0593) -- (-0.0243,4.0839) -- (0.099,10.2974);

\node at (0.2736,2.218) {$\overline{P}_i$};

\node at (3.8451,8.9481) {$\{x_1=0\}$};

\node at (2.4006,-1.0994) {$\overline{W}_i$};

\node at (-1.5359,6.0117) {$\widehat{W}'_i$};

\node at (-0.4499,0.7549) {$V_{I-i} \cap \overline{W}_i$};

\node at (2.4482,1.5513) {$V_i$};

\node at (2.1943,4.3926) {$\widehat{V}'_{I-i}$};

\draw[->] (1.7657,4.6148) -- (1.0514,5.5355);

\draw [<->](-2.8851,1.3767) -- (-3.4883,2.0752) -- (-3.4771,4.4085);

\node at (-3.6718,1.1387) {$\widetilde{\rho}_i = x_1$};

\node at (-4.2862,4.3926) {$\widetilde{\theta}_i = y_1$};

\draw [->](-2.3592,-2.0788) -- (-1.6608,-2.8089);

\node at (-1.9624,-2.8566) {$r_i$};

\draw [->](-2.1806,-1.888) arc (-174.9461:107.1027:0.2112);

\node at (-2.3037,-1.5825) {$\vartheta_i$};

\node at (-0.4733,6.8981) {$\check{V}_{I-i}$};

\draw [->](-0.4544,7.1569) -- (-0.0576,7.6648);

\draw[fill]  (-0.0297,-0.5781) circle (0.0452);

\end{tikzpicture}

Let $w_i>0$ be the wrapping number of $\theta$ around $V_i$.
Let $H \subset T\C^{n+1}|_{\check{V}_{I-i}}$ be a $2$ dimensional symplectic subbundle over $\check{V}_{I-i}$ containing $\ker{D\overline{P}_i|_{\check{V}_{I-i}}} = \text{Span}(\frac{\partial}{\partial y_1})|_{\check{V}_{I-i}}$
so that $H$ is contained in $T\widehat{V}'_{I-i}$.
Let $T^\perp \widehat{V}'_{I-i}$ be the set of vectors which are symplectically orthogonal to $T\widehat{V}'_{I-i}$ and define
$\widehat{H} \equiv H \oplus T^\perp \widehat{V}'_{I-i}|_{\check{V}_{I-i}}$.

Choose a smooth function $\widetilde{g}_i : \C^{n+1} \lra{} \R$ so that:
\begin{enumerate}[label=(\alph*)]
\item \label{eqn:wrappingshiftproperty}
 $\widetilde{g}_i(x_1,y_1,\cdots,x_n,y_n) = \widetilde{g}_i(x_1,y_1 + 2\pi, x_2,y_2,\cdots,x_n,y_n) + w_i$,
\item $dx_1(X_{d\widetilde{g}_i})>0$
at each point of $\check{V}_{I-i}$ and
\item \label{item:tangencycondition}
the $\omega_{\C^{n+1}}|_{\widehat{H}}$-dual of $d\widetilde{g}_i|_{\widehat{H}}$
is tangent to $\widehat{V}'_{I-i}$ at each point of $\check{V}_{I-i}$ where $\omega_{\C^{n+1}}$ is the standard symplectic structure on $\C^{n+1}$.
\end{enumerate}

Condition \ref{item:tangencycondition}
implies that $X_{d\widetilde{g}_i}$ is tangent
to $\widehat{V}'_{I-i}$ at each point of $\check{V}_{I-i}$.
By \ref{eqn:wrappingshiftproperty}
there is a closed $1$-form $\beta_i \in \Omega^1(W'_i)$
whose pullback to $\widetilde{W}'_i$
is equal to $d\widetilde{g}_i|_{\widetilde{W}'_i}$.

Define $W' \equiv \cap_{i \in \widehat{I}} W_i$.
Let $\theta_1 \in \Omega^1(W')$ be any
$1$-form of bounded norm
satisfying $d\theta_1 = \omega|_{W'}$.
Define
$$\Theta \in \Omega^1(W'-K), \quad \Theta \equiv \theta_1 + \sum_{i \in \widehat{I}} \beta_i.$$
By \ref{eqn:wrappingshiftproperty},
we get that the wrapping number of $\Theta$
around $V_i \cap W'$ is $w_i$ for all $i \in \widehat{I}$.
Hence (after shrinking $(W_i)_{i \in \widehat{I}}$ so that $W'$ deformation retracts on to $W' \cap K$)
there is a function $g : X - K \lra{} \R$
so that $(\theta + dg)|_{W'} = \Theta|_{W'}$.

Let $f : X - K \lra{} \R$ be compatible with
$(V_i)_{i \in S}$
and let $(\Psi_i)_{i \in S - \star_S}$ be its associated regularizations.
Then by definition $f|_{W'} = \sum_{i \in \widehat{I}} b_i \log(q_i \circ \Psi_i^{-1}) + \tau$
where 
\begin{itemize}
\item $\tau : W' \lra{} \R$ is a smooth function and $b_i>0$ are constants for all $i \in \widehat{I}$,
\item $q_i : \Dom(\Psi_i) \lra{} \R$
is equal to a square norm near $V_i$,
equal to $1$ outside a compact subset of $\Dom(\Psi_i)$ and non-zero on $\Dom(\Psi_i) - V_i$ and
\item $\log(q_i \circ \Psi_i^{-1})$
is defined to be zero outside $\Im(\Psi_i)$.
\end{itemize}
Define $s_i \equiv \log(q_i \circ \Psi_i^{-1})$
and $s_{i,\star} \equiv s_i|_{V_{\star_S} - K}$.
Define $\beta_{i,\star} \equiv \beta_i|_{V_{\star_S} \cap (W' - K)}$.
Since
\begin{equation}\label{eqn:boundbyoneoverdistance}
c_1/r_i < \|\beta_i\| < c_2/r_i, \quad
c_1/r_i < \|ds_i\| < c_2/r_i
\end{equation}
for some $c_1,c_2>0$ and  $\|\theta_1\|$ is bounded,
part (\ref{item:estimate3}) of our lemma holds.

Since $\|\tau\|$ and $\|\theta_1\|$ are bounded and since Equation (\ref{eqn:boundbyoneoverdistance}) holds,
it is sufficient for us to prove the following statements:
\begin{enumerate}
\item $ds_i(\sum_{j \in \widehat{I}} X_{\beta_i}) > \frac{c_f}{(r_i)^2}$ inside some small neighborhood $W_f$ of $W' \cap V_I$ and some constant $c_f > 0$  for all $i \in \widehat{I}$,
\item $ds_{i,\star}(\sum_{j \in \widehat{I}} X_{\beta_{j,\star}}) > \frac{c_f}{(r_i)^2}$ inside $W_f \cap V_{\star_S}$   for all $i \in \widehat{I}$.
\end{enumerate}

Let $\widehat{V}_{\star_S}$ be closure of $P_i^{-1}(V_{\star_S})$ inside $\C^{n+1}$.
Let $|\cdot|$ be the standard norm on $\C^{n+1}$.
Since there is  diffeomorphism $\Phi_i : W_i \lra{} W_i$
which is the identity on $V_i \cap W_i$, fixing $V_j \cap W_i$ for all $j \in I-i$
that is also isotopic to the identity through such diffeomorphisms and
pulling back $s_i$ to $\log(\rho_i)$, we have that
inequalities (1) and (2) above can be deduced from:
\begin{enumerate}
	\item $d\widetilde{\rho}_i(D\widetilde{\Phi}_i(X_{d\widetilde{g}_i})) > c'_f$ inside
	$\widetilde{W}_f - \overline{P}_i^{-1}(V_i)$  for some small neighborhood $\widetilde{W}_f$ of $\check{V}_{I-i}$ and some constant $c'_f > 0$ where $\widetilde{\Phi}_i : \widetilde{W}_i' \lra{} \widetilde{W}_i'$ is a lift of $\Phi|_{W_i'} : W_i' \lra{} W_i'$  ,
	\item $d\widetilde{\rho}_i(D\widetilde{\Phi}_i(X_{P_i^*\beta_j})) \to 0$ as we approach 
	$\check{V}_{I-i}$
	for all $j \in \widehat{I} - i$ and
	\item $X_{d\widetilde{g}_i}$  is tangent to $\widehat{V}_{\star_S}$ along $\check{V}_{I-i}$ for all $i \in \widehat{I}$.
\end{enumerate}
These properties follow from \ref{eqn:wrappingshiftproperty}-\ref{item:tangencycondition} above.
\qed

\begin{lemma} \label{lemma:boundsontheta}
Let $(\cO_X(\sum_{i \in S} m_i V_i), \Phi,\theta)$
be a model resolution and fix a metric $\|\cdot\|$ on $X$. 
Define $K \equiv \cup_{i \in S - \star_S} V_i$ and let
$f : X - K \lra{} \R$ be a smooth function
compatible with $(V_i)_{i \in S}$.
Then there is a smooth function $h : X - K \lra{} \R$ and constants $a_1, a_2 > 0$
so that
$a_1 \|df\| < \|\theta + dh\| < a_2 \|df\|$ near $K$.
\end{lemma}
\proof
We will use part (\ref{item:estimate3}) of Lemma \ref{lemma:locallink}
and an induction argument to do this.
Choose open sets $U_1,\cdots,U_m$ in $X$ along with subsets $I_1, \cdots I_m \subset S$
so that 
\begin{itemize}
	\item $\cup_{j=1}^m (U_j \cap V_{I_j}) = K$,
	\item $\overline{U}_j \cap V_{I'}$ is contained inside a contractible Darboux chart of $V_{I'}$
	for all ${I'} \subset S$, $j \in \{1,\cdots,m\}$ and 
	\item so that $\overline{U}_j \cap V_{k} = \emptyset$ for all $j \in S - I_j$ and all $j \in \{1,\cdots,m\}$.
\end{itemize}
Let $\check{U}_1,\cdots,\check{U}_m$ be open sets of $X$ so that the closure of $\check{U}_i$ is contained inside $U_i$ for all $i \in \{1,\cdots,m\}$
and $\cup_{j=1}^m (\check{U}_j \cap V_{I_j}) = K$.
Define $U_{<k} \equiv \cup_{j<k}U_j$ and $\check{U}_{<k} \equiv \cup_{j<k}\check{U}_j$.

Suppose, by induction, there is a smooth function $h_{\prec} : X - K \lra{} \R$ and constants $a_1^{\prec},a_2^{\prec} > 0$ so that
\begin{equation} \label{eqn:hprecinequality}
a_1^{\prec} \|df\| < \|\theta + dh_{\prec}\| < a_2^{\prec} \|df\|
\end{equation}
on a neighborhood $N \subset U_{<k}$ of the closure of $\check{U}_{<k} \cap K$
for some $k \in \{1,\cdots,m\}$.
By Lemma \ref{lemma:locallink},
there is a function $h_= : X - K \lra{} \R$
and constants $a^=_1, a^=_2 > 0$ so that
\begin{equation} \label{eqn:hkinequality}
a^=_1 \|df\| < \|\theta + dh_{\prec} + dh_=\| < a^=_2 \|df\|
\end{equation}
on a neighborhood $W_k$ of $U_k \cap K$ in $X$.
Now let $\rho : X \lra{} [0,1]$
be a smooth function
equal to $0$ on a neighborhood of $\overline{\check{U}_{<k}} \cap K$
and which is $1$ outside a compact subset of $N$.
Define $h_{\preceq} \equiv h_{\prec} + \rho h_=$.
Equations
(\ref{eqn:hprecinequality}) and (\ref{eqn:hkinequality}) tell us that
$\|dh_=\| \leq (a_2^{\prec} + a_2^=)\|df\|$
near $U_k \cap (N - \check{U}_{<k}) \cap K$
 which implies that $|h_=| < C |f| + \check{C}$
 near $(N - \check{U}_{<k}) \cap K$
for some $C,\check{C}>0$.
Hence
$$a^{\preceq}_1 \|df\| < \|\theta + dh_{\preceq}\| < a^{\preceq}_2 \|df\|$$
near $\overline{\check{U}_{<k+1}} \cap K$ for some $a^{\preceq}_1,a^{\preceq}_2>0$
and so we are finished by induction.
\qed

\bigskip

\begin{proof}[Proof of Proposition \ref{proposition:linkprop}]
By Lemma \ref{lemma:boundsontheta}
we can add an exact $1$-form to $\theta$
so that 
\begin{equation} \label{eqn:boundfortheta}
b_1 \|df\|  < \|\theta\| < b_2 \|df\|
\end{equation}
inside a neighborhood $N$ of $K$ for some constants $b_1,b_2>0$.
By Lemma \ref{lemma:locallink} we can find
open sets $W_1,\cdots,W_m$ of $X$ covering $K$,
smooth functions
$g_i : X - K \lra{} \R$, $i = 1,\cdots,m$,
and a constant $c>0$ so that
\begin{enumerate}
\item \label{item:ginequality}
$df(X^\omega_{\theta + dg_i})|_{W_i} > c \| \theta + dg_i\| \|df\||_{W_i}$
\item \label{item:gstarinequality}
$df_\star(X^{\omega_\star}_{\theta_\star + dg_i})|_{W_{\star}} > c \| \theta + dg_{i,\star}\| \|df_\star\||_{W_{i,\star}}$ where $f_\star \equiv f|_{V_{\star_S} - K}$, $W_{i,\star} = W_i \cap V_{\star_S}$,  $\theta_\star = \theta|_{V_{\star_S} - K}$ and $g_{i,\star} = g_i|_{V_{\star_S} - K}$ and
\item \label{item:boundforthetaplusdgi}
$c \|df\| < \|\theta + dg_i\| < \check{c} \|df\|$ inside $W_i$ for some constants $c,\check{c}>0$.
\end{enumerate}

Now choose smooth functions $\rho_i, i = 1,\cdots, m$ 
so that $\sum_{i=1}^m \rho_i|_K = 1$
and $\rho_i = 0$ outside a compact subset of $W_i$ for each $i = 1,\cdots,m$.
We define
$$g : X - K \lra{} \R, \quad g \equiv \sum_{i = 1}^m \rho_i g_i.$$
We define $g_{\star} \equiv g|_{V_{\star_S} - K}$.
The inequality (\ref{eqn:boundfortheta})
combined with property (\ref{item:boundforthetaplusdgi}) above
tells us that
$|g_i| < C|f| + \check{C}$ for some $C,\check{C}>0$ near $W_i \cap K$.
This means that
$df(X^\omega_{\theta + dg}) > c \| \theta + dg\| \|df\|$ near $K$,
$df(X^\omega_{\theta_\star + dg_\star}) > c \| \theta_\star + dg_\star\| \|df\|$
near $V_{\star_S} \cap K$
and $a_1 \|df\|  < \|\theta + dg\| < a_2 \|df\|$ near $K$ for some constants $a_1,a_2>0$.
\end{proof}

\begin{defn} \label{defn:linkofmodelresolution}
	The {\it link} of a model resolution
		$(\cO_X(\sum_{i \in S} m_i V_i),
		\Phi,\theta)$
	is a contact pair with normal bundle data
	$(B \subset C,\xi_C,\Phi_B)$ defined as follows:
	By \cite[Lemma 4.1]{McLeanTehraniZinger:normalcrossings},
	there is a tuple of regularizations $(\Psi_i)_{i \in S- \star_S}$
	compatible with $V_{\star_S}$ as in Definition \ref{defn:comptiblewithvstars}.
	Define $K \equiv \cup_{i \in S - \star_S} V_i$.
	Let $f : X - K \lra{} \R$
	be a smooth function compatible with $(V_i)_{i \in S}$
	so that $(\Psi_i)_{i \in S - \star_S}$ are associated regularizations of $f$.
	Then $f_\star \equiv f|_{V_{\star_S}}$ is a smooth function compatible
	with $(V_i \cap V_{\star_S})_{i \in S - \star_S}$.
	Hence by Proposition \ref{proposition:linkprop}
	we have that
	$df(X_{\theta + dg}) > 0$ and  $df_\star(X_{\theta_\star + dg_\star}) > 0$
	in an open neighborhood $U$ of $K$ for some smooth function $g : X - K \lra{} \R$ where $g_{\star} \equiv g|_{V_{\star_S} - K}$ and $\theta_\star \equiv \theta|_{V_{\star_S} - K}$.
	Let $c \ll -1$ be a constant satisfying
	$f^{-1}(c) \subset U$.
	Define
	$$C \equiv f^{-1}(c), \quad B \equiv C \cap V_{\star_S}, \quad \xi_C \equiv \ker{(\theta+dg)|_C}.$$
	Finally the trivialization $\Phi_B$ of the normal bundle of $B$ in $C$
	is induced from the trivialization $\Phi$ since the normal bundle
	of $B$ is naturally isomorphic as an oriented vector bundle
	to $\cN_X{V_{\star_S}}|_B$ which in turn is naturally isomorphic to
	$\cO_X(\sum_{i \in S} m_i V_i)|_B$ since $m_{\star_S} = 1$.
	
	If $(\cO_X(\sum_{i \in S} m_i V_i),
	\Phi,\theta)$ has a grading
	then this gives us an induced grading
	on the link as in the proof of Lemma \ref{defn:contacthypersurfacegrading} since $C - B$ is a contact hypersurface
	of $(X - \cup_{i \in S - \star_S} V_i, \omega)$
	where $\omega$ is the symplectic form associated to
	our model resolution.
	We will call this the {\it induced grading} on
	$(B \subset C,\xi_C,\Phi_B)$.
	\end{defn}

	The link does not depend on the choice of neighborhood $U$,
	constant $C$ or function $f$ by the following Lemma:
	\begin{lemma} \label{lemma link invariance}
		Suppose that $(\cO_X(\sum_{i \in S} m_i V_i),
		\Phi,\theta)$ and $(\cO_{\widehat{X}}(\sum_{i \in \widehat{S}} \widehat{m}_i \widehat{V}_i),
		\widehat{\Phi},\widehat{\theta})$ are (graded) isotopic.
		Then their links are also (graded) isotopic for any choice of neighborhood $U$ constant $C$ or function $f$ chosen for each of these two model resolutions.
	\end{lemma}
	\proof
		This follows from Lemma 
		\ref{lemma:compatiblefunctionsareconnected}
		and Proposition \ref{proposition:linkpropparameterized}.
	\qed

\subsection{Constructing a Contact Open Book from a Model Resolution}
\label{section making the monodromy nice}

The aim of this section is to construct a contact open book for each model resolution so that the contact pair associated to this open book is the link of our model resolution.

\begin{defn} \label{defn:oneformcompatiblewithregularization}
Let $(\cO_X(\sum_{i \in S} m_i V_i),
\Phi,\theta)$ be a model resolution with associated symplectic structure $\omega$. Suppose
$(V_i)_{i \in S}$ admits an $\omega$-regularization
\begin{equation} \label{eqn:modelregularization}
{\mathcal R} \equiv ((\rho_i)_{i \in S}, (\Psi_I)_{I \subset S})
\end{equation}
with associated Hermitian structures $(\rho_{I;i},\nabla^{(I;i)})$
on $\cN_X V_i|_{V_I}$ for each $i \in I \subset S$.
Let $\alpha_{I;i} \equiv \alpha_{\rho_{I;i},\nabla^{(I;i)}} \in \Omega^1(\cN_X V_i|_{V_I} - V_I)$ 
be the associated Hermitian connection $1$-form
on $\cN_X V_i|_{V_I}$.
Define $K \equiv \cup_{i \in S - \star_S} V_i$.
Let $w_i$ be the wrapping number of $\theta$
around $V_i$ for each $i \in S - \star_S$ and define $a_{\star_S} \equiv 0$.
Let
\begin{equation} \label{eqn:projectionmap}
pr_{I;i} : \cN_X V_I \lra{} \cN_X V_i|_{V_I} 
\end{equation}
be the natural projection map for all $I \subset S$.
We say that $\theta$ is
{\it compatible with ${\mathcal R}$} if the restriction of
\begin{equation} \label{eqn:1formequation}
(\Psi_I)^*\theta
- \sum_{i \in I} pr_{I;i}^* \left(\left(\rho_{I;i} + \frac{w_i}{2\pi}\right) \alpha_{I;i}\right)
\end{equation}
to each fiber of $\pi_{\cN_X V_I}|_{\Psi_I^{-1}(X-K)}$ is $0$
for every $I \subset S$.
\end{defn}

\begin{lemma} \label{lemma:findingcompatibleoneform}
Let $(\cO_X(\sum_{i \in S} m_i V_i),
\Phi,\theta)$ be a model resolution so that
$(V_i)_{i \in S}$ admits an $\omega$-regularization ${\mathcal R}$.
Define $K \equiv \cup_{i \in S - \star_S}V_i$.
Then there is a smooth function
$g : X - K \lra{} \R$ so that
$\theta+dg$ is compatible with $\check{\mathcal R}$
for some regularization $\check{\mathcal R}$ which is germ equivalent
to ${\mathcal R}$.
\end{lemma}
\proof
This is done by induction on the strata of $\cup_i V_i$.
We will use the notation from Definition
\ref{defn:oneformcompatiblewithregularization} above.
Let $\preceq$ be a total order on the set of subsets of $S$ with the property that if $|I'| < |I|$ then $I \preceq I'$.
We write $I \prec I'$ when $I \preceq I'$ and $I \neq I'$.
Suppose for some $I^* \subset S$,
we have constructed open sets
$U^{\prec}_I$
inside $\Dom(\Psi_I)$ containing $V_I$ for all $I \prec I^*$ and
a smooth function
$g^{\prec} : X - K \lra{} \R$
with the property that
\begin{equation} \label{eqn:1formequationinductionstep}
(\Psi_I)^*(\theta + dg^{\prec})
- \sum_{i \in I} pr_{I;i}^* \left(\left(\rho_{I;i} + \frac{w_i}{2\pi}\right) \alpha_{I;i}\right)
\nonumber
\end{equation}
vanishes along each fiber of
$\pi_{\cN_X V_I}|_{U_I^{\prec} \cap \Dom(\Psi_I)}$
for all $I \prec I^*$.

We now want these properties to hold for all $I \preceq I^*$. 
For all $I \preceq I^*$, let $U^{\preceq}_I \subset \text{Dom}(\Psi_I)$
be open sets  containing $V_I$ whose closure is compact
so that the closure of $U^{\preceq}_I$
is contained in $U^\prec_I$ when $I \prec I^*$
and so that
$\pi_{\cN_X V_I}|_{U_I^{\preceq}}$
has contractible fibers for all $I \preceq I^*$.
Since the wrapping number of
$\theta$ around $V_i$ is $w_i$ for all $i \in I$
and since
$(\Psi_I)^* (d\theta) = \omega_{(\rho_{I;i},\nabla^{(I;i)})_{i \in I}}|_{\Dom(\Psi_I)}$,
we get that the restriction of
$$A \equiv (\Psi_{I^*})^* (\theta + dg^{\prec}) - \sum_{i \in I^*} pr_{I^*;i}^* \left(\left(\rho_{I^*;i} + \frac{w_i}{2\pi}\right) \alpha_{I^*;i}\right)$$
to each fiber of $\pi_{\cN_X V_I}|_{\Dom(\Psi_{I^*})}$
is exact.
Also by our induction hypothesis,
the restriction of $A$ to the fibers of $\pi_{\cN_X V_{I^*}}|_{(\Psi_{I^*})^{-1}(\Psi_I(U_I^{\prec}))}$ is $0$
for all $I \subsetneq I^*$.
This means that
there is a smooth function
$g^= : X - K\lra{} \R$
so that $g^=$ restricted to a small neighborhood of the closure
of $\Psi_I(U_I^{\preceq})$ is $0$
for all $I \subsetneq I^*$
and so that
$A + (\Psi_{I^*})^*dg^=$ restricted
to each fiber of $\pi_{\cN_X V_{I^*}}|_{U_{I^*}^{\preceq} \cap \Dom(\Psi_{I^*})}$ is $0$.
Define $g^\preceq \equiv g^\prec + g^=$.
Then
\begin{equation} \label{eqn:1formequationinductionend}
(\Psi_I)^*(\theta + dg^{\preceq})
- \sum_{i \in I} pr_{I;i}^* \left(\left(\rho_{I;i} + \frac{w_i}{2\pi}\right) \alpha_{I;i}\right)
\nonumber
\end{equation}
vanishes along each fiber of
$\pi_{\cN_X V_I}|_{U_I^{\preceq} \cap \Dom(\Psi_I)}$
for all $I \preceq I^*$.
Hence by induction we have shown that
there is a smooth function $g : X - K \lra{} \R$
and open subsets $U_I \subset \Im(\Psi_I)$
containing $V_I$
so that
$$(\Psi_I)^*(\theta + dg)
- \sum_{i \in I} pr_{I;i}^* \left(\left(\rho_{I;i} + \frac{w_i}{2\pi}\right) \alpha_{I;i}\right)
$$
vanishes along each fiber of
$\pi_{\cN_X V_I}|_{U_I \cap  \Dom(\Psi_I)}$
for all $I \subset S$.
By \cite[Lemma 5.5]{McLeanTehraniZinger:normalcrossings},
we can shrink these open subsets $U_I$
so that
$\check{\mathcal R} \equiv ((\rho_i|_{\Psi_I(U_I)})_{i \in S}, (\Psi_I|_{U_I})_{I \subset S})$
is a regularization.
\qed

\begin{defn} \label{defn:regularizationofmodelresolution}
	Let  
	$${\mathcal M} \equiv (\cO_X(\sum_{i \in S} m_i V_i), \Phi,\theta)$$
	be a model resolution with associated symplectic form $\omega$
	and let $U \subset X$ be an open set.
	A {\it regularization
	of ${\mathcal M}$ of radius $R$ along $U$}
for some $R<1$
	is an $\omega$-regularization ${\mathcal R}$ 
	of $(V_i)_{i \in S}$ as in Equation (\ref{eqn:modelregularization})
	so that
	\begin{enumerate}
		\item the line bundle
		$\cO_X(\sum_{i \in S} m_i V_i)$
		is also defined using the regularization maps $(\Psi_i)_{i \in S}$ from ${\mathcal R}$,
		\item $\Phi$ is radius $R$ compatible with ${\mathcal R}$ along $U$
		as in Definition
		\ref{defn:compatibletrivialization} and
		\item $\theta$ is compatible with ${\mathcal R}$.
	\end{enumerate}
	A {\it regularization
	of ${\mathcal M}$ along $U$}
	is a regularization of ${\mathcal M}$ of radius $R$ along $U$
	for some $R<1$ smaller than the tube radius of ${\mathcal R}$ along $U$.
\end{defn}

We wish to show that every model resolution is isotopic to one admitting a regularization as above. Before we do this we need a preliminary lemma.

\begin{lemma} \label{lemma:thetatmoving}
Let $X$ be a smooth manifold
with a smooth family of cohomologous symplectic forms
$(\omega_t)_{t \in [0,1]}$
and $(V_i)_{i \in S}$ a compact
SC divisor on $X$ with respect to $\omega_t$ for all $t \in [0,1]$.
Define $K \equiv \cup_{i \in S} V_i$ and
let $\theta$ be
a $1$-form on $X - K$ satisfying
$d\theta = \omega_0|_{X - K}$.
Then there exists a smooth family of $1$-forms $(\theta_t)_{t \in [0,1]}$
on $X - K$ so that $d\theta_t = \omega_t|_{X - K}$ for all $t \in [0,1]$
and so that
the wrapping number of $\theta_t$ around $V_j$
does not depend on $t \in [0,1]$ for all $j \in S$.
\end{lemma}
\proof
Since $\omega_t - \omega_0$ is exact for all $t \in  [0,1]$, there is (by exploiting the Hodge decomposition theorem for differential forms)
a smooth family of $1$-forms $(\beta_t)_{t \in [0,1]}$ on $X$ so that
$d\beta_t = \omega_t - \omega_0$.
Let $U \subset X$ be a neighborhood of $K$
whose closure is a compact manifold with boundary
which deformation retracts onto $K$
and $\rho : X \lra{} [0,1]$ a smooth function equal to $0$
near $K$ and equal to $1$ outside a compact subset of $U$.
Define $\check{\beta}_t \equiv \theta + \beta_t \in \Omega^1(X - K)$ for all $t \in [0,1]$.
Since $\rho \check{\beta}_t = 0$ near $K$, we can think of this as a smooth $1$-form on $X$ by defining it to be $0$ along $K$.
Let $c_t \equiv [\omega_t - d(\rho \check{\beta}_t)] \in H^2(X,X-U;\R)$
for all $t \in [0,1]$ and define
$c \equiv [\omega_0 - d(\rho \theta)] \in  H^2(X,X-U;\R)$.
Since $H^2(X,X-U;\R) = H^2(X,X-K;\R)$
we have the long exact sequence:
$$H^1(X-K;\R) \lra{\alpha} H^2(X,X-U;\R) \lra{\check{\alpha}} H^2(X;\R) \lra{} H^2(X - K;\R).$$
Since $\check{\alpha}(c_t) = \check{\alpha}(c) = [\omega_0]$ for all $t \in [0,1]$, we have a smooth family
closed $1$-forms $b_t \in \Omega^1(X - K)$ so that $\alpha(b_t) = c - c_t$ for all $t \in [0,1]$.
Let $\theta_t = \check{\beta}_t + b_t$.
Then $[\omega_t|_U - d(\rho \theta_t)] \in H^2_c(U;\R)$
is independent of $t$ which proves our lemma.
\qed

\begin{lemma} \label{lemma:modelresolutionregulatrizationexistence}
	Let  $(\cO_X(\sum_{i \in S} m_i V_i),
	\Phi,\theta)$ be a model resolution
	and let $U \subset X$ be a relatively compact open set.
	Then $(\cO_X(\sum_{i \in S} m_i V_i),
	\Phi,\theta)$ is isotopic to a model resolution
	$(\cO_X(\sum_{i \in S} m_i V_i),
	\widehat{\Phi},\widehat{\theta})$
	admitting a regularization
	along $U$.
\end{lemma}
\proof
By \cite[Theorem 2.17]{McLeanTehraniZinger:normalcrossings},
there a smooth family of cohomologous
symplectic forms $(\omega_t)_{t \in [0,1]}$
so that
$\omega_0 = \omega$
and $(V_i)_{i \in S}$ admits an $\omega_1$-regularization
$${\mathcal R} \equiv ((\rho_i)_{i \in S},(\Psi_I)_{I \subset S}).$$
Lemma \ref{lemma:thetatmoving} then tells us that there
is a smooth family of $1$-forms
$(\theta_t)_{t \in [0,1]}$ on $X - \cup_{i \in S - \star_S} V_i$
so that $\theta_0 = \theta$
and $d\theta_t = \omega|_{X - \cup_{i \in S - \star_S} V_i}$
for all $t \in [0,1]$ and so that the wrapping number of $\theta_t$ around $V_i$ is independent of $t$ for each $i \in S - \star_S$.
We can assume that $\cO_X(\sum_{i \in S} m_i V_i)$ is defined using
the regularizations $(\Psi_i)_{i \in S}$ as changing the regularization needed to define a line bundle
as in Equation (\ref{eqn:linebundle})
creates an isomorphic line bundle.
Now we isotope $\Phi$ through trivializations to a trivialization $\widehat{\Phi}$ so that
$\widehat{\Phi}$ is compatible with ${\mathcal R}$ along $U$ by Lemma
\ref{lemma:nicetrivialization}.
By
Lemma \ref{lemma:findingcompatibleoneform} we have, after
replacing ${\mathcal R}$ with a germ equivalent regularization,
that $\widehat{\theta} \equiv \theta_1 + dg$ is compatible with ${\mathcal R}$
for some $g \in C^\infty(X - \cup_{i \in S - \star_S} V_i)$.

Hence
$(\cO_X(\sum_{i \in S} m_i V_i),
\Phi,\theta)$
is isotopic
to
$(\cO_X(\sum_{i \in S} m_i V_i),
\widehat{\Phi},\theta_1)$
which in turn is isotopic to
$(\cO_X(\sum_{i \in S} m_i V_i),
\widehat{\Phi},\widehat{\theta})$
which admits a regularization along $U$.
\qed

\bigskip

\begin{lemma} \label{lemma:symplecticfibers}
	Let $(\cO_X(\sum_{i \in S} m_i V_i),\Phi,\theta)$
	be a model resolution admitting a regularization along $U$
for some relatively compact open set $U \subset X$	as in Definition \ref{defn:regularizationofmodelresolution}.
	Let $\Phi_2$ be the composition of $\Phi$ with the natural projection
	$X \times \C \twoheadrightarrow \C$.
	Define
	$$\pi_\Phi : U \lra{} \C, \quad \pi_\Phi \equiv \Phi_2 \circ s_{(m_i)_{i \in S}}|_U$$
	where $s_{(m_i)_{i \in S}}$
	is the canonical section of
	$\cO_X(\sum_i m_i V_i)$ as defined in 
	Equation (\ref{eqn:canonicalsection}).
	Then there is some $\epsilon>0$ so that
	$\pi_\Phi^{-1}(z)$ is a symplectic submanifold
	of $U$ for all $z \in \D(\epsilon) - \{0\}$.
	Also the restriction of $\theta$ to $\pi_\Phi^{-1}(\{|z| = \epsilon'\})$
	is a contact form for all $0 < \epsilon' \leq \epsilon$.
\end{lemma}
\proof
Since $U$ is relatively compact, it is sufficient for us to show that for every $x \in U \cap (\cup_{i \in S} V_i)$,
there is a small open set $U_x \subset X$ containing $x$
so that $\pi_\Phi|_{U_x \cap \pi_{\Phi}^{-1}(\C - 0)}$ has symplectic fibers and so that $\theta$ restricted to
$\pi_\Phi^{-1}(\{|z|=\epsilon'\}) \cap U_x$ is a contact form.
We will first show that the fibers are symplectic.
Suppose that $I \subset S$ is the largest set satisfying
$x \in V_I$.
Let $a_R$ be the function defined in Definition
\ref{defn:compatibletrivialization} and let
$\Pi_{(m_i)_{i \in S},I}(v)$ be as in Equation
(\ref{equation:projection}).
Near $x$, we have that $a_R(\rho_i) = \rho_i$.
Therefore
$$
\pi_\Phi(y)
=
\Phi_2(\Pi_{(m_i)_{i \in S},I}(\Psi_I^{-1}(y)))
$$
for all $y \in X$ near to $x$.
Since $\Psi_I$ is a regularization,
it is sufficient for us to show that the fibers of
$\Phi_2 \circ \Pi_{(m_i)_{i \in S},I}$
restricted to a small neighborhood of $x$ inside $\cN_X V_I$
are symplectic with respect to
\begin{equation} \label{equation of symplectic form near x}
\omega_{(\rho_{I;i},\nabla^{(I;i)})_{i \in I}} \overset{(\ref{equation symplectic form})}{=} \pi_{\cN_XV_I}^* (\omega|_{V_I}) + \frac{1}{2}  \bigoplus_{i \in I} pr_{I;i}^* \left( d(\rho_{I;i} \alpha_{I;i})\right)
\end{equation}
where
$pr_{I;i}$ is the natural projection map
from Equation (\ref{eqn:projectionmap}).
Let $W_x \subset V_I$ be a small open neighborhood of
$x$ that is contractible and choose unitary trivializations
$T_i : \cN_X V_i|_{W_x} \lra{} W_x \times \C$ for all $i \in I$.
Let $z_i : \cN_X V_I|_{W_x} \lra{} \C$ be the composition of $T_i \circ pr_{I;i}$
with the projection map to $\C$.
Hence along $W_x$, Equation (\ref{equation of symplectic form near x}) becomes:
\begin{equation} \label{eqn:sypmlecticforminwx}
\omega_{(\rho_{I;i},\nabla^{(I;i)})_{i \in I}}|_{W_x} = \pi_{\cN_XV_I}^* (\omega|_{V_I}) + \beta + \frac{i}{2}  \bigoplus_{i \in I} dz_i \wedge d\overline{z}_i
\end{equation}
where $\beta \in \Omega^2(\cN_X V_I|_{W_x})$
is a closed $2$-form whose restriction to the fibers of $\pi_{\cN_XV_I}|_{W_x}$
is zero and whose restriction to the zero section is also zero.
This means that near $x$ we have that
$\omega_{(\rho_{I;i},\nabla^{(I;i)})_{i \in I}}|_{W_x}$ is $C^0$ close to
$$\check{\omega} \equiv \pi_{\cN_XV_I}^* (\omega|_{V_I})+ \frac{i}{2}  \bigoplus_{i \in I} dz_i \wedge d\overline{z}_i.$$
We can choose our trivializations $T_i$
so that $\Phi_2  \circ \Pi_{(m_i)_{i \in S},I}$ is equal to $\prod_{i \in I} z_i^{m_i}$ inside $\pi_{\cN_XV_I}|_{W_x}$.
Since the fibers of $\prod_{i \in I} z_i^{m_i}$
are symplectic with respect to $\check{\omega}$
near $W_x \cap V_I$
and since $\omega$ is equal to $\check{\omega}$ at the point $x$,
there is a small neighborhood $\widetilde{U}_x \subset \cN_X V_I$ containing
$x$ so that the fibers of
$\prod_{i \in I} z_i|_{\widetilde{U}_x}$ are symplectic with respect to
$\omega_{(\rho_{I;i},\nabla^{(I;i)})_{i \in I}}|_{W_x}$.
Hence the fibers of $\pi_\Phi|_{U_x}$
are symplectic where $U_x \equiv \Psi_I(\widetilde{U}_x)$.

We now wish to show that $\theta$ restricted to $Y_r \equiv \pi_\Phi^{-1}(\{|z| = r\}) \cap U_x$ is a contact form for all $r>0$.
Since $\omega$ restricted to the fibers of $\pi_\Phi|_{U_x}$ are symplectic, it is sufficient for us to show that
$\theta$ restricted to the kernel of
$\omega|_{Y_r}$ is non-zero at every point for all $r > 0$.
By Equation (\ref{equation of symplectic form near x})
and the fact that
$\Phi_2 \circ \Pi_{(m_i)_{i \in S,I}}$
is equal to $\prod_{i \in I} z_i^{m_i}$
inside $\cN_X V_I|_{W_x}$,
we get that the kernel of
$\Psi_I^* \omega|_{\Psi_I^{-1}(Y_r)}$
is tangent to the fibers of
$\pi_{\cN_X V_I}$ inside $\Dom(\Psi_I)$.
Therefore,
since the restriction
of the $1$-form (\ref{eqn:1formequation})
to the fibers of $\pi_{\cN_X V_I}$ inside $\Dom(\Psi_I)$
is zero,
$\Psi_I^* \theta$ restricted to the kernel
of $\Psi_I^* \omega|_{\widetilde{U}_x}$
is non-zero so long as $\widetilde{U}_x \subset \Dom(\Psi_I)$.
Hence $\theta$ restricted to $Y_r$ is a contact form for all $r>0$ so long as $U_x \subset \Im(\Psi_I)$.
\qed

\begin{defn} \label{defn:abstractopenbookfrommodelresolution}
	Let $(\cO_X(\sum_{i \in S} m_i V_i),
	\Phi,\theta)$
	be a model resolution admitting a radius $R$ regularization
\begin{equation} \label{eqn:originalregularization}
{\mathcal R}  \equiv ((\rho_i)_{i \in I},(\Psi_I)_{I \subset S})
\end{equation}
along a relatively compact open set $U \subset X$
as in Definition \ref{defn:regularizationofmodelresolution}
where $U$ contains $K \equiv \cup_{i \in S - \star_S} V_i$.
Let $T_{r,I}$ be the radius $r$ tube of $V_I$ as in Equation
(\ref{eqn:radiusrtube}).
Let $\Phi_2 : X \lra{} \C$ be the composition
of $\Phi$ with the natural projection map
$X \times \C \lra{} \C$.
Choose $\epsilon>0$ small enough so that
\begin{equation} \label{eqn:smallenoughepsilon}
(\Phi_2 \circ s_{(m_i)_{i \in S}})^{-1}(\D_\epsilon) \cap U \subset \cup_{i \in S} T_{R,i}
\end{equation}
and so that the fibers
$(\Phi_2 \circ s_{(m_i)_{i \in S}})^{-1}(z) \cap U$
are symplectic for $z \in \D_\epsilon - 0$
by Lemma \ref{lemma:symplecticfibers}.
Let $\check{T}$ be a smoothing of the compact manifold with corners
$\cup_{i \in S - \star_S} T_{R,i}$
so that
\begin{equation} \label{eqn:horizontalboundary}
\partial \check{T} \cap T_{R,\star_S} = \Psi_{I^*}(\pi_{\cN_X V_{\star_S}}^{-1}({V_{I^*} \cap \partial \check{T}})) \cap T_{R,\star_S},
\end{equation}
$X_\theta$ points outwards along
$\partial \check{T} \cap T_{R,\star_S}$
and so that $\check{T} \subset \cup_{i \in S - \star_S} T_{R,i}$.
We also assume that the smoothing is small enough so that $\cup_{i \in S -\star_S} T_{3R/4,i}$ is contained in the interior of $\check{T}$.

Define
$$\pi_\Phi : \check{T} \lra{} \C, \quad \pi_\Phi \equiv \Phi_2 \circ s_{(m_i)_{i \in S}}|_{\check{T}}.$$

The {\it Milnor fiber} of $(\cO_X(\sum_{i \in S} m_i V_i),
\Phi,\theta)$
is the pair
$$(M,\theta_M) \equiv (\pi_\Phi^{-1}(\epsilon),\theta|_{\pi_\Phi^{-1}(\epsilon)}).$$
This is a Liouville domain for $\epsilon>0$ small enough since $X_\theta$ is tangent to $V_{\star_S} - K$ and $X_\theta$ is transverse to $\partial \check{T}$
and pointing outwards.

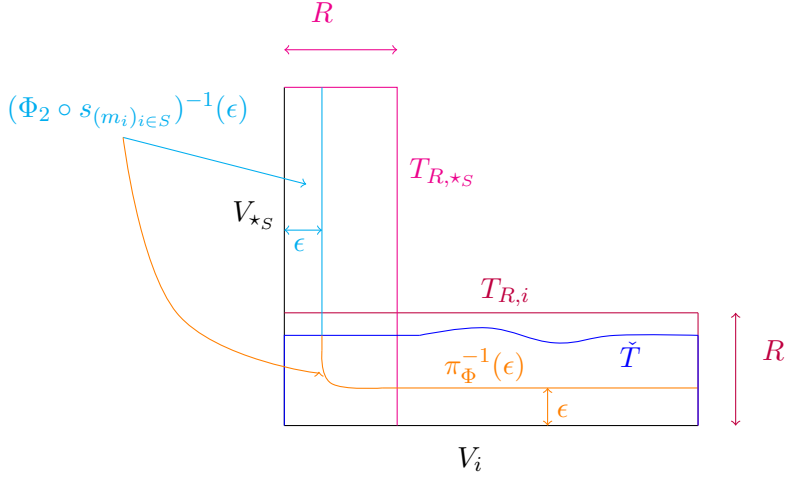
\begin{tikzpicture}

\draw (0,3) node (v4) {} -- (0,-1.5) node (v1) {};

\draw (0,-1.5) -- (5.5,-1.5) node (v3) {};

\draw [purple](5.5,0) node (v2) {} -- (0,0);

\draw [purple](5.5,0) -- (5.5,-1.5);

\draw [magenta](1.5,-1.5) -- (1.5,3) -- (0,3);

\draw [blue](0,-1.5) -- (0,-0.3) -- (1.8,-0.3) node (v8) {};

\draw [orange](5.5,-1) -- (1.3,-1) node (v6) {};

\draw [cyan](0.5,3) -- (0.5,-0.3) {};

\draw [orange](0.5,-0.3) -- (0.5,-0.5) node (v5) {};

\draw [orange] plot[smooth, tension=.7] coordinates {(v5) (0.6227,-0.9381) (v6)};

\node at (-0.4,1.3) {$V_{\star_S}$};

\node at (2.4695,-1.9437) {$V_i$};

\node at (2.9412,0.2774) {$\color{purple} T_{R,i}$};

\node at (2.111,1.8623) {$\color{magenta} T_{R,\star_S}$};

\node at (-2.0776,2.6925) {$\color{cyan} (\Phi_2 \circ s_{(m_i)_{i \in S}})^{-1}(\epsilon)$};

\draw [->, cyan](-2.1437,2.334) node (v7) {} -- (0.2903,1.7114);

\node at (2.6676,-0.7037) {$\color{orange} \pi_\Phi^{-1}(\epsilon)$};

\draw [->, orange] plot[smooth, tension=.7] coordinates {(v7) (-1.4267,-0.0055) (0.4601,-0.8169)};

\draw [blue](5.5,-0.3) node (v9) {} -- (5.5,-1.5);

\draw [blue] plot[smooth, tension=.7] coordinates {(v8) (2.7,-0.2) (3.6,-0.4) (4.4,-0.3) (v9)};

\node at (4.6,-0.55) {$\color{blue} \check{T}$};

\draw[<->,magenta] (0,3.5) -- (1.5,3.5);

\node at (0.5,4) {$\color{magenta} R$};

\draw[<->, purple] (6,0) -- (6,-1.5);

\node at (6.5,-0.5) {$\color{purple} R$};

\draw[<->,orange] (3.5,-1) -- (3.5,-1.5);

\node at (3.7,-1.3) {$\color{orange} \epsilon$};

\draw[<->,cyan] (0,1.1) -- (0.5,1.1);

\node at (0.2,0.9) {$\color{cyan} \epsilon$};

\end{tikzpicture}

Because
\begin{itemize}
	\item the $1$-form
	(\ref{eqn:1formequation})
	restricted to each fiber of $\pi_{\cN_X V_I}|_{\Psi_I^{-1}(X-K)}$ is $0$
	for every $I \subset S$,
	\item Equation (\ref{eqn:horizontalboundary}) holds
	\item $\Phi$ is radius $R$ compatible with ${\mathcal R}$
	along $U$ and
	\item  $\cup_{i \in S -\star_S} T_{3R/4,i}$ is contained in the interior of $\check{T}$,
\end{itemize}
we get that
the monodromy map $\phi : M \lra{} M$ of $\pi_\Phi$
around the loop
$$[0,1] \lra{} \partial \D(\epsilon), \quad t \lra{} \epsilon e^{2\pi it}$$
with respect to the symplectic connection
associated to $\omega$ 
exists and has compact support.
In addition,
since $\omega|_{\pi_\Phi^{-1}(\partial \D(\epsilon))} = d\theta|_{\pi_\Phi^{-1}(\partial \D(\epsilon))}$,
$\phi$ is an
exact symplectomorphism with compact support.
We call $(M,\theta_M,\phi)$
the {\it abstract contact open book associated to
	$(\cO_X(\sum_{i \in S} m_i V_i),
	\Phi,\theta)$}.

Now suppose that our model resolution has a choice of grading.
Since $\pi_\Phi^{-1}(\partial \D(\epsilon))$
is a contact submanifold of $(X - K,\omega)$ with contact form given by restricting
$\theta$ by Lemma \ref{lemma:symplecticfibers} after possibly shrinking $\epsilon$, we get an induced grading on this contact submanifold by Lemma
\ref{defn:contacthypersurfacegrading}.
Since the contact distribution is isotopic to $Q \equiv \ker(D\pi_\Phi|_{\pi_\Phi^{-1}(\partial(\epsilon))})$ through hyperplane distributions $Q_t, \ t \in [0,1]$ where $\omega|_{Q_t}$ is non-degenerate for all $t$,
we get a grading 
$$
\iota : \widetilde{Fr}(Q) \times_{\widetilde{Sp}(2n)} Sp(2n) \cong Fr(Q)
$$ on $Q$ and hence on $(M,d\theta_M)$.
Since the parallel transport maps of $\pi_\Phi$ along $\partial \D(\epsilon)$
have lifts to $\widetilde{Fr}(Q)$, $\phi$ has an induced grading
and hence $(M,\theta_M,\phi)$ is a graded abstract contact open book.
We will call this the {\it induced grading} on $(M,\theta_M,\phi)$.
\end{defn}

\begin{lemma} \label{lemma:smoothfamilyofcontactpairs}
	Let $(B_t \subset C, \xi_t, \Phi_t), t \in [0,1]$ be a smooth family of contact pairs.
	Then there is a smooth family of contactomorphisms
	$\Psi_t : C \lra{} C$ between
	$(B_t \subset C, \xi_0, \Phi_0)$
	and 
	$(B_t \subset C, \xi_t, \Phi_t)$
	(as in Definition \ref{defn:contactpair})
	for all $t \in [0,1]$
	so that $\Psi_0 = \text{id}$.
\end{lemma}
\proof
By Gray's stability theorem, there is a smooth family of contactomorphisms
$\check{\Phi}_t : C \lra{} C$ starting from the identity map
so that $\check{\Phi}_t$ is a contactomorphism between $(C,\xi_0)$ and $(C,\xi_t)$.
Therefore by pulling everything back by
$\check{\Phi}_t$,
we can assume that $\xi_t = \xi_0$ for all $t \in [0,1]$.
Also by Gray's stability theorem,
there is a smooth family of embeddings
$\iota_t : B_0 \lra{} C$
mapping $B_0$ to $B_t$ so that
\begin{itemize}
	\item $\iota_0|_{B_0} : B_0 \lra{} B_0$ is the identity map and
	\item $\iota_t|_{B_t} : B_0 \lra{} B_t$
	is a contactomorphism.
\end{itemize}
Again by Gray's stability theorem, there is a neighborhood $N$ of $B$ and a smooth family of contact embeddings
$\widetilde{\iota}_t : (N,\xi_0|_N) \lra{} (C,\xi_C)$
whose restriction to $B_0$ is $\iota_t$
and where $\widetilde{\iota}_0|_N : N \lra{} N$ is the identity map.
Let $H_t : \widetilde{\iota}_t(N) \lra{} \R$ be a smooth family of functions
{\it generating the contact isotopy $\widetilde{\iota}_t$}.
By definition this means that there is a contact form $\alpha$ compatible with
$\xi_0$ so that 
$$i_{\frac{d}{dt}\widetilde{\iota}_t(x)}\alpha = -H_t, \quad i_{\frac{d}{dt}\widetilde{\iota}_t(x)} d\alpha = dH_t - (i_R dH_t) \alpha, \quad \forall \ x \in \iota_t(N), \ t \in [0,1]$$
where $R$ is the Reeb vector field of $\alpha$ (see \cite[Lemma 3.49]{McduffSalamon:sympbook}).
Choose a smooth family of functions
$K_t, \ t \in [0,1]$ equal to $H_t$ near
$\iota_t(B_t)$.
Then $K_t$ generates a smooth family of contactomorphisms $\Psi_t$ satisfying the properties we want.
\qed

\begin{lemma} \label{lemma:farlinksofmodelresolutions}
Let $(\cO_X(\sum_{i \in S} m_i V_i),\Phi,\theta)$
be a model resolution admitting
a regularization $${\mathcal R} \equiv (\rho_i)_{i \in S}, (\Psi_I)_{I \subset S}$$ of radius $R<1$ along an open set $U \subset X$
containing $K \equiv \cup_{i \in S - \star_S} V_i$
as in Definition \ref{defn:regularizationofmodelresolution}.
Since $m_{\star_S} = 1$, we let
$$\Phi_\star \equiv \Phi|_{V_{\star_S} - K}  : \cN_X(V_{\star_S} - K) \lra{} (V_{\star_S} - K) \times \C$$
be the induced trivialization of the normal bundle
$\cN_X(V_{\star_S} - K) = \cO_X(\sum_{i \in S} m_i V_i)|_{V_{\star_S} - K}$ induced by $\Phi$.
Let $C \subset \cup_{i \in S - \star_S} T_{R,i}  - K$ be a closed hypersurface
transverse to $X_{\theta}$ and $V_{\star_S}$
and define $B \equiv C \cap V_{\star_S}$.
Let $\Phi_B$
be a trivialization
of the normal bundle of the contact submanifold $B \subset C$
induced by $\Phi_\star|_B$.

Then the contact pair
$(B \subset C,\ker(\theta|_C),\Phi_B)$
is contactomorphic to the link of the model resolution 
$(\cO_X(\sum_{i \in S} m_i V_i),\Phi,\theta)$.
If our model resolution is graded, then both of these contact pairs have
induced gradings by
Lemma
\ref{defn:contacthypersurfacegrading}
and the above contactomorphism becomes a graded contactomorphism
with respect to these gradings.
\end{lemma}
\proof
Choose $\check{R} > R$ smaller than the tube radius of our model resolution along $U$ so that $\check{R} < 1$.
Let
$\alpha : [0,\check{R}] \lra{} [0,1]$
be a smooth function so that
$\alpha' \geq 0$,
$\alpha(x) = x$ for all $x \leq R$
and $\alpha(x) = 1$ near $\check{R}$.
Define $\alpha_{\rho_i} : X - K \lra{} \R$
to be equal to $\alpha(\rho_i)$ inside $T_{\check{R},i} - K$ and $1$ otherwise.
Define
$$f :  X - K \lra{} \R, \quad f \equiv \sum_{i \in S - \star_S} \log(\alpha_{\rho_i}).$$
Then $f$ is compatible with
$(V_i)_{i \in S - \star_S}$ as in Definition \ref{defninition function compatible with}.
Let $c \ll -1$ and define
$$\check{C} \equiv f^{-1}(c), \quad \check{B} \equiv \check{C} \cap V_{\star_S}, \quad \xi_{\check{C}} \equiv \ker{(\theta)|_{\check{C}}}.$$
The normal bundle of $\check{B}$
inside $\check{C}$ has a natural trivialization
$\Phi_{\check{B}}$
induced by the trivialization
$\Phi_\star$.
Since $df(X_\theta)>0$ near
$K$ and $c \ll -1$	,
we get that
$(\check{B} \subset \check{C}, \xi_{\check{C}},\Phi_{\check{B}})$
is the link of
our model resolution
$(\cO_X(\sum_{i \in S} m_i V_i),
		\Phi,\theta)$
by Definition \ref{defn:linkofmodelresolution}.

Since $df(X_\theta) > 0$ inside
$\cup_{i \in S - \star_S} T_{R,i} - K$
and $\cup_i V_i$ is connected, we can choose a smooth family
of hypersurfaces
$(C_t)_{t \in [0,1]}$ joining $C$ and $\check{C}$
so that $C_t$ is transverse to $X_\theta$ and $V_{\star_S}$
for all $t \in [0,1]$.
Define $B_t \equiv C_t \cap V_{\star_S}$
and $\xi_t \equiv \ker(\theta|_{C_t})$.
Also let $\Phi_{B_t}$ be the trivialization
of the normal bundle of $B_t$ inside $C_t$
induced by $\Phi_\star$ so that $\Phi_{B_0} = \Phi_B$ and $\Phi_{B_1} = \Phi_{\check{B}}$.
Then $(B_t \subset C_t, \xi_t, \Phi_{B_t})$
is a smooth family of contact pairs
joining $(B \subset C,\ker(\theta|_C),\Phi_B)$
and
$(\check{B} \subset \check{C}, \xi_{\check{C}},\Phi_{\check{B}})$.
Therefore they are isomorphic by Lemma \ref{lemma:smoothfamilyofcontactpairs}.
Also if
$(\cO_X(\sum_{i \in S} m_i V_i),
		\Phi,\theta)$ is graded then
they are graded isomorphic since
all of our contact pairs have induced gradings
from our model resolution by Lemma
\ref{defn:contacthypersurfacegrading}.
\qed

\begin{lemma} \label{lemma:openbookcontactomorphism}
The link of a (graded) model resolution $(\cO_X(\sum_{i \in S} m_i V_i),
		\Phi,\theta)$
supports a (graded) contact open book which is contactomorphic to
$OBD(M,\theta_M,\phi)$ where
$(M,\theta_M,\phi)$ is the (graded) abstract contact open book
associated to this model resolution
as in Definition \ref{defn:abstractopenbookfrommodelresolution}.
\end{lemma}
\proof
In this proof we will use the same notation as in Definition \ref{defn:abstractopenbookfrommodelresolution}.
We will introduce it again here for the sake of clarity.
By Lemma \ref{lemma:modelresolutionregulatrizationexistence}
we can isotope our model resolution so that it admits
a regularization $${\mathcal R}  \equiv ((\rho_i)_{i \in I},(\Psi_I)_{I \subset S})$$
of radius $R$ along $U$
for some relatively compact open $U$
containing $K \equiv \cup_{i \in S - \star_S} V_i$.
By Lemma \ref{lemma link invariance},
the link does not change after this isotopy.
Let $T_{r,I}$ be the radius $r$ tube of $V_I$ as in Equation
(\ref{eqn:radiusrtube}).
Let $\check{T}$ be a smoothing of the manifold with corners
$\cup_{i \in S - \star_S} T_{R,i}$
as in Definition \ref{defn:abstractopenbookfrommodelresolution}.
In other words,
$\check{T}$ satisfies Equation (\ref{eqn:horizontalboundary}),
$X_\theta$ points outwards along $\partial \check{T}$
and $\check{T} \subset \cup_{i \in S - \star_S} T_{R,i}$.
Also we require that $\check{T}$ contains
$\cup_{i \in S - \star_S} T_{3R/4,i}$.
Define
$$\pi_\Phi : \check{T} \lra{} \C, \quad \pi_\Phi \equiv \Phi_2 \circ s_{(m_i)_{i \in S}}|_{\check{T}}$$
where $\Phi_2 : \cO_X(\sum_{i \in S} m_i V_i) \lra{} \C$ is the composition of $\Phi$ with the natural projection map
$X \times \C \lra{} \C$.
Then we can assume that
$$(M,\theta_M) \equiv (\pi_\Phi^{-1}(\epsilon),\theta|_{\pi_\Phi^{-1}(\epsilon)})$$
for $\epsilon> 0$ small enough 
so that Equation (\ref{eqn:smallenoughepsilon})
is satisfied.
Let $\omega$ be the symplectic form associated to our model resolution.
Here
$\phi : M \lra{} M$
is the monodromy map
around the loop 
\begin{equation} \label{equation path aroudn disk}
[0,1] \lra{} \partial \D(\epsilon), \quad s \lra{} \epsilon e^{2\pi is}
\end{equation}
with respect to the symplectic connection associated to $\omega$. Then $(M,\theta,\phi)$ is the abstract open book associated to our model resolution so long as $\epsilon>0$ is sufficiently small.

Define
$$L_r \equiv \cup_{i \in S - \star_S} T_{r,i}.$$
Let
$$\Phi_\star \equiv \Phi|_{V_{\star_S} - K}  : \cN_X(V_{\star_S} - K) \lra{} (V_{\star_S} - K) \times \C$$
be the trivialization of the normal bundle
$\cN_X(V_{\star_S} - K) = \cO_X(\sum_{i \in S} m_i V_i)|_{V_{\star_S} - K}$ induced by $\Phi$
as defined in the statement of Lemma \ref{lemma:farlinksofmodelresolutions}
and let $\Phi_{\star,2}$ be the composition of $\Phi_\star$ with the natural projection map
$V_{\star_S - K} \times \C \lra{} \C$.
Let
$$P_{\star_S} : \Im(\Psi_{\star_S}) \lra{} V_{\star_S}, \quad P_{\star_S} \equiv \pi_{\cN_X V_{\star_S}} \circ \Psi_{\star_S}^{-1}$$
be the natural projection map
and $(r,\vartheta)$ polar coordinates on $\C$.
Let $W : L_R \cap V_{\star_S} \lra{} [0,1]$
be a smooth function
equal to $0$ inside $L_{4R/5} \cap V_{\star_S}$
and equal to $1$ inside $(L_R - L_{5R/6}) \cap V_{\star_S}$
and define
$$\widetilde{W} : T_{R,\star_S}  \cap L_R \lra{} \R, \quad \widetilde{W} \equiv W \circ P_{\star_S}.$$
We now define $\theta_t \in \Omega^1(((T_{R,\star_S} \cap L_R) \cup L_{4R/5}) - K)$ for $t \in [0,1]$
to be
$\theta$ inside $L_{4R/5} - K$
and
equal to
\begin{equation} \label{equation for theta zero and theta one}
(1-t)\theta + t \left( (1- \widetilde{W})\theta +  \widetilde{W} \left( P_{\star_S}^*(\theta|_{V_{\star_S}-K}) +  \frac{1}{2} \rho_{\star_S} (\Psi_{\star_S}^{-1})^*\Phi_{\star,2}^* (d\vartheta) \right) \right)
\end{equation}
inside $T_{R,\star_S} \cap L_R - K$.
For $R_1>0$ small enough with respect to $R$,
we get that $d\theta_t$ is a symplectic form
inside $L \equiv L_{4R/5} \cup (L_R \cap T_{R_1,\star_s})$ and $d\theta_t$ restricted to $\pi_\Phi^{-1}(x) \cap L$ is a symplectic form for all $x \in \C-0$ and $t \in [0,1]$.

Let $\kappa : V_{\star_S} \lra{} \R$
be a smooth function which is negative
in the interior of $\check{T} \cap V_{\star_S}$ and positive outside $\check{T} \cap V_{\star_S}$
and so that $\kappa^{-1}(0) = \partial \check{T} \cap V_{\star_S}$ is a regular level set.
We can assume that
our perturbation $\check{T}$ from
Definition \ref{defn:abstractopenbookfrommodelresolution}
is small enough so that
$\partial \check{T} \subset L_R - L_{5R/6}$.
Choose a constant $\check{\delta}>0$
small enough so that
$\kappa^{-1}(-\check{\delta},0] \subset L_R - L_{5R/6}$
and so that
$X^{d\theta_t}_{\theta_t}$
is transverse to $\widetilde{\kappa}^{-1}(s) \cap L$
for all $s \in (-\check{\delta},\check{\delta})$.
%
Define
$$\widetilde{\kappa} : \Im(\Psi_{\star_S}) \lra{} \R, \quad \widetilde{\kappa} \equiv  \kappa \circ P_{\star_S}.$$
Define $\delta \equiv 1 - e^{-\check{\delta}}$.
Let
$$h_1, h_2 : [0,\delta) \lra{} \R$$
be smooth functions satisfying:
\begin{enumerate}
	\item $h_1'(r) < 0$, $h_2'(r) \geq 0$ for all $r>0$,
	\item $h_1(r) = 1-r^2$ and $h_2(r) = \frac{1}{2} r^2$ for $r$ near $0$,
	\item $h_1(r) = 1 - r$ and $h_2(r) =\epsilon^2$ for $r$ in $[\delta/2,\delta)$.
\end{enumerate}

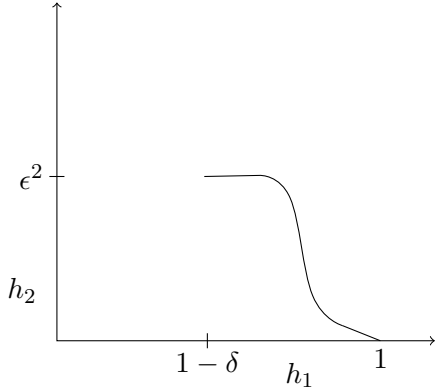
\begin{tikzpicture}

\draw[<->] (3.5238,-1) -- (-1.5,-1) -- (-1.5,3.5);

\node at (-1.9606,-0.3272) {$h_2$};
\node at (1.732,-1.4537) {$h_1$};

\node at (2.8046,-1.2387) {$1$};

\node at (0.5,-1.3) {$1-\delta$};

\draw (0.5,-1.1) -- (0.5,-0.9);

\draw (1.2,1.2) node (v1) {} -- (0.4592,1.1853);

\node at (-1.8386,1.1853) {$\epsilon^2$};
\draw (-1.6,1.1853) -- (-1.4,1.1853);

\draw  plot[smooth, tension=.7] coordinates {(v1) (1.6,0.9) (1.9,-0.4) (2.3,-0.8)};

\draw (2.3,-0.8) -- (2.8,-1);
\end{tikzpicture}

Now define
\begin{equation} \label{equation construction of C}
C \equiv (\pi_\Phi^{-1}(\partial \D_\epsilon) - \widetilde{\kappa}^{-1}((-\check{\delta},0])) \cup \bigcup_{s \in [0,\delta]} \left(\rho_{\star_S}^{-1}(h_2(s)) \cap \widetilde{\kappa}^{-1}(\log(h_1(s)))\right).
\end{equation}
This is a smooth hypersurface in $X$ since $\Phi$ is radius $R$ compatible with ${\mathcal R}$
along $U$ and since $\epsilon>0$ can be made small enough so that $\epsilon < 3R/4$.
We can also ensure that $\epsilon>0$ is small enough
so that $\pi_\Phi^{-1}(\D_\epsilon) \subset L$.
This ensures that $d\theta_t$ is a symplectic form
near $\pi_\Phi^{-1}(\D_\epsilon)$ 
and that $d\theta_t$ restricted to the fibers of $\pi_\Phi|_{\pi_\Phi^{-1}(\partial \D_\epsilon)}$ is a symplectic form for all $t$.

Define $B \equiv C \cap V_{\star_S}$.
This is also equal to $\check{T}\cap V_{\star_S} = \kappa^{-1}(0)$.
For $R$ small enough,
we have that $(C,\ker(\theta_t)|_C)$
is a smooth family of contact submanifolds
of $X$.
The trivialization $\Phi_\star$ gives us a trivialization $\Phi_{B,t}$ of the normal bundle
of the contact submanifold $B$
inside $(C,\ker(\theta_t)|_C)$ since $C$ is transverse to $V_{\star_S}$.
Hence we get a smooth family of contact pairs
$$P_t \equiv (B \subset C, \ker(\theta_t)|_C,\Phi_{B,t})$$
which are all contactomorphic by Lemma \ref{lemma:smoothfamilyofcontactpairs}.
Also by Lemma
\ref{lemma:farlinksofmodelresolutions},
the contact pair
$P_0$ is contactomorphic to the link of 
$(\cO_X(\sum_{i \in S} m_i V_i), \Phi,\theta)$
for $R$ small enough
and hence
$P_1$ is contactomorphic to the link of
$(\cO_X(\sum_{i \in S} m_i V_i), \Phi,\theta)$.
Therefore to complete this Lemma, it is sufficient
to show that the contact pair
$P_1$ is contactomorphic
to the contact pair associated to
$OBD(M,\theta_M,\phi)$.
In fact since $(M,\theta_t|_M)$
is a smooth family of Liouville domains
and since the monodromy map of
$\pi_\Phi$ around the path
(\ref{equation path aroudn disk})
with respect to the fiberwise symplectic $2$-form $d\theta_t|_{\pi_\Phi^{-1}(\partial \D(\epsilon)}$
is equal to $\phi$ for all $t$,
it is sufficient for us to show that the
contact pair $P_1$ is contactomorphic
to the contact pair associated to
$OBD(M,\theta_1|_M,\phi)$.
Note that there is a resemblance between the construction of $C$ and the construction of $OBD(M,\theta_1|_M,\phi)$
from Definition \ref{defn:openbookfromabstractone}.
We will now make this precise.

The contact pair $P_1$ can be constructed as follows:
Define $V \equiv \pi_\Phi^{-1}(\partial \D_\epsilon) - \widetilde{\kappa}^{-1}((\log(1-\frac{\delta}{2}),0])$.
Let $T_\phi =  M \times [0,1] / \sim$ be the mapping torus of $\phi$.
We have a diffeomorphism
$\Phi : T_\phi \lra{} \pi_\Phi^{-1}(\partial \D_\epsilon)$
sending $(x,s)$ to the parallel transport of $x \in M$ along
$\partial \D(\epsilon)$ in the counter clockwise direction
from $\epsilon \in \partial \D(\epsilon)$
to $\epsilon e^{is} \in \partial \D(\epsilon)$ with respect to the $2$-form
$d\theta_1|_{\pi_\Phi^{-1}(\partial \D(\epsilon))}$.
Hence we will assume that $T_\phi = \pi_\Phi^{-1}(\partial \D_\epsilon)$ under the identification $\Phi$ and that $V$ is naturally a subset of $T_\phi$.
Since $\phi$ has compact support inside $M$,
we have the standard collar neighborhood

\begin{equation} \label{eqn:neighborhoodidentification}
(1-\delta,1] \times \partial M \times (\R / \Z)
\subset T_\phi
\end{equation}
as in Definition \ref{defn:mappingtorusofphi}
(here $\delta>0$ is the same small constant defined above, which might have to be made smaller).
We can choose $\kappa$
so that $e^{\widetilde{\kappa}}|_{T_\phi}$ is the natural projection to $(1-\delta,1]$
in the neighborhood (\ref{eqn:neighborhoodidentification}).
This means that
$\theta_1$ restricted to the region
(\ref{eqn:neighborhoodidentification})
is equal to
$e^{\widetilde{\kappa}} \alpha_M + \pi \epsilon^2 dt$
where $t$ parameterizes $\R / \Z$
and where $\alpha_M = \theta_1|_{\partial M}$
by Equation (\ref{equation for theta zero and theta one}).

Using the diffeomorphism $\Phi$ and  definition (\ref{equation construction of C}) of $C$, we have that
$C$ is naturally diffeomorphic to
$$\check{C} \equiv (\partial M \times \D(\delta)) \sqcup V / \sim$$
where
$\sim$ identifies
$(x,z) \in \partial M \times (\D(\delta) - \D(\frac{\delta}{2}))$ with
$$(1-|z|,x,\frac{1}{2\pi} arg(z))
\in \left(1-\delta,1-\delta/2\right] \times \partial M \times (\R / \Z) \subset V.$$
Because $\theta_1$ restricted to
$T_{R,\star_S} \cap (L_R - L_{5R/6})$
is equal to
$P_{\star_S}^*(\theta|_{V_{\star_S}-K}) +  \frac{1}{2} \rho_{\star_S} (\Psi_{\star_S}^{-1})^*  \Phi_\star^* (d\vartheta)$
by Equation (\ref{equation for theta zero and theta one}) and because
$P^*_{\star_S}(\theta|_{V_{\star_S}-K})|_{M \cap \widetilde{\kappa}^{-1}(-\check{\delta},0]} = e^{\widetilde{\kappa}} \alpha_M$ inside the cylinder $(1-\delta,1] \times \partial M \subset M$, we have that
the contact form $\theta_1|_C$ inside $\check{C}$ under the above identification
is equal to
\begin{equation}
\alpha_1 \equiv \left\{ 
\begin{array}{ll}
h_1(r)\alpha_M + \frac{1}{2} h_2(r) d\vartheta & \text{inside} \ \partial M \times \D(\delta/2) \\
\theta_1|_{T_\phi} & \text{inside} \ V
\end{array}
 \right. .
\end{equation}
Notice that this description of
$P_1$ resembles the construction of the
open book associated to
the abstract contact open book $(M,\theta_1|_M,\phi)$
as in Definition
\ref{defn:openbookfromabstractone}.
All we need to do is deform the above construction until it is actually equal to $OBD(M,\theta_1|_M,\phi)$.

We will now do this explicitly.
From now on we let $t : V \lra{} \R / \Z$ be the coordinate
$\pi_\Phi^*(\vartheta) / 2\pi$.
Since the monodromy map $\phi$
has compact support, there is a smooth function
$F_\phi : M \lra{} \R$
so that $\phi^*(\theta_1|_M) = \theta_1|_M + dF_\phi$.
Let $\rho : [0,1] \lra{} [0,1]$ be a smooth function equal to $0$ near $0$ and $1$ near $1$.
Since $T_\phi = M \times [0,1] / \sim$
where $\sim$ identifies $(x,1)$ with $(\phi(x),0)$,
we have a well defined $1$-form
$\theta_1|_M +  d (\rho(t)F_\phi)$
on $T_\phi$.
For $s \in [0,1]$, define
$$\alpha_s \in \Omega^1(T_\phi), \quad \alpha_s \equiv (1-s)\theta_1|_{T_\phi} + s(\theta_1|_M + d(\rho(t)F_\phi)) + c_s dt$$
where $(c_s)_{s \in [0,1]}$ is a smooth family of constants
where $c_0 = 1$
and $c_t$ is sufficiently large so that
$\alpha_s$ is a contact form for all $s \in [0,1]$.
Then $(T_\phi,\alpha_1)$ is the mapping torus of $(M,\theta_1|_M,\phi)$
as in Definition \ref{defn:mappingtorusofphi}.

Choose a smooth family of functions
$$h_1^s, h_2^s : [0,\delta) \lra{} [0,\infty), \quad s \in [0,1]$$
satisfying
\begin{enumerate}
	\item $(h_1^s)'(r) < 0$, $(h_2^s)'(r) \geq 0$ for all $r>0$,
	\item $h_1^s(r) = 1-r^2$ and $h_2^s(r) = \frac{1}{2} r^2$ for $r$ near $0$,
	\item $h_1^s(r) = 1 - r$
	and $h_2^s(r) = (1-s) \epsilon^2 + c_s$ for $r$ in $[\delta/2,\delta)$,
	\item $h_1^0(r) = h_1(r)$
	and $h_2^0(r) = h_2(r)$ for all $r \in [0,\delta)$.
\end{enumerate}

Define
\begin{equation}
\alpha_1^s \equiv \left\{ 
\begin{array}{ll}
h_1^s(r) \alpha_M + \frac{1}{2} h_2^s(r) d\vartheta & \text{inside} \ \partial M \times \D(\delta/2) \\
\alpha_s & \text{inside} \ V \subset T_\phi
\end{array}
 \right.
\end{equation}
for all $s \in [0,1]$.
Then $(C, \ker(\alpha_1^s))_{s \in [0,1]}$ is a smooth family
of contact manifolds so that $B \subset C$
is a contact submanifold.
Also we have a smooth family of trivializations
$\Phi^s_1$ of the normal bundle of $B$
inside $(C, \ker(\alpha_1^s))$
so that $\Phi^0_1 = \Phi_{B,1}$.
Therefore
$$\check{P}_t \equiv (B \subset C, \ker(\alpha_1^s), \Phi_1^s)$$
is a smooth family of contact pairs
and so by Lemma \ref{lemma:smoothfamilyofcontactpairs},
they are all contactomorphic.
By construction, $\check{P}_1$ is equal to
$OBD(M,\theta_M,\phi)$.
Since $\check{P}_1$ is contactomorphic to $\check{P}_0 = P_1$
and $P_1$ is contactomorphic to $P_0$ which in turn
is contactomorphic to the link of our model resolution,
we get that
$OBD(M,\theta_M,\phi)$
is contactomorphic to the link of our model resolution.
\qed

\subsection{Dynamics of abstract contact open books associated to model resolutions.} \label{section good dynamical properties}

In this subsection we show that the fixed points of a positive slope perturbation of the symplectomorphism associated to the graded abstract contact open book associated to a model resolution form a union of specific codimension $0$ families of fixed points. We also compute the indices of these fixed points.

\begin{defn} \label{defn:discrepancyofmodelresolution}
Let $(\cO_X(\sum_{i \in S} m_i V_i),\Phi,\theta)$
be a graded model resolution with associated symplectic form $\omega$ on $X$
where $n+1 = \frac{1}{2}\text{dim}(X)$.
%
Let $J$ be an $\omega$-compatible almost complex structure on $X$.
By Definition \ref{defn:canonicalbundledefinition},
the grading on $\check{X} \equiv X - \cup_{i \in S - \star_S} V_i$
corresponds to a trivialization
$\Phi : \kappa_J|_{\check{X}} \lra{} \check{X} \times \C$
of the canonical bundle.
Let $U$ be a small neighborhood of $\cup_{i \in S - \star_S} V_i$
which deformation retracts on to $V_i$.
Choose a smooth section $s$ of $\kappa_J$
which is transverse to $0$ and so that
$\Phi \circ s|_{\check{X} - U}$ is a non-zero constant section of $(\check{X} - U) \times \C$.
By a Mayor-Vietoris argument, the homology group
$H_{2n}(\cup_{i \in S - \star_S} V_i;\Z) = H_{2n}(U;\Z)$ is freely generated by the
fundamental classes $[V_i]$ of $V_i$.
Let $[s^{-1}(0)] \in H_{2n}(U)$ be the
homology class represented by the zero set.
Then $[s^{-1}(0)] = \sum_{i \in S - \star_S} a_i [V_i]$
for unique numbers $a_i \in \Z, \ i \in S - \star_S$.
The {\it discrepancy} of $V_i$
is defined to be $a_i$ for all $i \in S - \star_S$.
\end{defn}

In the case of Example \ref{defn:modelresolutionmotivatingexample},
the discrepancy and multiplicity of $E_i$ as defined in Definition \ref{defn:discrepancyofmodelresolution}
is identical to
the discrepancy and multiplicity of $f$ along $E_i$
as in Definition \ref{defn:logresolution}.
Similarly we have a notion of multiplicity $m$ separating
resolution as in Definition \ref{defn:minimalmultiplicitydiscrepancy} for model resolutions which coincide in the case of Example \ref{defn:modelresolutionmotivatingexample}:
\begin{defn} \label{defn:multiplicitymseparatingformodelresolutions}
A model resolution $(\cO_X(\sum_{i \in S} m_i V_i),\Phi,\theta)$
is called a {\it multiplicity $m$ separating resolution}
if $m_i + m_j > m$ for all $i, j \in S$
satisfying $V_i \cap V_j \neq \emptyset$.
\end{defn}

\begin{defn} \label{defn:naturalmicovering}
Let $(\cO_X(\sum_{i \in S} m_i V_i),\Phi,\theta)$ be a model resolution.
Let $i \in S - \star_S$.
Define $V^o_i \equiv V_i - \cup_{j \in S - i} V_j$
and $X_i \equiv X - \cup_{j \in S - i} V_j$.
Let $U_i$ be an open neighborhood of $V^o_i$ inside $X_i$ which deformation retracts on to $V^o_i$ and let $\iota_i : U_i - V^o_i \lra{} U_i$ be the natural inclusion map.
Let $s_{(m_i)_{i \in S}}$ be the canonical section of $\cO_X(\sum_{i \in S} m_i V_i)$ as in Equation (\ref{eqn:canonicalsection}).
Let $\Phi_2 : \cO_X(\sum_{i \in S} m_i V_i) \lra{} \C$ is the composition of $\Phi$ with the natural projection map to $\C$.
Define
$$Q_i : U_i - V^o_i \lra{} \C^*, \quad Q_i(x) := \Phi_2 \circ s_{(m_i)_{i \in S}}(x).$$
Then the {\it natural $m_i$-fold covering of $V_i^o$}
is the $m_i$-fold covering of $V_i^o$
given by a disjoint union of covers diffeomorphic to the cover corresponding to the normal subgroup
$$G_i := (\iota_i)_*(\ker((Q_i)_*)) \subset \pi_1(U_i) = \pi_1(V^o_i)$$
and the number of such covers is $m_i$ divided by the index of $G_i$ in $\pi_1(V^o_i)$
(see Lemma \ref{lemma:indexdividesmi} below).
Such a cover does not depend on the choice of neighborhood $U_i$.
In fact it is an invariant of the model resolution up to isotopy.
\end{defn}

\begin{lemma} \label{lemma:indexdividesmi}
The index of $G_i$ divides $m_i$.
\end{lemma}

The proof of this lemma also gives us a geometric interpretation of $\widetilde{V}^o_i$.

\begin{proof}
After an isotopy, we can assume that our model resolution admits a regularization
$${\mathcal R}  \equiv ((\rho_j)_{j \in I},(\Psi_I)_{I \subset S})$$
of radius $R$ along $U$ for some relatively compact open $U$
containing $\cup_{j \in S - \star_S} V_i$.
Let $T_{r,i}$ be the radius $r$ tube of $V_i$ as in Equation 
(\ref{eqn:radiusrtube}) for some $r < 3R/4$.
We can assume that the open neighborhood $U_i$ from Definition \ref{defn:naturalmicovering} is equal to
$T_{r,i} - \cup_{j \in S - i} V_j$.
We have that our map $Q_i$
is equal to
$$Q_i : U_i - V^o_i \lra{} \C^*, \quad Q_i(x) \equiv \Phi_2 \circ s_{(m_i)_{i \in S}}(x).$$

Define $\D(\epsilon)^* \equiv \D(\epsilon) - 0$ where $\D(\epsilon) \subset \C$ is the $\epsilon$-disk.
Then $Q_i$ restricted to $Q_i^{-1}(\D(\epsilon)^*)$
for $\epsilon>0$ small enough is a fibration whose fibers are smooth manifolds with corners.
Combining this with the fact that $\pi_2(\D(\epsilon)^*) = 0$ we get that
the map
$$\pi_1(Q_i^{-1}(\epsilon)) \lra{} \ker((Q_i|_{Q_i^{-1}(\D(\epsilon)^*)})_*) = \ker((Q_i)_*)$$
is an isomorphism
by a fibration long exact sequence argument.
Therefore the natural map
$$\pi_1(Q_i^{-1}(\epsilon)) \lra{} \pi_1(U_i)$$
has image $G_i$.
Also for $0 < \epsilon \ll r \ll 1$, the map
$$P : Q_i^{-1}(\epsilon) \lra{} V_i^o, \quad P(x) \equiv \pi_{\cN_X} \circ \Psi_i^{-1}|_{Q_i^{-1}(\epsilon)}(x)$$
is covering map of order $m_i$ over $\Im(P)$ and $V_i^o$ is homotopic to the image $\Im(P)$.
Hence the index of $G_i$ divides $m_i$.
\end{proof}

\begin{theorem} \label{theorem dynamical properties of monodromy}
Let $m \in \N_{>0}$ and let $(\cO_X(\sum_{i \in S} m_i V_i),\Phi,\theta)$
be a graded model resolution
that is also a multiplicity $m$
separating resolution and define
$V_i^o \equiv V_i - \cup_{j \in S - i} V_j$ for all $i \in S$ and define
$$S_m \equiv \{ i \in S - \star_S \ : \ m_i \ \text{divides} \ m \}.$$
Let $a_i$ be the discrepancy of $V_i$ for each $i \in S - \star_S$.
Then there is a graded abstract contact open book
$(M,\theta_M,\phi)$ so that the contact pair associated to it
is graded contactomorphic to the link
of our model resolution.
Also there is small positive slope perturbation $\check{\phi}$ of $\phi^m$ so that
the fixed point set of $\check{\phi}$
is a
disjoint union of codimension $0$ families of fixed points $(B_i)_{i \in S_m}$ satisfying 
\begin{enumerate}
\item \label{item:homologyoffixedpointset}
 $H^*(B_i;\Z) = H^*(\widetilde{V}^o_i;\Z)$ where $\widetilde{V}^o_i$ is the natural $m_i$-fold covering of $V^o_i$ as in Definition \ref{defn:naturalmicovering},
\item the action of $B_i$ is equal to $-m_i w_i - \pi (m_i - m) \epsilon^2$ where $w_i$ wrapping number of $\theta$ around $V_i$ and
\item \label{item:czindexoffixedpointset}
$CZ(\check{\phi},B_i) = 2k_i (a_i + 1)$ where $k_i \equiv \frac{m}{m_i}$
\end{enumerate}
for all $i \in S - \star_S$.
\end{theorem}

\begin{proof}[Proof of Theorem \ref{theorem dynamical properties of monodromy}.]
In this proof we will use the same notation as in Definition \ref{defn:abstractopenbookfrommodelresolution}.
We will introduce it again here for the sake of clarity.
After an isotopy, we can assume that our model resolution admits
a regularization $${\mathcal R}  \equiv ((\rho_i)_{i \in I},(\Psi_I)_{I \subset S})$$
of radius $R$ along $U$
for some relatively compact open $U$
containing $K \equiv \cup_{i \in S - \star_S} V_i$ since
by Lemma \ref{lemma link invariance},
the link does not change after this isotopy.
Our abstract contact open book $(M,\theta_M,\phi)$
will be the graded abstract contact open book associated to this model resolution
as in Definition \ref{defn:abstractopenbookfrommodelresolution}.
By Lemma \ref{lemma:openbookcontactomorphism},
the link of $OBD(M,\theta_M,\phi)$
is contactomorphic to the link of our model resolution.

We now wish to show that $\phi$
satisfies properties
(\ref{item:homologyoffixedpointset})-(\ref{item:czindexoffixedpointset}) listed in the statement of this theorem. To do this, we need
to recall the construction
of $(M,\theta_M,\phi)$.
Let $T_{r,I}$ be the radius $r$ tube of $V_I$ as in Equation
(\ref{eqn:radiusrtube})
and let $T^o_{r,I}$ be the interior of $T_{r,I}$.
Let $\check{T}$ be a smoothing of the manifold with corners
$\cup_{i \in S - \star_S} T_{R,i}$
as in Definition \ref{defn:abstractopenbookfrommodelresolution}.
In other words,
$\check{T}$ satisfies Equation (\ref{eqn:horizontalboundary}),
$X_\theta$ points outwards along $\partial \check{T}$,
$\check{T} \subset \cup_{i \in S - \star_S} T_{R,i}$
and $\cup_{i \in S-\star_S} T_{3R/4,i}$ is contained in the interior of $\check{T}$.
Define
$$\pi_\Phi : \check{T} \lra{} \C, \quad \pi_\Phi \equiv \Phi_2 \circ s_{(m_i)_{i \in S}}|_{\check{T}}$$
where $\Phi_2 : \cO_X(\sum_{i \in S} m_i V_i) \lra{} \C$ is the composition of $\Phi$ with the natural projection map
$X \times \C \lra{} \C$.
Then we can assume that
$$(M,\theta_M) \equiv (\pi_\Phi^{-1}(\epsilon),\theta|_{\pi_\Phi^{-1}(\epsilon)})$$
for some small $\epsilon> 0$.
We will assume that $\epsilon>0$ is small enough 
so that $M \subset \cup_{i \in S} T_{R/4,i}$ and the fibers of $\pi_\Phi|_{\pi_\Phi^{-1}(\D(\epsilon))}$
are symplectic by Lemma \ref{lemma:symplecticfibers}.
Let $\omega$ be the associated symplectic form of our model resolution.
Define
$\phi : M \lra{} M$
to be the monodromy map
around the loop $$[0,1] \lra{} \partial D(\epsilon), \quad t \lra{} e^{2\pi it}$$
with respect to the symplectic connection associated to $\omega$.

First of all, we will compute the fixed points of the map $\phi^m$.
To do this, we will show that they correspond to certain periodic orbits of the flow of a Hamiltonian on $\check{T}$.
Define 
$$H : \Dom(\pi_\Phi) = \check{T} \to \R, \quad H(x) = |F(x)|, \quad \forall \ x \in \check{T}.$$
It is sufficient for us to find the periodic orbits
of $X_H$ starting inside $M$ which map under $\pi_\Phi$ to loops in $\C^*$ which wrap around $0$ exactly $m$ times in the anti-clockwise direction.
This is because there is a $1-1$ correspondence
between fixed points of $\phi^m$
and such orbits. This correspondence
sends a fixed point $p$ of $\phi^m$
to the unique flowline of $X_H$
starting and ending at $p$
whose image under $\pi_\Phi$ wraps around $0$ exactly $m$ times in the anti-clockwise direction.

Define $\check{T}_{R,I} \equiv T^o_{\frac{3R}{4},I} - \cup_{i \in S - I} T^o_{\frac{3R}{4},i}$ for each $I \subset S$.
Since $M \subset \cup_{I \subset S} \check{T}_{R,I}$, it is sufficient for us to calculate the fixed points of $\phi^m$ inside $M \cap \check{T}_{R,I}$ for each $I \subset S$.
Therefore we will now compute the periodic orbits of $X_H$ starting
inside $M \cap \check{T}_{R,I}$ for all $I \subset S$ which project to loops in $\C^*$ wrapping $m$ times around $0$ in an anti-clockwise direction.
Let $a_R : [0,\infty) \to [0,\infty)$
be the smooth function defined in Definition
\ref{defn:compatibletrivialization}.
In other words, $a_R$ satisfies:
\begin{enumerate}
	\item $a'_R(x) > 0$ for $x \in [0,3R/4)$,
	\item $a_R(x) = x$ for $x \leq R/4$,
	\item $a_R(x) = 1$ for $x \geq 3R/4$.
\end{enumerate}

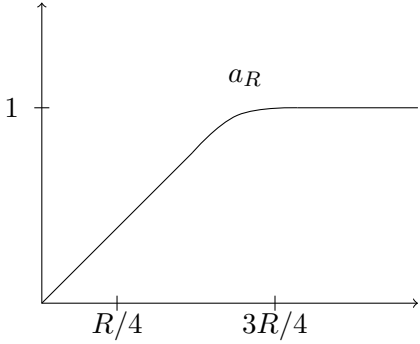
\begin{tikzpicture}

\draw [<->](-1.5,2.5) -- (-1.5,-1.5) node (v1) {} -- (3.5,-1.5);

\draw [] (-1.5,-1.5) -- (0.5,0.5) node (v2) {};

\draw [] plot[smooth, tension=.7] coordinates {(0.5,0.5)(1.1,1) (1.9,1.1)};

\draw (1.9,1.1) -- (3.5,1.1);

\draw (1.6,-1.6) -- (1.6,-1.4);
\draw (-0.5,-1.6) -- (-0.5,-1.4);

\node at (-0.5,-1.8) {$R/4$};

\node at (1.6,-1.8) {$3R/4$};

\node at (1.2,1.5) {$a_R$};

\draw (-1.4,1.1) -- (-1.6,1.1);

\node at (-1.9,1.1) {$1$};

\end{tikzpicture}

Define
$$b_R : [0,\infty) \to [0,\infty), \quad b_R(x) \equiv \sqrt{a_R(x)}.$$
Let
$$p_I : T_{R,I} \to V_I, \quad p_I(x) \equiv \pi_{\cN_X V_I}(\Psi_I^{-1}(x))$$ be the natural projection
map.
Inside $\check{T}_{R,I}$ we have that
$$H(x) = \prod_{i \in I} \left(b_R(\rho_i(x)) \right)^{m_i}, \quad \forall \ x \in \check{T}_{R,I}$$
since the bundle trivialization $\Phi$
is radius $R$ compatible with
our regularization ${\mathcal R}$
along
$T_{R,I}$.
Hence
\begin{equation} \label{equation flow of xh}
X_H|_x = \sum_{i \in I} \left( m_i b_R'(\rho_i(x))b_R(\rho_i(x))^{m_i} \prod_{j \in I - i} b_R(\rho_j(x))^{m_j}  \right) X_{\rho_i}|_x \quad \forall \ x \in \check{T}_{R,I}.
\end{equation}
This means that all the periodic orbits of $X_H$
starting inside $\check{T}_{R,I}$
are contained inside the fibers of $p_I$
since the vector fields $X_{\rho_i}$ are tangent to these fibers.
Also since $b_R(\rho_i(x)) > 0$ and
 $b'_R(\rho_i(x)) > 0$ for all
$x \in \check{T}_{R,I}$ and
all $i \in I$, we have that
any disk contained inside a fiber of $p_I$
bounding any such orbit must intersect
$V_i$ positively for all $i \in I$.
This implies that the projection
of this orbit to $\C^*$
wraps around $0$ more than $m$ times if $|I| > 1$
since our model resolution is a multiplicity $m$ separating resolution.
This means that if the set of periodic orbits
of $X_H$ starting inside $M \cap \check{T}_{R,I}$ whose image in $\C^*$ wraps $m$ times around $0$ is non-empty then $|I| = 1$.
Hence all fixed points of $\phi^m$
are contained inside $\cup_{i \in S} M \cap \check{T}_{R,i}$.
Similar reasoning ensures that $i \in S_m \cup \{\star_S\}$ and that the set of fixed points of $\phi^m$
inside $T_{R,i}$ is
$B_i \equiv M \cap \check{T}_{R,i}$ for all such $i$.

By Lemma \ref{lemma test for codmension zero families} below with $W = T_{R,i}$,
$h = \pi\rho_i|_{T_{R,i}}$,
$B_1 = h^{-1}(\pi\sqrt{\epsilon})$,
$B_2 = T_{R,i} \cap \{|\pi_\Phi| = \epsilon\}$, 
and $f_j \equiv \frac{1}{2\pi}arg(\pi_\Phi)|_{B_j}$ for $j=1,2$, we have that $B_i$ is a codimension $0$ family of fixed points of $\phi^m$ for all $i \in S_m$.
%
Since $B_i$ is homotopic to
$\pi_\Phi^{-1}(\epsilon) \cap T_{R,i}$
which in turn is homotopic to the fiber $Q_i^{-1}(\epsilon)$ constructed in the proof of Lemma \ref{lemma:indexdividesmi}
we have
$H^*(B_i;\Z) = H^*(\widetilde{V}^o_i;\Z)$ for all $i \in S -\star_S$.

%

We now need to construct a small positive slope perturbation $\check{\phi}$ of $\phi^m$ without creating any extra fixed points so that $B_{\star_S}$ disappears and so that $\phi = \check{\phi}$
near $\cup_{i \in S_m} B_i$.
Since $B_{\star_S}$
is a codimension $0$ family of fixed points of $\phi$,
there is a neighborhood $N_{\star_S}$ of $B_{\star_S}$
and a Hamiltonian $H_{\star_S} : N_{\star_S} \lra{} (-\infty,0]$
so that $\phi^m$ is the time $1$-flow of $H_{\star_S}$ inside $N_{\star_S}$
and so that $B_{\star_S} = H_{\star_S}^{-1}(0)$.
Choose $\delta_\star > 0$ small enough
so that $H_{\star_S}$ has no $q$-periodic orbits inside $H_{\star_S}^{-1}(-\delta_\star,0)$
for all $q \in [0,2]$.
Since the vector field (\ref{equation flow of xh}) is tangent to the fibers
of $p_I$ inside $T_{R,I} \cap T_{3R/4,\star_S}$ for all $I \subset S$
and since $\Psi_I$ is a regularization,
we have that $H_{\star_S}$
must be a function
of the variables
$(\rho_i)_{i \in I}$ inside $T_{R,I} \cap T_{3R/4,\star_S}$ only.
This implies that we can construct
a smooth function
$\check{b} : N_{\star_S} \lra{} \R$ for $\delta_\star>0$ small enough
so that 
\begin{itemize}
\item $\check{b} = F \circ H_{\star_S}$ for some smooth function $F : \R \lra{} \R$ 
inside $H_{\star_S}^{-1}([-\delta_{\star_S},-\delta_{\star_S}/3])$ where $F \circ H_{\star_S} = H_{\star_S}$ near $H_{\star_S}^{-1}(-\delta_{\star_S})$,
\item $\check{b}$ is $C^2$ small inside $H_{\star_S}^{-1}([-\delta_{\star_S}/3,0])$
\item $\check{b} = \delta r_M$ near $\partial M$ where $r_M$ is the radial coordinate on $M$ and
\item $\check{b}$ has no critical points.
\end{itemize}
This implies that the time $1$-flow of $\check{b}$ has no fixed points inside $N_{\star_S}$ and is equal to $H_{\star_S}$ outside a compact subset of $N_{\star_S}$.
Define $\check{\phi}$ to be equal to $\phi^m$ outside $N_{\star_S}$ and the time $1$-flow of $\check{b}$ inside $N_{\star_S}$.
This is a positive slope perturbation of $\phi^m$
so that the set of fixed points of $\check{\phi}$
is $\cup_{i \in S - \star_S} B_i$
and
$\check{\phi} = \phi^m$ in a neighborhood of these fixed points.

Next we need to compute the action of $B_i$ for each $i \in S - \star_S$. Let $p \in B_i$ and let $\gamma : \R / \Z \to \check{T}$
be the unique loop starting at $p \in M$
which is symplectically orthogonal to the fibers of
$\pi_\Phi$
and satisfying $\pi_\Phi \circ \gamma(t) = e^{2i\pi m t}$ for all $t \in \R$.
Then the action of $p$ is equal to
$-\int_{0}^{1} \gamma^* \theta + \pi m \epsilon^2 = -m_i w_i - \pi (m_i-m) \epsilon^2$.

We now need to compute the Conley-Zehnder index of $B_i$ for each $i \in S_m$.
Fix $i \in S_m$ and let $p \in B_i \subset M$.
Let $\gamma : \R / \Z \to \check{T}$
be the unique loop starting at $p \in M$
which is symplectically orthogonal to the fibers of
$\pi_\Phi$
and satisfying $\pi_\Phi \circ \gamma(t) = e^{2i\pi m t}$ for all $t \in \R$.
Let $J$ be an $\omega$-compatible almost complex structure on $X$
so that $\pi_\Phi$ becomes $J$-holomorphic.
Let
$$T^{\text{ver}} \check{T} \equiv \text{ker}(D\pi_\Phi)|_{\check{T}-K} \subset T (\check{T}-K)$$ be the vertical tangent bundle.
Let $(T^{\text{ver}} \check{T})^\perp \subset T (\check{T}-K)$
be the set of vectors which are $\omega$-orthogonal to the
vertical tangent bundle.
This is a $J$-holomorphic subbundle of $T(\check{T}-K)$.
Let $\tau_{\C^*} : T \C^* \lra{} \C^* \times \C$ be the holomorphic trivialization which sends $\frac{\partial}{\partial \vartheta}$ to the constant section $1$
and let $\tau_{\C^*,2} : T \C^* \lra{} \C$ be the composition of $\tau_{\C^*}$
with the natural projection map $\C^* \times \C \lra{} \C$.
We then have a trivialization
$$\tau^\perp : (T^{\text{ver}}\check{T})^\perp \lra{} (\check{T} - K) \times \C.$$
$$ \tau_{T^\perp}(Y) \equiv (x,\tau_{\C^*,2}(D\pi_\Phi(Y))), \ \forall \ Y \in (T^{\text{ver}}\check{T})^\perp|_x, \ x \in \check{T}-K.$$
Let $$(\tau^\perp)^* :  ((T^{\text{ver}}\check{T})^\perp)^* \lra{} (\check{T} \cap \check{X}) \times \C$$ be the corresponding trivialization of the dual bundle.

Let $\kappa_{J,\phi}$ be the canonical bundle of $T^{\text{ver}} \check{T}$.
Then we have a canonical isomorphism
\begin{equation} \label{eqn:canonicalbundleidentity}
\kappa_J|_{\check{T}-K} \cong \kappa_{J,\phi} \otimes ((T^{\text{ver}} \check{T})^\perp)^*.
\end{equation}
Since $(X-K,\omega)$ is a graded symplectic manifold, we get a natural choice of trivialization
$\tau  : \kappa_J|_{X-K} \lra{} (X-K) \times \C$ by Definition \ref{defn:gradingtrivializationcorrespondence}.
The trivializations $\tau$
and $(\tau^\perp)^*$ give us a trivialization $\tau^{\text{ver}} : \kappa_{J,\phi} \lra{} (\check{T} - K) \times \C$ of
$\kappa_{J,\phi}$
by the identity (\ref{eqn:canonicalbundleidentity}).

Let $s_\phi$ be a section of $\kappa_{J,\phi}$
so that it is equal to the constant section $1$ with respect to our trivialization
$\tau^{\text{ver}}$.
Let $s$ be a section of $\kappa_J$
so that $s^{-1}(0)$ is transverse to $0$
and contained inside a small neighborhood $N$ of $\cup_{i \in S-\star_S} V_i$
which deformation retracts on to $\cup_{i \in S-\star_S} V_i$ and so that $\tau \circ s|_{X - N}$ is the constant section $1$.
Then by definition $[s^{-1}(0)]$ is a homology class homologous
to $\sum_{i \in S} a_i [V_i]$.
Now define $\check{T}_\epsilon \equiv \pi_\Phi^{-1}(\partial \D(\epsilon))$.
By construction, $\tau^{\text{ver}}|_{\check{T}_\epsilon}$ is homotopic to the induced trivialization from Lemma \ref{defn:contacthypersurfacegrading} (after identifying the contact hyperplane distribution in $\check{T}_\epsilon$ with
$T^{\text{ver}}\check{T}$ using an isotopy between these symplectic subbundles).

We can choose $J$ so that near the image $\gamma$,
the Hamiltonian flow $\phi^{\frac{1}{2}\rho_i}_t : T_{R,i} \lra{} T_{R,i}$ of $\frac{1}{2}\rho_i|_{T_{R,i}}$ is $J$-holomorphic.
Hence on some neighborhood $N_\gamma$ of $\gamma$ invariant under the flow of $X_{\frac{1}{2}\rho_i}$,
we have that $\phi^{\frac{1}{2}\rho_i}_t$
lifts to a map $$\widetilde{\phi}_t : \kappa_J|_x \to \kappa_J|_{\phi^{\frac{1}{2}\rho_i}_t(x)}$$
for all $x \in \check{T}_\epsilon \cap N_\gamma$
given by the highest wedge power
of the $J$-holomorphic bundle map $((D\phi^{\frac{1}{2}\rho_i}_t)^{-1})^*$.
Also since $D\phi^{\frac{1}{2}\rho_i}_t(v) \in T^{\text{ver}} \check{T}_\epsilon$
for all $v \in T^{\text{ver}} \check{T}_\epsilon|_{N_\gamma \cap \check{T}_\epsilon}$,
we get an induced map
$$\widetilde{\phi}^{\text{ver}}_t : \kappa_{J,\phi}|_{\check{T}_\epsilon \cap N_\gamma}
\lra{} \kappa_{J,\phi}|_{\check{T}_\epsilon \cap N_\gamma}.$$
Let $\tau_2 : \kappa_J|_{X-K} \lra{} \C$ and
$\tau^{\text{ver}}_2 : \kappa_{J,\phi} \lra{} \C$ be the composition of $\tau$ and $\tau^{\text{ver}}$ respectively with the natural projection map to $\C$.
The winding number of the map
$$w_\phi : \R / 2\pi m\Z \lra{} \C^* \cong \text{Aut}(\C,\C), \quad w_\phi(t) = \tau^\text{ver}_2 \circ (\widetilde{\phi}^{\text{ver}}_t|_p) \circ (\tau^\text{ver}_2|_{\kappa_{J,\phi}|_p})^{-1}$$
is equal to the winding number of the map
$$w_\tau : \R / 2\pi m\Z \lra{} \C^* \cong \text{Aut}(\C,\C), \quad w_\tau(t) = \tau_2 \circ (\widetilde{\phi}_t|_p) \circ (\tau_2|_{\kappa_J|_p})^{-1}.$$
Since $[s^{-1}(0)]$ is  homologous
to $\sum_{i \in S-\star_S} a_i [V_i]$
and $\gamma(t) = \phi^{\frac{1}{2}\rho_i}_t(p)$ for all $t$,
we have that the winding number of $w_\tau$
is equal to  $$k_i (-1 - a_i).$$

Hence by
Lemma \ref{lemma alternative grading}
and the fact that
the winding number of $w_\phi$ is the winding number of $w_\tau$, we have
that
the Conley-Zehnder index of the fixed point $p$ of $\phi^m$
is
minus $2$ times the winding number of
$w_\tau$ and so
$CZ(\phi^m,B_i) = 2k_i(a_i + 1)$.
\end{proof}

\bigskip

Here is a technical lemma that was used in the proof of Theorem \ref{theorem dynamical properties of monodromy}
above.

\begin{lemma} \label{lemma test for codmension zero families}
Let $(W,\omega)$ be a symplectic manifold admitting
a free Hamiltonian $S^1$-action generated by a Hamiltonian
$h : W \lra{} \R$ (I.e. $\phi^h_1 = \text{id}_W$ and $\phi^h_t(x) \neq x$ for all $x \in W$ and $t \in (0,1)$).
Let $B_1, B_2 \subset W$ be two
real hypersurfaces inside $W$ with maps
$$f_1 : B_1 \lra{} \R/\Z, \quad f_2 : B_2 \lra{} \R / \Z$$
so that
\begin{enumerate}
\item the fibers of $f_i$ are symplectic submanifolds of $W$
and the corresponding monodromy map $\phi_i : f_i^{-1}(0) \lra{} f_i^{-1}(0)$ of $f_i$ around $\R / \Z$ is well defined (I.e. no points
parallel transport off to infinity in finite time)
 for $i = 1,2$,
\item $\phi_1 = id_{B_1}$, $B_1 = h^{-1}(C)$
and $B_2 \subset h^{-1}([C,\infty))$
for some $C>0$,
\item $B_i$ is invariant under our
Hamiltonian $S^1$-action for $i = 1,2$
and for all $t \in S^1 = \R / \Z$ and $x \in B_1$  we have
$f_1(t \cdot x) = f_1(x) + t$ and
\item $f_1^{-1}(0) \cap f_2^{-1}(0)$ is equal to a compact codimension $0$ submanifold of $f_2^{-1}(0)$ with boundary and corners
and $f_1|_{B_1 \cap B_2} = f_2|_{B_1 \cap B_2}$.
\end{enumerate}
Then $f_1^{-1}(0) \cap f_2^{-1}(0)$ is a codimension $0$ family of fixed points of $\phi_2^m$ for all $m > 0$.
\end{lemma}
\begin{proof}[Proof of Lemma \ref{lemma test for codmension zero families}.]
Let $Q \subset B_1$ be an $S^1$ invariant relatively compact
open set containing $B_1 \cap B_2$,
let $V \equiv Q \cap f_1^{-1}(0)$
and $\omega_V \equiv \omega|_V$.
For all $r_1,r_2 > 0$ let
$A_{r_1,r_2} \subset \C$
be the open annulus whose inner radius is $r_1$
and whose outer radius is $r_2$
with the standard symplectic form. 
Let $r : \C \lra{} [0,\infty)$ be the
radial function $z \lra{} |z|$
and $\theta : \C - 0 \lra{} \R / 2\pi \Z$
the angle coordinate.
Define $\check{C} \equiv \sqrt{C/\pi}$.
Let $S_{\check{C}} \subset \C$ be the circle of radius $\check{C}$.
After shrinking $Q$ slightly
we can,
by an equivariant Moser theorem (see \cite{GuilleminSternbergSymplecticTechniques}), find an $S^1$-equivariant
open set $U \subset W$
symplectomorphic to $(V \times A_{\check{C}-\delta,\check{C}+\delta}, \omega_V + \frac{1}{2} d(r^2) \wedge d\vartheta)$
so that
$Q = V \times S_{\check{C}}$
and $h|_{U} = \pi r^2$.

If we smoothly deform $f_2$ inside
$B_2$ through fibrations whose fibers are always transverse to the line
field given by $\ker(\omega|_{B_2})$
then the symplectic form on the fibers and the monodromy map does not change.
This is because such a deformation can be realized by a flow along a vector field tangent to the line field $\ker(\omega|_{B_2})$.
In particular, we can assume that
in some small $S^1$ invariant neighborhood $\check{U} \subset U$
of $B_1 \cap B_2$ that
\begin{equation} \label{eqn:thetaequation}
f_2|_{\check{U} \cap B_2} = (\theta/2\pi)|_{\check{U} \cap B_2}.
\end{equation}
Let $pr_1 : V \times A_{\check{C} - \delta,\check{C} + \delta} \lra{} V$
be the natural projection map.
Define $V_2 \equiv f_2^{-1}(0) \cap \check{U}$ and $\omega_{V_2} \equiv \omega|_{V_2}$.
We can assume that $\check{U}$ is small enough so that
$pr_1|_{V_2} : V_2 \lra{} V$ is a diffeomorphism onto its image.
This map is also a symplectic embedding
by Equation (\ref{eqn:thetaequation}).
Define
$$ H : V_2 \lra{} \R, \quad H(x) \equiv \pi r(x)^2 - 2\pi \check{C}  = h|_{V_2} - 2\pi \check{C}.$$

Since $\omega_V + \frac{1}{2} d(r^2) \wedge d\vartheta = \omega_V + dh \wedge d(\frac{1}{2\pi}\vartheta)$ inside $\check{U}$ and $pr_1|_{V_2}$
is a symplectic embedding,
we get that the vector field
$$-X_{H}^{\omega_{V_2}} + 2\pi \frac{\partial}{\partial \vartheta}$$
is tangent to $\check{U} \cap B_2$
and
lies in the kernel of $\omega|_{B_2 \cap \check{U}}$.
Then for all $m > 0$, $\phi_2^m$ is equal to the time $1$ flow of $-mH$ near $B_1\cap f_2^{-1}(0)$
inside the symplectic manifold $(V_2,\omega_{V_2})$.
Hence $B_1\cap f_2^{-1}(0)$ is a codimension $0$
family of fixed points of $\phi_2$.
\end{proof}

\section{Proof of Theorem \ref{theorem:mainspectralsequence} and Corollary \ref{corollary:spectralsequencenumbers}.} \label{section:proofofmainresult}

\begin{proof}[Proof of Theorem \ref{theorem:mainspectralsequence}.]
Let $L \equiv (L_f \subset S_\epsilon, \xi_{S_\epsilon}, \Phi_f)$
be the contact pair associated to $f$ with the standard grading as in Example \ref{example:contactpairoff}.
Let $(\cO_X(\sum_{i \in S} m_i E_i),
\Phi,\theta)$ be a graded model resolution coming from the log resolution $\pi : Y \lra{} \C^{n+1}$ as in Example \ref{defn:modelresolutionmotivatingexample}.
The wrapping number of $\theta$ around $E_j$ is $w_j$ for all $j \in S - \star_S$.
By using the function $|z|^2$ on $\C^{n+1}$
combined with Lemma \ref{lemma:smoothfamilyofcontactpairs}, one
can show that the link
of this model resolution is contactomorphic to $L$.
Hence Theorem \ref{theorem:mainspectralsequence}
follows from Theorem \ref{theorem dynamical properties of monodromy}, \ref{item:HFOBD} and \ref{item:spectralsequenceproperty}.
\end{proof}

\begin{lemma} \label{lemma:spectralsequencelemma}
Suppose we have a cohomological spectral sequence converging
to a $\Z$-graded abelian group $G^*$ with $E^1$ page $(E^1_{p,q})_{p \in \Z,q\in \Z}$. 
Define $$m \equiv \sup \{p+q \ : \ E_1^{p,q} \neq 0\} \quad \text{and}$$
$$k_p \equiv \sup \{p+q \ : q \in \Z, \ E_1^{p,q} \neq 0\} \quad \forall \ p \in \Z.$$
Suppose that $m$ is finite and $k_p \neq m-1$ for all $p \in \Z$.
Then $G^m \neq 0$ and $G^k = 0$ for all $k > m$.
\end{lemma}

\proof
Let
$$p_m \equiv \inf \{p \in \Z \ : \ \exists q \in \Z, \ p+q=m, \ E^{p,q}_1 \neq 0\}.$$
We will show that each element of $E^{p_m,m-p_m}_j$ can never kill or be killed by the spectral sequence differential for each $j \geq 1$.
Since $k_{p_m-j} \neq m-1$ for all $j$,
we get that $k_{p_m - j} < m-1$ for all $j \geq 1$.
Therefore $E^{p_m-j,m-p_m+j-1}_j = 0$ for all $j \geq 1$.
%
Hence the differential:
$$d^{p_m-j,m-p_m+j-1}_j : E^{p_m-j,m-p_m+j-1}_j \lra{} E^{p_m,m-p_m}_j$$
is zero for all $j$.
Also since
$(p_m+j) + (m-p_m-j+1) = m+1 > m$,
we get that
$E^{p_m+j,m-p_m-j+1}_j = 0$ for all $j$.
Therefore the differential:
$$d^{p_m,m-p_m}_j : E^{p_m,m-p_m}_j \lra{} E^{p_m+j,m-p_m-j+1}_j$$
is zero for all $j$ .
Hence $G^m \neq 0$.
Also $G^k = 0$ for all $k > m$
since $E^{p,q}_1 = 0$ for all $p,q \in \Z$
satisfying $p+q = k$.

%
%
%
%
%
%
%
%
%
%
%
%
%
%

\qed

\begin{proof}[Proof of Corollary \ref{corollary:spectralsequencenumbers}.]
The numbers $\mu_m$ do not depend on the choice of log resolution for all $m > 0$ by Lemmas
\ref{lemma:multiplicitymresolutiondiscrepancycalculation} and
 \ref{lemma:multipicitymresolutionindepcedence}.
Hence Corollary \ref{corollary:spectralsequencenumbers}
follows immediately from Theorem \ref{theorem:mainspectralsequence}
combined with Lemmas \ref{lemma:existenceofmultiplicitymseparating} and \ref{lemma:spectralsequencelemma}. 
\end{proof}

\section{Appendix A: Gradings and Canonical Bundles} \label{section:gradingsandcanonicalbundles}

In this section we will develop tools so that
we can construct
gradings (See Definition \ref{defn:gradedsymplectomorphism})
and relate them to other kinds of topological information.
In this paper we will only need to study gradings up to isotopy
which will be defined now.
We will first give a definition of a grading
for any principal $G$ bundle and then relate
it to gradings of $(E,\Omega)$.
Throughout this section,
 $G$ will be a Lie group, $\widetilde{G}$ its universal cover
and $p : W \lra{} B$ will be a principal $G$ bundle.
Also $\pi : E \lra{} V$,
will be a symplectic vector bundle
with symplectic form $\Omega$
whose fibers have dimension $2n$.

\begin{defn} 
A {\it grading} of $p$
consists of a principal $\widetilde{G}$ bundle $\widetilde{p} : \widetilde{W} \lra{} B$ along with a $G$ bundle isomorphism
$$\iota : \widetilde{W} \times_{\widetilde{G}} G \cong W.$$
Note that a grading of $(E,\Omega)$
is equivalent to a grading of the principal
$Sp(2n)$ bundle $Fr(E)$.
Let 
$$\iota_j : \widetilde{W}_j \times_{\widetilde{G}} G \cong W, \quad j = 0,1$$
be gradings of $p$.
An {\it isotopy}
between these two gradings consists
of a $\widetilde{G}$-bundle isomorphism
$$\Psi : \widetilde{W}_0 \lra{} \widetilde{W}_1$$
together with a smooth family of $G$ bundle isomorphisms:
$$\check{\iota}_t : \widetilde{W}_0 \times_{\widetilde{G}} G \cong W$$
joining $\iota_0$ and $\iota_1 \circ \check{\Psi}$
where $\check{\Psi} : \widetilde{W}_0 \times_{\widetilde{G}} G \lra{} \widetilde{W}_1 \times_{\widetilde{G}} G$ is the natural isomorphism induced by $\Psi$.
An {\it isotopy} between two gradings of $(E,\Omega)$
is an isotopy between the corresponding gradings
on the principal $Sp(2n)$ bundle $Fr(E)$.
We can define isotopies of gradings of symplectic manifolds
and contact manifolds in a similar way.
\end{defn}

\begin{defn}
	Let
	$$\iota : \widetilde{W} \times_{\widetilde{G}} G \cong W$$
	be a grading of $p$.
	The {\it associated covering map} of this grading
	is the natural map:
	$$\widetilde{W} \lra{} \widetilde{W} \times_{\widetilde{G}} G \stackrel{\iota}{\lra{}}
	W.$$
\end{defn}

The following lemma gives a topological characterization of
gradings.
For simplicity we will assume that the base $B$ is connected. Let $\star \in W$ be a choice of basepoint.

\begin{lemma} \label{lemma:normalsubgroupgradingbijection}
	Let $N_W$ be the set of normal subgroups
	$A \triangleleft \pi_1(W,\star)$
	so that
	$$p_* : \pi_1(W,\star) \lra{} \pi_1(B,p(\star))$$
	restricted to $N_W$ is an isomorphism.
	Let $Gr_W$ be the set of isotopy classes of gradings of $W$.
	Then the map
	$Q_W : Gr_W \lra{} N_W$
	sending a grading to
	\begin{equation} \label{eqn:fundamentalgroupimage}
	\Im(P_*) \subset \pi_1(W,\star)
	\end{equation}
	is an isomorphism
	where $P$ is the associated covering map of this grading.
\end{lemma}
\proof
We will first show that the map $Q_W$ is well defined
by showing that the image
(\ref{eqn:fundamentalgroupimage}) is a normal subgroup.
Let
$$\iota : \widetilde{W} \times_{\widetilde{G}} G \cong W$$
be a grading of $p$.
The covering map $P$ of such a grading
is isomorphic (using the map $\iota$) to the natural map
$$\widetilde{W} \lra{} \widetilde{W} \times_{\widetilde{G}} G.$$
The deck transformations of this map are equal to
$\ker(\widetilde{G} \lra{} G)$
and these act transitively.
Hence the image (\ref{eqn:fundamentalgroupimage})
is a normal subgroup. Combining this with the fact that the fibers of the natural fibration $\widetilde{W} \lra{} B$ are simply connected, the image (\ref{eqn:fundamentalgroupimage}) is contained in $N_W$ and hence the map $Q_W$ is well defined.

We will now construct an inverse to $Q_W$.
Let $N \triangleleft \pi_1(W,\star)$
be an element of $N_W$.
Let $$P : \widetilde{W} \lra{} W$$
be a covering map with a choice of basepoint $\widetilde{\star} \in \widetilde{W}$ mapping to $\star$
so that the map
$$P_* : \pi_1(\widetilde{W},\widetilde{\star}) \lra{} \pi_1(W,\star)$$
has image equal to $N$.
Let $F$
be a fiber of $p$.
Let $\iota_F : F \hookrightarrow W$ be the natural inclusion map.
Since $P_*|_{P_*^{-1}(N)}$ is injective,
we get that $(\iota_F)_*^{-1}(P_*^{-1}(N)) = \{id\}$
and hence $P|_{P^{-1}(F)} : P^{-1}(F) \lra{} F$
is the universal covering map.
This implies that each fiber has a natural
$\widetilde{G}$ action and hence
$$p \circ P : \widetilde{W} \lra{} B$$ is a $\widetilde{G}$ bundle
with a natural isomorphism
$$\widetilde{W} \times_{\widetilde{G}} W \cong W.$$
We define $Q_W^{-1}(N)$ to be the above grading.
This is an inverse to $Q_W$.
\qed

\bigskip

We have the following immediate corollary of Lemma
\ref{lemma:normalsubgroupgradingbijection}.

\begin{corollary} \label{corollay:isomorphismduetofundamentalgroups}
Suppose that $p_j : W_j \lra{} B$
is a principal $G_j$ bundle for some Lie group $G_j$
for $j = 1,2$.
Let $\Phi : W_1 \lra{} W_2$ be a map of fiber bundles
so that the induced map on the fibers is a fundamental group
isomorphism.
Then the map $\Phi_* : N_{W_1} \lra{} N_{W_2}$
induces a natural bijection between
isotopy classes of gradings
on $W_1$ and isotopy classes of gradings on $W_2$.
\end{corollary}

Here the sets $N_{W_1}$ and $N_{W_2}$ in this corollary
are defined as in Lemma \ref{lemma:normalsubgroupgradingbijection}.

\begin{lemma} \label{lemma:trivializationsandgradingsofu1bundles}
Let $\pi_K : K \lra{} B$ be a principal $U(1)$ bundle.
Then there is a natural $1-1$ correspondence between
homotopy classes of trivializations of $\pi_K$
and isotopy classes of gradings of $\pi_K$.
\end{lemma}
\proof
There is a $1-1$
correspondence between
trivializations $\Phi : K \lra{} B \times U(1)$
of $K$ up to homotopy
and sections of $\pi_K$ up to homotopy given
by the map sending
$\Phi$ to the section whose image is $\Phi^{-1}(1)$.
Hence all we need to do is construct a natural $1-1$
correspondence between sections up to homotopy
and isotopy classes of gradings.
Let $S$ be the set of sections up to homotopy
and $Gr_K$ the set of isotopy classes of gradings.
By Lemma \ref{lemma:normalsubgroupgradingbijection},
it is sufficient for us to construct a bijection between
$S$ and $N_K$.
We define
$$\Psi : S \lra{} N_K, \quad \Psi(s) = \Im(s_* : \pi_1(B,\star) \lra{} \pi_1(K,s(\star))).$$
The inverse of this map is constructed as follows:
we start with a normal subgroup $N \in N_K$.
This gives us a grading
$$\iota : \widetilde{K} \times_{\R} U(1) \cong K$$
by Lemma \ref{lemma:normalsubgroupgradingbijection}
since $\widetilde{U(1)} = \R$.
Since the fibers of $\widetilde{K}$ are contractible,
there is a smooth section $\widetilde{s} : B \lra{} \widetilde{K}$
by \cite[Theorem 9]{palaishomotopytheoryinfinite} combined with
the Steenrod approximation theorem \cite[Section  6.7, Main Theorem]{steenrodfiberbundle}.
The composition $s \equiv P_K \circ \widetilde{s}$
where $P_K : \widetilde{K} \lra{} K$ is the associated covering
map is then a smooth section of $K$.
We then define $\Psi^{-1}(N) \equiv s$.
This is the inverse of $\Psi$.
\qed

\bigskip

We will now focus on the principal $Sp(2n)$ bundle $Fr(E)$.

\begin{defn} \label{defn:canonicalbundledefinition}
Let $J$ be a complex structure on $E$ compatible with $\Omega$.
The  {\it frame bundle} $Fr(E,\Omega,J)$ of
the unitary bundle $(E,\Omega,J)$ is the principal $U(n)$-bundle whose fiber over $v \in V$ is the space of unitary bases $e_1,\cdots,e_n$ of $(E,\Omega,J)|_v$.
The {\it anti-canonical bundle} $\kappa^*_J$ of $(E,J)$
is the highest exterior power of the complex vector bundle
$(E,J)$.
The associated
$U(1)$-bundle $\kappa_{\Omega,J}^* \subset \kappa_J^*$ of $\kappa_J^*$
has a fiber at $v \in V$ equal to the subset of elements
$e_1 \wedge \cdots \wedge e_n$
where $e_1,\cdots,e_n$ is a unitary basis for $(E,\Omega,J)|_v$.
Therefore we have a natural map
	$det_{\Omega,J} : Fr(E,\Omega,J) \lra{} \kappa_{\Omega,J}$.

The {\it canonical bundle} $\kappa_J$ of $(E,J)$ is the dual of $\kappa^*_J$ (or equivalently the anti-canonical bundle of the dual bundle of $(E,J)$).
In a similar way, we can define the {\it (anti-)canonical bundle}
of a symplectic manifold with a choice of compatible almost complex structure,
or of a contact manifold $(C,\xi_C)$ with a choice of compatible
contact form $\alpha$ and a $d\alpha|_{\xi_C}$-compatible
almost complex structure $J$ on $\xi_C$.
\end{defn}

\begin{defn} \label{defn:gradingtrivializationcorrespondence}

	Let $J$ be an $\Omega$-compatible complex structure
	on $E$.
	Let $\iota_J : Fr(E,\Omega,J) \lra{} Fr(E,\Omega)$ be the natural inclusion map.
	By \cite[Proposition 2.22, 2.23]{McDuffSalamon:Jholomorphiccurves},
	the natural maps
	$det_{\Omega,J}$ and $\iota_J$
	above are bundle maps whose restriction each fiber
	is a fundamental group isomorphism.
	Hence by Corollary \ref{corollay:isomorphismduetofundamentalgroups}
	and Lemma \ref{lemma:trivializationsandgradingsofu1bundles}
	there is a natural $1-1$ correspondence
	between isotopy classes of gradings of $(E,\Omega)$
	and homotopy classes of trivializations of $\kappa_{\Omega,J}^*$.
	Combining this with the fact that there is a natural $1-1$ correspondence
	homotopy classes of trivializations of $\kappa_{\Omega,J}^*$
	and homotopy classes of trivializations of $\kappa_J^*$ and hence of trivializations of the canonical bundle $\kappa_J$,
	we get a natural $1-1$ correspondence:
\begin{equation} \label{equation:gradingtrivializationcorrespondence}
Gr : \{\text{Trivializations of} \ \kappa_J\} / \text{homotopy}
 \lra{1-1}	
 	\{\text{Gradings of} \ (E,\Omega)\} / \text{isotopy}.
\end{equation}
	
	Given a trivialization $\Phi$ of $\kappa_J$
	we will call the grading $Gr(\Phi)$
	the {\it grading associated to $\Phi$}.
	Given a grading $g$ of $(E,\Omega)$,
	we will call $Gr^{-1}(g)$
	the {\it trivialization associated to this grading}.

%
%
%
%
%
%
%
%
%
\end{defn}
%
%
%
%
%
%

\bigskip

The above discussion enables us to compute the Conley-Zehnder index of a fixed point of a graded symplectomorphism in some nice cases.
Let $\phi : M \lra{} M$ be a graded exact symplectomorphism of a Liouville domain $(M,\theta_M)$.
Let $V$ be the unique vector field on the mapping torus $T_\phi$ from Definition \ref{defn:mappingtorusofphi}
given by the lift of the vector field $\frac{d}{dt}$ on $\R / \Z$ satisfying $\iota_V d\alpha_{T_\phi} = 0$.
Let $\phi^V_t$ be the time $t$ flow of $V$.
Let $x$ be a fixed point of $\phi$.
Suppose that there is a compatible complex structure $J$
on the vertical tangent bundle
$(T^{\text{ver}} T_\phi \equiv \ker(D\pi_{T_\phi}),d\alpha_{T_\phi})$
so that $D\phi^V_t$ restricted to $T^{\text{ver}} T_\phi|_x$ is $J$-holomorphic for all $t \in [0,1]$.
Then $\phi^V_t$ lifts to a map
$$\widetilde{\phi}_t : \kappa_J|_x \lra{} \kappa_J|_{\phi^V_t(x)}, \quad \widetilde{\phi}_t(\wedge_{i=1}^n e_i^*) = \wedge_{i=1}^n(\phi_t^*)^{-1} e_i^*, \quad \forall \ e_1^*,\cdots,e_n^* \in T_x^* M.$$
Since $\phi$ is graded, we see
by Definition \ref{defn:mappingtorusofphi}
that there is a natural grading on the vertical tangent bundle.
Therefore by Definition \ref{defn:gradingtrivializationcorrespondence},
there is a natural trivialization $\Phi : \kappa_J \lra{} T_\phi \times \C$ of $\kappa_J$
associated to this grading.
Let $\Phi_2 : \kappa_J \lra{} \C$ be the composition of $\Phi$ with the natural projection map to $\C$.

\begin{lemma} \label{lemma alternative grading}
Let $x \in M = \pi_{T_\phi}^{-1}(0)$ be a fixed point of $\phi$
and suppose that $D\phi|_x : T_x M \lra{} T_x M$ is the identity map.
Then $CZ(\phi,x)$ is equal to minus $2$ times the winding number of the map:
$$w : \R / \Z \lra{} \C^* = \text{Aut}(\C,\C), \quad t \lra{} \Phi_2 \circ \widetilde{\phi}_t \circ (\Phi_2|_{\kappa_J|_x})^{-1}.$$
\end{lemma}
\proof
Let $\gamma : \R / m\Z \lra{} T_\phi$
be the $m$-periodic orbit of $V$ whose initial point is $x$.
Then there is a unique (up to homotopy) unitary trivialization
$T$ of $\gamma^* T^{\text{ver}} T_\phi$ 
so that $\Phi$ is equal to the highest wedge
power of $T$.
Because of the correspondence (\ref{equation:gradingtrivializationcorrespondence}),
we can also ensure that $T$ maps the grading on
$\gamma^* T^{\text{ver}} T_\phi$
(given by pulling back the grading on $T^{\text{ver}} T_\phi$ via $\gamma$)
to the trivial grading
on $(\R/m\Z) \times \C^{n+1}$ (maybe after changing the grading to an isotopic one).

Under this trivialization, the flow of $V$
corresponds to a smooth family of Hermitian matrices
$(A_t)_{t \in [0,1]}$
and the degree of $w$ is $-1$ times the winding number of $t \lra{} \text{det}_\C(A_t)$.
Using the correspondence (\ref{equation:gradingtrivializationcorrespondence}) and the trivialization $T$,
such a family of matrices corresponds to a point
in the universal cover
$\widetilde{Fr}(TM)|_x$ of $Fr(TM)|_x$.
Hence the Conley-Zehnder index of $(A_t)_{t \in [0,1]}$
is equal to $CZ(\phi,x)$.
Since $A_t$ are unitary matrices,
we get that $CZ((A_t)_{t \in [0,1]})$ is equal to twice the 
winding number of $t \lra{} \text{det}_\C(A_t)$.
Hence $CZ(\phi,x)$ is $-2$ times the winding number of $w$.
\qed

\section{Appendix B: Contactomorphisms of Mapping Tori and Floer Cohomology} \label{section:appendix}

The aim of this section is to show that Property
\ref{item:HFOBD} holds. Here is a statement of this property:

{\it Suppose that $(M_1,\theta_{M_1},\phi_1)$,
$(M_2,\theta_{M_2},\phi_2)$ are graded abstract
contact open books
so that
the graded contact pairs associated to them are graded contactomorphic.
Then $HF^*(\phi_0,+) = HF^*(\phi_1,+)$.
}

We will prove this by using an intermediate Floer cohomology group called $S^1$-equivariant Hamiltonian Floer cohomology on a certain mapping cylinder of our symplectomorphism.  

\begin{defn} \label{defn:mappingcylinder}
Let $(M,\theta_M,\phi)$ be an abstract contact open book.
Let $\check{\phi}$ be a small positive slope perturbation of $\phi$.
The {\it mapping cylinder of $\check{\phi}$}
is a triple $(W_{\check{\phi}},\pi_{\check{\phi}},\theta_{\check{\phi}})$
where
\begin{enumerate}
\item $W_{\check{\phi}} \equiv (\R \times \R \times M) / \Z$
where the $\Z$ action on $(\R \times \R \times M)$ has the property that $1 \in \Z$
sends $(s,t,x)$ to $(s,t-1,\check{\phi}(x))$,
\item $\pi_{\check{\phi}} : W_{\check{\phi}} \lra{} \R \times (\R/\Z)$ sends $(s,t,x)$ to $(s,t) \in \R \times \R/\Z$ and
\item $\theta_{\check{\phi}} = sdt + \kappa \theta_M + \kappa d(\rho(t)F_{\check{\phi}})$
where \begin{itemize}
\item $F_{\check{\phi}} : M \lra{} \R$ is a smooth function with support in the interior of $M$ satisfying $(\check{\phi})^* \theta_M = \theta_M + dF_{\check{\phi}}$,
\item $\rho: [0,1] \lra{} [0,1]$ is a  a smooth function equal to $0$ near $0$ and $1$ near $1$
and
\item $\kappa>0$ is a constant small enough to ensure that $d\theta_{\check{\phi}}$ is symplectic.
\end{itemize}  
%
\end{enumerate}

Let $r_M : (0,1] \times \partial M \lra{} (0,1]$
be the cylindrical coordinate on $M$.
Let $\delta_{\check{\phi}} > 0$ be small enough so that the symplectomorphism $\check{\phi}$
is equal to the time $1$ flow of $\delta r_M$
inside $(1-\delta_{\check{\phi}},1] \times \partial M$ for some $\delta>0$.
Let $\phi^{\delta r_M}_t : (1-\delta_{\check{\phi}},1] \times \partial M \lra{} (1-\delta_{\check{\phi}},1] \times \partial M$
be the time $t$ flow of $\delta r_M$.
Then we have a natural embedding:
$$\iota_{\check{\phi}} : C_{\check{\phi}} \equiv (\R \times (\R/\Z) \times (1-\delta_{\check{\phi}},1] \times \partial M) \  \hookrightarrow W_{\check{\phi}}, \quad \iota_{\check{\phi}}(s,t,r_M,y) \equiv (s,t,(\phi^{\delta r_M}_{-t}(r_M,y))$$
called the {\it vertical cylindrical end of $W_{\check{\phi}}$}.
The coordinate
\begin{equation} \label{eqn:verticalcylindralcoordinate}
r_{\check{\phi}} : C_{\check{\phi}} \lra{} (1-\delta_{\check{\phi}},1], \quad r_{\check{\phi}}(s,t,x) \equiv r_M(x)
\end{equation}
is called the {\it vertical cylindrical coordinate}.
A {\it grading} on a mapping cylinder $(W_{\check{\phi}},\pi_{\check{\phi}},\theta_{\check{\phi}})$
is a grading on the symplectic manifold
$(W_{\check{\phi}},d\theta_{\check{\phi}})$.
%
%

Two mapping cylinders 
$(W_{\check{\phi}_1},\pi_{\check{\phi}_1},\theta_{\check{\phi}_1})$,
$(W_{\check{\phi}_2},\pi_{\check{\phi}_2},\theta_{\check{\phi}_2})$
are {\it isomorphic}
if there is a diffeomorphism $\Phi : W_{\check{\phi}_1} \lra{} W_{\check{\phi}_2}$
and a constant $\delta > 0$ so that
\begin{itemize}
\item $\Phi^* \theta_{\check{\phi}_2} = \theta_{\check{\phi}_1} + dq$ for some $q : W_{\check{\phi}_1} \lra{} \R$ where $q$ has support in $W_{\check{\phi}_1} - \{r_{\check{\phi}_1} \geq 1-\delta\}$ and
\item $\pi_{\check{\phi}_1}|_{\{r_{\check{\phi}_1} \geq 1-\delta\}} = (\pi_{\check{\phi}_2} \circ \Phi)|_{\{r_{\check{\phi}_1} \geq 1-\delta\}}$.
\end{itemize}
They are {\it graded isomorphic} if, in addition, $\Phi$ is a graded symplectomorphism from $(W_{\check{\phi}_1},d\theta_{\check{\phi}_1})$
to
$(W_{\check{\phi}_2},d\theta_{\check{\phi}_2})$.

\end{defn}

Note that the definition above has many similarities with the definition of
the mapping torus from Definition \ref{defn:mappingtorusofphi}.
Also if we define the mapping torus
$$\pi_{T_{\check{\phi}}} : T_{\check{\phi}} \lra{} \R/\Z, \ \alpha_{T_{\check{\phi}}}$$
of our positive slope perturbation $\check{\phi}$
in exactly the same way as in Definition \ref{defn:mappingtorusofphi},
then $W_{\check{\phi}} = \R \times T_{\check{\phi}}$, $\theta_{\check{\phi}} = (s - \kappa) dt + \frac{\kappa}{C} \alpha_{\check{\phi}}$ for some $C>0$ and $\pi_{\check{\phi}} = \text{id}_\R \times \pi_{T_{\check{\phi}}}$.
The following calculation will be useful later on.
If we have a Hamiltonian $H$ equal to $\pi_{\check{\phi}}^* K$ for some $K : \R \times (\R/\Z) \lra{} \R$ then $X_H$ is equal to the horizontal lift of $X^{ds \wedge dt}_K$ with respect to the symplectic connection associated to $d\theta_{\check{\phi}}$.


\begin{lemma} \label{lemma:obdisomorphism}
Let $$(B_1 \subset C_1, \xi_1, \tau_1), \quad (B_2 \subset C_2, \xi_2, \tau_2)$$
be the (graded) contact pairs associated to the (graded) abstract contact open books
$(M_1,\theta_{M_1},\phi_1)$, $(M_2,\theta_{M_2},\phi_2)$ respectively.
%
Let $\check{\phi}_1$ be a small positive slope perturbation of $\phi_1$.
If the above contact pairs are (graded) contactomorphic then $(M_2,\theta_{M_2},\phi_2)$
is (graded) isotopic to an abstract contact open book
$(M_3,\theta_{M_3},\phi_{M_3})$
so that
the mapping cylinders
$$(W_{\check{\phi}_1},\pi_{\check{\phi}_1},\theta_{\check{\phi}_1}), \quad (W_{\check{\phi}_3},\pi_{\check{\phi}_3},\theta_{\check{\phi}_3})$$
are (graded) isomorphic where $\check{\phi}_3$ is a small positive slope perturbation of $\phi_3$.
\end{lemma}
\proof
Since the corresponding open books are  contactomorphic, and the boundary
$\partial M_j$ is contactomorphic to the binding of $OBD(M_j,\theta_{M_j},\phi_j)$
for $j=1,2$,
we get that $\partial M_1$ is contactomorphic to $\partial M_2$.
Hence there is a diffeomorphism
$\Psi : \partial M_2 \lra{} \partial M_1$
so that $\Psi^* \alpha_1 = f \alpha_2$
where $\alpha_j = \theta_{M_j}|_{\partial M_j}$ for $j=1,2$ and $f : \partial M_2 \lra{} (0,\infty)$.
After multiplying $\theta_2$ by a positive constant, we can assume that $f > 1$.
Choose $\delta>0$ small enough so that the subset of the cylindrical end
$(1-\delta,1] \times \partial M_2 \subset M_2$
is disjoint from the support of $\phi_2$
and let $r_{M_2} : (1-\delta,1] \times \partial M_2 \lra{} (1-\delta,1]$ be the associated cylindrical coordinate.
Let $\rho : (1-\delta,1] \lra{} [0,1]$
be a smooth function with non-negative derivative so that it
is equal to $0$ near $1-\delta$ and $1$ near $1$.
Let $F_t : M_2 \lra{} \R$, $t \in [0,1]$
be a smooth family of functions equal to
$1 + t \rho(r_{M_2}) (f-1)$ inside $(1-\delta] \times \partial M_2$ and $1$ otherwise.
Then
$(M_2,F_t\theta_{M_2},\phi_2)$
is an isotopy of abstract contact open books.
Define $(M_3,\theta_{M_3},\phi_3) \equiv (M_2,F_1\theta_{M_2},\phi_2)$.
Hence there is a diffeomorphism $\check{\Psi} : \partial M_3 \lra{} \partial M_1$ so that $(\check{\Psi})^* \alpha_{M_1} = \alpha_{M_3}$
where $\alpha_{M_3} \equiv \theta_{M_3}|_{\partial M_3}$.

Choose a small positive slope perturbation $\check{\phi}_3$ of $\phi_3$
so that $\check{\phi}_3|_{(1-\check{\delta},1] \times \partial M_3}$ is equal to $(\text{id}_{(1-\check{\delta},1]} \times \check{\Psi})^* \check{\phi}_1|_{(1-\check{\delta},1] \times \partial M_1}$ for some $\check{\delta} > 0$ smaller than $\delta$.
Let $(T_{\phi_j},\pi_{T_{\phi_j}},\alpha_{\phi_j})$ be the mapping torus of $\phi_j$ and let
$(T_{\check{\phi}_j},\pi_{T_{\check{\phi}_j}},\alpha_{\check{\phi}_j})$ be the mapping torus of $\check{\phi}_j$ for $j=1,3$.
Let $(B_3 \subset C_3, \xi_3, \tau_3)$
be the contact pair associated to
$(M_3,\theta_{M_3},\phi_3)$.
By lemma \ref{lemma:smoothfamilyofcontactpairs},
we get that the contact pair associated to
$(M_3,\theta_{M_3},\phi_3)$
is contactomorphic to the contact pair associated to
$(M_2,\theta_{M_3},\phi_2)$ which in turn is contactomorphic to the contact pair associated to
$(M_1,\theta_{M_1},\phi_1)$.
There is a contactomorphism
$Q : T_{\phi_3} \lra{} T_{\phi_1}$
so that $Q|_{B_3} : B_3 \lra{} B_1$ is equal to $\check{\Psi}$ under the identification $B_i = M_i$ for $i = 1,3$ and so that $\pi_{T_{\phi_3}} = \pi_{T_{\phi_1}} \circ Q$
near $\partial T_{\phi_3}$.
Hence we can find a contactomorphism
$\check{Q} : T_{\check{\phi}_3} \lra{} T_{\check{\phi}_1}$
satisfying  $(\check{Q})^* \alpha_{\check{\phi}_1} = \alpha_{\check{\phi}_3}$ near $\partial T_{\check{\phi}_3}$ and so that $\pi_{T_{\check{\phi}_3}} = \pi_{T_{\check{\phi}_1}} \circ \check{Q}$
near $\partial T_{\check{\phi}_3}$.

Since $W_{\check{\phi}_j}$ is naturally diffeomorphic to $\R \times T_{\check{\phi}_j}$ for $j=1,3$ we can define
$$W : W_{\check{\phi}_3} \lra{} W_{\check{\phi}_1}, \quad W \equiv (\text{id}_\R,\check{Q}).$$
Now $W^* \theta_{\check{\phi}_3} = \theta_{\check{\phi}_1}$ outside a subset $K \subset W_{\check{\phi}_3}$ whose intersection with each fiber of $\pi_{\check{\phi}_3}$ is compact.
Since $\theta_{\check{\phi}_3}$ and $W^* \theta_{\check{\phi}_1}$ scale at most linearly in $C^1$ norm as we translate in the $s$ coordinate direction,
we can use a Moser argument applied to
$\tau \theta_{\check{\phi}_3} + (1-\tau) W^* \theta_{\check{\phi}_1}, \tau \in [0,1]$ giving us our isomorphism.
\qed

\bigskip

\begin{lemma} \label{lemma:isotopicimpliesisomorphicHF}
	Suppose that $(M_1,\theta_{M_1},\phi_1)$, $(M_2,\theta_{M_2},\phi_2)$ are isotopic abstract contact open books.
	Then $HF^*(\phi_1,+) = HF^*(\phi_2,+)$.
\end{lemma}
\proof
Since the above abstract contact open books are isotopic,
we can assume that $M_1 = M_2$ and that there is a smooth family of Liouville forms $(\theta_s)_{s \in [0,1]}$ so that $\theta_0 = \theta_{M_1}$ and $\theta_1 = \theta_{M_2}$.
Also there is a smooth family of exact symplectomorphisms $\psi_s : M_1 \lra{} M_1, \ s \in [0,1]$
with respect to $\theta_s$ with support in a fixed compact set joining $\phi_1$ and $\phi_2$.
Let $r_s$ be the cylindrical coordinate for $(M_1,\theta_s)$
and choose $\delta>0$ small enough so that
$\{r_s \geq 1-\delta\}$ is disjoint from the support of $\psi_s$ for all $s \in [0,1]$.
By pulling back $r_s$ and $\theta_s$ and $\psi_s$ by a smooth family of diffeomorphisms starting at the identity and parameterized by $s \in [0,1]$,
we can assume that $r_s = r_0$
inside the region $\{r_0 \geq 1-\delta\} = (1-\delta,1] \times \partial M_1$.
By Gray's stability theorem, we can also assume
that $\theta_s = r_0 f_s \alpha$
inside $\{r_0 \geq 1-\delta\} = (1-\delta,1] \times \partial M_1$
for some contact form $\alpha$ on $\partial M_1$ and some smooth family of functions $f_s : \partial M_1 \lra{} (0,\infty), \ s \in [0,1]$. 

Now choose a smooth family of functions
$g_s : (1-\delta,1] \times \partial M_1 \lra{} (0,\infty), \ s \in [0,1]$
so that $\frac{\partial g_s}{\partial r_0} > 0$, $g_s$ is equal to $f_s$ inside ${(1-\delta,1-\delta/2] \times \partial M_1}$,
and $g_s = Cf_0$ inside $(1-\delta/4,1] \times \partial M_1$ for some large constant $C>0$ and for all $t \in [0,1]$.
Define $\check{\theta}_s$ to be equal to $\theta_s$ outside $(1-\delta,1] \times \partial M_1$ and equal to $r_0 g_s \alpha$ inside this region. Then $(M_1,\check{\theta}_s), \ s \in [0,1]$ is a smooth family of Liouville domains so that $\check{\theta}_s = \theta_s$ inside $(1-\delta,1-\delta/2] \times \partial M_1$ and $\check{\theta}_s$ is independent of $s$ near $\partial M_1$.

Let $(K_{s,t})_{(s,t) \in [0,1]^2}$ be a smooth family of almost complex structures compatible with $d\theta_s$ so that $K_{s,t}$ is cylindrical inside $(1-\delta] \times M_1$ with respect to $\theta_s$ for all $(s,t) \in [0,1]^2$.
Choose a smooth family of almost complex structures $(J_{s,t})_{(s,t) \in [0,1]^2}$ compatible with $d\check{\theta}_s$
equal to $K_{s,t}$ outside $(1-\delta/2,1] \times \partial M_1$
and equal to $K_{0,t}$ inside $(1-\delta/4,1] \times \partial M_1$.
Let $\check{\psi}_s$ be a smooth family of exact symplectomorphisms with respect to $\check{\theta}_s$ which are small positive slope perturbations of $\psi_s$ 
so that $\check{\psi}_s$ is the time $1$ flow of $\eta r_s$ inside 
$(1-\delta,1] \times \partial M_1$ for some very small $\eta>0$ (so that there are no fixed points of this symplectomorphism in this region).
Let $\widehat{\psi}_s$ be a smooth family of positive slope perturbations of $\psi_s$ with respect to $\theta_s$ which are equal to the time $1$ flow of $\eta r_s$ inside $(1-\delta,1] \times \partial M_1$ with respect to the symplectic structure $d\theta_s$.
We assume that $\eta$ is small enough so that  $\widehat{\psi}_s$ has no fixed points inside $(1-\delta,1] \times \partial M_1$ for all $s \in [0,1]$.

Since \begin{itemize}
\item  $\check{\psi}_s,\widehat{\psi}_s$ is the time $1$ flow of a linear Hamiltonian inside $(1-\delta,1-\delta/2) \times \partial M_1$,
\item $J_{s,t}$ and $K_{s,t}$ are cylindrical inside this region and 
\item $(\check{\psi}_s,(J_{s,t})_{t \in [0,1]})$,
$(\widehat{\psi}_s,(K_{s,t})_{t \in [0,1]})$ are equal outside $(1-\delta,1] \times \partial M_1$ and
\item $\check{\psi}_s$ and $\widehat{\psi}_s$ has no fixed points inside $(1-\delta,1] \times \partial M_1$
\end{itemize} 
 then a maximum principle (\cite[Lemma 7.2]{SeidelAbouzaid:viterbo})
tells us that
\begin{equation} \label{eqn:equalityofHFs}
HF^*(\check{\psi}_s,(J_{s,t})_{t \in [0,1]}) = HF^*(\widehat{\psi}_s,(K_t)_{t \in [0,1]})
\end{equation}
for all $s \in [0,1]$.

Since $\check{\theta}_s$ is independent of $s$ near $\partial M_1$, we can assume, by a Moser argument, that
$\check{\theta}_s = \theta_0 + \beta_s$ for some smooth family of compactly supported closed $1$-forms $\beta_s, \ s \in [0,1]$.
Then by \cite[Theorem 2.34]{uljarevic:symplectichomologysymplectomorphisms},
we get that
$HF^*(\check{\psi}_s,(J_{s,t})_{t \in [0,1]})$
is independent of $s \in [0,1]$.
Hence
$HF^*(\widehat{\psi}_s,(K_s)_{t \in [0,1]})$
is independent of $s$ by Equation (\ref{eqn:equalityofHFs}) which implies that
$HF^*(\phi_1,+) = HF^*(\phi_2,+)$.
\qed

\begin{defn}
Let $(W_{\check{\phi}},\pi_{\check{\phi}},\theta_{\check{\phi}})$
be a mapping cylinder.
An almost complex structure $J$ on $W_{\check{\phi}}$
is {\it strictly compatible with}
$(W_{\check{\phi}},\pi_{\check{\phi}},\theta_{\check{\phi}})$
if
\begin{enumerate}
\item $J$ is compatible with $d\theta_{\check{\phi}}$,
\item $\pi_{\check{\phi}} : W_{\check{\phi}} \lra{} \R \times \R/\Z$
is $(J,j)$-holomorphic
(I.e. $D\pi_{\check{\phi}} \circ J = j \circ D\pi_{\check{\phi}}$)
where $j$ is the complex structure sending $\frac{\partial}{\partial s}$ to $\frac{\partial}{\partial t}$ where $(s,t)$ parameterizes $\R \times \R/\Z$,
\item the restriction of $J$ to the cylindrical end $C_{\check{\phi}}$
is a product $j \oplus J_M$ where $J_M$ is a fixed cylindrical almost complex structure
inside the cylindrical end $(1-\delta_{\check{\phi}},1] \times \partial M$ and 
\item $J$ is invariant under translations in the $s$ coordinate.
\end{enumerate}
We will call $J_M$ the {\it associated cylindrical almost complex structure on $M$}.
An almost complex structure $J$ on $W_{\check{\phi}}$ is {\it compatible with
$(W_{\check{\phi}},\pi_{\check{\phi}},\theta_{\check{\phi}})$} if there is an almost complex structure $\check{J}$ compatible with
 $(W_{\check{\phi}},\pi_{\check{\phi}},\theta_{\check{\phi}})$
 and a compact subset $K$ in the interior of $W_{\check{\phi}}$ so that $J|_{W_{\check{\phi}}-K} = \check{J}|_{W_{\check{\phi}}-K}$.
\end{defn}

\begin{defn}
	Recall that a $1$-periodic orbit of a time dependent Hamiltonian $H_t : W_{\check{\phi}} \lra{} \R$ is a smooth map
	$\gamma : \R/\Z \lra{} M$
	satisfying $\frac{d}{dt}\gamma = X^{d\theta_{\check{\phi}}}_{H_t}$.
	Since there is a natural $1$-$1$ correspondence between $1$-periodic orbits and fixed points of the corresponding Hamiltonian symplectomorphism $\phi^{H_t}_1$,
	we will call $\gamma$ the {\it $1$-periodic orbit associated to the fixed point $\gamma(0)$}.
	A Hamiltonian is {\it autonomous} if it does not depend on time.
	An {\it $S^1$-family of $1$-periodic orbits} of $H$ is a family $(\gamma_t)_{t \in \R/\Z}$ of $1$-periodic orbits where $\gamma_t(\check{t}) = \gamma_0(\check{t}+t)$
	for all $t,\check{t} \in \R/\Z$.
	
	Let $(H,J)$, $(\check{H},\check{J})$ be pairs consisting of autonomous Hamiltonians $H$, $\check{H}$ and almost complex structures $J$, $\check{J}$.
	A smooth family of pair $(H_s,J_s)_{s \in \R}$ {\it joins $(H,J)$ and $(\check{H},\check{J})$}
	if $(H_s,J_s) = (H,J)$ for all sufficiently negative $s$
	and $(H_s,J_s) = (\check{H},\check{J})$ for all sufficiently positive $s$.

Let $h : \R \lra{} \R$ be a smooth function which is bounded from below
with positive derivative satisfying $h''(s) = 0$ whenever $s$ is sufficiently positive and $h'(s) < 1$ for $s$ sufficiently negative.
 The value of $h'(s)$ for large enough $s$ is called the {\it slope of $h$}.
A Hamiltonian is {\it strictly compatible with $(W_{\check{\phi}},\pi_{\check{\phi}},\theta_{\check{\phi}})$}
if it is equal to $\pi_{\check{\phi}}^* h(s)$ everywhere.
A Hamiltonian is {\it compatible with $(W_{\check{\phi}},\pi_{\check{\phi}},\theta_{\check{\phi}})$}
if it is equal to $\pi_{\check{\phi}}^* h(s)$
outside a compact subset of the interior of $W_{\check{\phi}}$.
The slope of such a Hamiltonian is defined to be the slope of $h$.

All the $1$-periodic orbits of $h(s)$
 on the symplectic manifold
 $(\R \times \R/\Z, ds \wedge dt)$
 wrapping around $\R \times \R/\Z$ once come in $S^1$ families in the region $h'(s) = 1$
and for the unique $s$ satisfying $h'(s) = 1$
we have $1$-periodic orbits
$$\gamma_{s,q} : \R/\Z \lra{} \R \times \R/\Z, \quad \gamma_q(t) = (s,t+q)$$
for all $q \in [0,1)$.
Also the $1$-periodic orbits of $\pi_{\check{\phi}}^* h(s)$ project to $1$-periodic orbits of $h(s)$.

A pair $(H,J)$ is {\it (strictly) compatible with  $(W_{\check{\phi}},\pi_{\check{\phi}},\theta_{\check{\phi}})$}
if $H$ is a Hamiltonian (strictly) compatible with
 $(W_{\check{\phi}},\pi_{\check{\phi}},\theta_{\check{\phi}})$
and $J$ is an almost complex structure (strictly) compatible with
 $(W_{\check{\phi}},\pi_{\check{\phi}},\theta_{\check{\phi}})$.
A smooth family $(H_s,J_s)_{s \in \R}$
of pairs compatible with
$(W_{\check{\phi}},\pi_{\check{\phi}},\theta_{\check{\phi}})$
has {\it non-increasing slope}
if the slope of $H_s$
is greater than or equal to the slope of
$H_{\check{s}}$ for all $s \leq \check{s}$.

\end{defn}

\begin{defn} \label{defn:maximumprincipal}
Let $(H_{\sigma},J_{\sigma})_{\sigma \in \R}$ be a smooth family of pairs of Hamiltonians and almost complex structures compatible with
a mapping cylinder $(W_{\check{\phi}},\pi_{\check{\phi}},\theta_{\check{\phi}})$.
An open subset $V \subset W_{\check{\phi}}$
satisfies the {\it maximum principal with respect to $(H_{\sigma},J_{\sigma})_{\sigma \in \R}$}
if for every compact codimension $0$ submanifold
$\Sigma \subset \R \times \R/\Z$ and every smooth map
$u : \Sigma \lra{} W_{\check{\phi}}$
satisfying
\begin{enumerate}
\item $u(\partial \Sigma) \subset V$
\item  \label{eqn:floerequation}
$\partial_\sigma u(\sigma,\tau) + J_\sigma \partial_\tau u(\sigma,\tau) = J_\sigma X_{H_\sigma}$,
\end{enumerate}
also satisfies $\Im(u) \subset V$.

We say that $(H_{\sigma},J_{\sigma})_{\sigma \in \R}$
satisfies the {\it maximum principle}
if there is a sequence of relatively compact open sets $(V_i)_{i \in \N}$
whose union is $W_{\check{\phi}}$
so that $(H_\sigma,J_\sigma)_{\sigma \in \R}$
satisfies the maximum principle with respect to $V_i$ for all $i \in \N$.
\end{defn}

\begin{lemma} \label{lemma:maximumprincipalregion}
	Let $(M,\theta_M,\phi)$ be an abstract contact open book and 	let $W \equiv (W_{\check{\phi}},\pi_{\check{\phi}},\theta_{\check{\phi}})$ be a mapping cylinder of some positive slope perturbation of $\phi$.
	Let $(K_\sigma \equiv \pi_{\check{\phi}}^*(k_\sigma(s)))_{\sigma \in \R}$ be a smooth family of Hamiltonians strictly compatible with $W$ so that $\frac{dk_\sigma}{d\sigma} \leq 0$,
	$\frac{dk'_\sigma}{d\sigma} \leq 0$
and $\frac{dk''_\sigma}{d\sigma} \leq 0$.
	Let $Y$ be an almost complex structure strictly compatible with $W$.
	
	Let $\delta, S > 0$ and let $r_{\check{\phi}}$ be the vertical cylindrical coordinate of $W$.
	Let $(H_\sigma,J_\sigma)_{\sigma \in \R}$ be a smooth family of pairs compatible with $W$ which are equal to $(K_\sigma,Y)$ near the boundary of the set
	$V_{\delta,S} \equiv \pi_{\check{\phi}}^{-1}((-S,S) \times (\R/\Z)) \subset W_{\check{\phi}}$
	and also in the region $\{r_{\check{\phi}} \geq 1-\delta\}$.
	Then $V_{\delta,S}$ satisfies the maximum principal with respect to $(H_\sigma,J_\sigma)_{\sigma \in \R}$. 
\end{lemma}
\proof
Let $u : \Sigma \lra{} W_{\check{\phi}}$
be as in Definition \ref{defn:maximumprincipal} with $V$ replaced by $V_{\delta,S}$.
Let $$\iota_{\check{\phi}} : C_{\check{\phi}} \equiv (\R \times (\R/\Z) \times (1-\delta_{\check{\phi}},1] \times \partial M) \  \hookrightarrow W_{\check{\phi}}, \quad \iota_{\check{\phi}}(s,t,r_M,y) \equiv (s,t,(\phi^{\delta r_M}_{-t}(r_M,y))$$
be the vertical cylindrical end of $W_{\check{\phi}}$.
Then $r_{\check{\phi}} : C_{\check{\phi}} \lra{} (1-\delta_{\check{\phi}},1]$ 
is the natural projection map.
Let $P_M : C_{\check{\phi}} \lra{} (1-\delta_{\check{\phi}},1] \times \partial M)$
be the natural projection map.
Consider the map
$$\widehat{u} : u^{-1}(C_{\check{\phi}}) \lra{} (1-\delta_{\check{\phi}},1] \times \partial M), \quad \widehat{u}(\sigma,\tau)(x) \equiv P_M \circ u(x).$$
Let $J_{M,\sigma,t}$ be the natural almost complex structure on $(1-\delta_{\check{\phi}}) \times \partial M$ induced by $J_\sigma|_{\pi_{\check{\phi}}^{-1}(s,t)}$ for some $s$. Such an almost complex structure does not depend on $s$ and is cylindrical by definition.
Since $J_\sigma X_{H_{\sigma}}$ is a multiple of $\frac{\partial}{\partial s}$,
we get that $\widehat{u}$ satisfies
$$\frac{\partial \widehat{u}}{\partial \sigma} + J_M \frac{\partial \widehat{u}}{\partial \tau} - J_MX_{\delta r_M} = 0$$ where $r_M$ is the cylindrical coordinate on $M$ and $\delta>0$ is some constant.
Therefore by applying
\cite[Lemma 7.2]{SeidelAbouzaid:viterbo}
to $\widehat{u}$ we see that such a map cannot intersect the region
$(1-\delta,1] \times \partial M)$
and hence
the image of $u$ cannot intersect the region
$\{r_{\check{\phi}} \geq 1-\delta\}$.
%
%
%

Therefore we only need to show that the image of $u$ is contained inside $\pi_{\check{\phi}}^{-1}((-S,S) \times \R/\Z)$.
First of all we can make $S$ very slightly smaller so that $u$ is transverse to $\pi_{\check{\phi}}^{-1}(\{S,-S\})$ and so that $u(\partial \Sigma) \subset V_{\delta,S}$.
This implies that $\check{\Sigma} \subset \R \times \R/\Z$ is a smooth submanifold with boundary.
Suppose for a contradiction that $\check{\Sigma}$ is nonempty.
Choose $\delta_1 > 0$ small enough so that
$(H_\sigma,J_\sigma) = (K_\sigma,Y_\sigma)$ inside
$\pi_{\check{\phi}}^{-1}(([-S-\delta_1,-S] \cup [S,S+\delta_1]) \times \R/\Z)$ for all $\sigma \in \R$
and so that $u$ intersects $\pi_{\check{\phi}}^{-1}(\{s\} \times \R/\Z)$ transversely
for all $s \in [-S-\delta_1,-S] \cup [S,S+\delta_1]$.
Now let $\beta : \R \lra{} \R$ be a smooth function satisfying $\beta' \geq 0$,
$\beta|_{(-S-\delta_1/2,S + \delta_1/2)} = 0$,
$\beta'|_{(-S-3\delta_1/4,-S-\delta_1/2) \cup (S+\delta_1/2,S+3\delta_1/4)} > 0$
and
$\beta$ is constant outside $(-S-3\delta_1/4,S + 3\delta_1/4)$.
Choose a smooth function
$q_\sigma : \R \lra{} \R$
so that $q_\sigma|_{(-S-\delta_1/2,S + \delta_1/2)} = 0$
and $q_\sigma' = \beta' k'_\sigma$.
Then $\frac{dq'_\sigma}{d\sigma} \leq 0$.
Hence
$$0 < \int_{\check{\Sigma}} \beta'(s(u)) \left|(\pi_{\check{\phi}})_* \left(\frac{\partial u}{\partial \sigma}\right) \right|^2 d\sigma \wedge d\tau = \int_{\check{\Sigma}} u^*d(\beta(s)dt))\left(\frac{\partial u}{\partial \sigma},\frac{\partial u}{\partial \tau} - X_{H_{\sigma}}\right) d\sigma \wedge d\tau$$
$$
= \int_{\check{\Sigma}} u^*d(\beta(s)dt)) - u^* (\beta'(s) dH_{\sigma}) \wedge d\tau
= $$
$$
\int_{\check{\Sigma}} u^*d(\beta(s)dt)) -  d \left( u^* \left(q_{\sigma}(s)\right)\right) \wedge d\tau  + u^* \left( \frac{dq_{\sigma(s)}}{d\sigma} \right) d\sigma \wedge d\tau
$$
$$
\leq \int_{\check{\Sigma}} u^*d(\beta(s)dt) - d(q_\sigma(s(u)) d\tau)
= \int_{\partial \check{\Sigma}} u^*(\beta(s)dt) - q_\sigma(s(u)) d\tau = 0,
$$
giving us a contradiction.
\qed

\begin{corollary} \label{corollary:themaximumprinciap}
	Any smooth family $(H_\sigma,J_\sigma)_{\sigma \in \R}$
	compatible with $(W_{\check{\phi}},\pi_{\check{\phi}},\theta_{\check{\phi}})$
	with non-increasing slope
	satisfies the maximum principle.
\end{corollary}

\begin{defn}
	For any Mapping cylinder
	$W \equiv (W_{\check{\phi}},\pi_{\check{\phi}},\theta_{\check{\phi}})$, define
	$\beta_{\check{\phi}} \subset H_1(W_{\check{\phi}})$
	to be the set of homology classes represented by loops which project under $\pi_{\check{\phi}}$ to loops homotopic to
	$$\R/\Z \lra{} \R \times \R/\Z, \quad t \lra{} (0,t).$$
	For each $a, b \in [-\infty,\infty]$ where $a < b$,
	any non-degenerate Hamiltonian $(H_t)_{t \in [0,1]}$ and smooth family of almost complex structures $(J_t)_{t \in [0,1]}$ compatible with $W$
	we can define
	$HF^*_{[a,b],\beta_{\check{\phi}}}(\phi_1^{H_t})$
	in the same way as
	Floer cohomology of $\phi^{H_t}_1$
	except that we only consider fixed points of action in $[a,b]$ and whose associated $1$-periodic orbits represent an element of $\beta_{\check{\phi}}$.
	We also define
	$HF^*_{[-\infty,\infty],\beta_{\check{\phi}}}(\phi_1^{H_t}) \equiv HF^*_{\beta_{\check{\phi}}}(\phi_1^{H_t})$.

		A {\it non-degenerate autonomous Hamiltonian} is an autonomous Hamiltonian $H : W_{\check{\phi}} \lra{} \R$
		so that for every fixed point $p$ of the time $1$ flow $\phi^H_1$, the eigenspace of
		$D\phi^H_1(p)$ associated with the eigenvalue $1$ is one dimensional (this eigenspace is the tangent line to the $1$-periodic orbit associated to $p$).
		Now suppose that the mapping cylinder $(W_{\check{\phi}},\pi_{\check{\phi}},\theta_{\check{\phi}})$ is graded with grading
		$$\iota : \widetilde{Fr}(TW_{\check{\phi}}) \times_{\widetilde{Sp}(2n)} Sp(2n) \cong Fr(TW_{\check{\phi}}).$$
		Since any Hamiltonian $H_t$ is isotopic to the identity map $\text{id}_{W_{\check{\phi}}}$,
		we get that $\phi^{H_t}_1$ is naturally graded since the identity map on $\widetilde{Fr}(TW_{\check{\phi}})$
		makes
		$\text{id}_{W_{\check{\phi}}}$ into a graded symplectomorphism. We will call this the {\it standard grading}
		and from now on we will assume that every graded Hamiltonian symplectomorphism has the standard grading.
		A {\it standard perturbation} of a non-degenerate autonomous Hamiltonian $H$ where $\phi^H_1$ is graded is a time dependent Hamiltonian $(H_t)_{t \in [0,1]}$ which is $C^\infty$ close to $H$ and equal to $H$ outside a compact set where
		\begin{itemize}
			\item $(H_t)_{t \in [0,1]}$ is non-degenerate,
			\item every $1$-periodic orbit
			$\gamma$ of $(H_t)_{t \in [0,1]}$ is a $1$-periodic orbit of $H$ and
			\item for every $S^1$ family of $1$-periodic orbits $\gamma$ of $H$ there are exactly two $1$-periodic orbits $\gamma_-$ and $\gamma_+$
			in this family which are also $1$-periodic orbits of $H_t$.
			These orbits satisfy  $CZ(\phi^{H_t}_1,\gamma_\pm) = CZ(\phi^H_1,\gamma) \pm 1/2$.
		\end{itemize}
		Such a perturbation exists by \cite[Proposition 2.2]{CieliebakFloerHofer:SymhomII}.
\end{defn}

In order to prove our theorem we need another group called
{\it $S^1$-equivariant Hamiltonian Floer cohomology.}
See \cite{Viterbo:functorsandcomputations} and
\cite{OanceaBourgeois:s1equviariant} for a definition.
We will not define this here but we will just state some of the properties that we need.
We write these groups as $HF^*_{S^1,[a,b],\beta_{\check{\phi}}}(H,J)$
for any non-degenerate autonomous Hamiltonian $H$ and almost complex structures $J$ compatible with
$(W_{\check{\phi}},\pi_{\check{\phi}},\theta_{\check{\phi}})$ and any $a,b \in [-\infty,\infty]$
satisfying $a < b$.
We also define
$HF^*_{S^1,\beta_{\check{\phi}}}(H,J) \equiv HF^*_{S^1,[-\infty,\infty],\beta_{\check{\phi}}}(H,J).$

These groups satisfy the following properties.

\begin{S1HF}
\item \label{item:propertychaincomplex}
Let $S_H$ be the set of
$S^1$ families
of $1$-periodic orbits of $H$ 
with action in $[a,b]$ representing a class in $\beta_{\check{\phi}}$.
Let $(H_t)_{t \in [0,1]}$ be a $C^\infty$ small standard
perturbation of $H$.
This means that for each $\gamma \in S_H$,
there are two $1$-periodic orbits $\gamma_-, \gamma_+$ of $(H_t)_{t \in [0,1]}$ which are also orbits in $\gamma$
 satisfying
$CZ(\phi^{H_t}_1,\gamma_\pm) = CZ(\phi^H_1,\gamma) \pm 1/2$.
Let $S_{H_t}$ be the set of such orbits $\gamma_\pm$.

Then the chain complex  $CF^*_{S^1,[a,b],\beta_{\check{\phi}}}(H)$ defining $HF^*_{S^1,[a,b],\beta_{\check{\phi}}}(H,J)$ is a free $\Z[u]$-module generated by $S_{H_t}$ and graded by the Conley-Zehnder index taken with negative sign
and where the degree of $u$ is $-2$.
Let $C^*_{[a,b],\beta_{\check{\phi}}}(H) \subset C^*_{S^1,[a,b],\beta_{\check{\phi}}}(H)$
be the $\Z$-submodule generated by elements of $S_{H_t}$.
Then the differential $\partial$ on $CF^*_{S^1,[a,b],\beta_{\check{\phi}}}(H)$ is equal to $\partial_0 + \partial_1$
where $\partial_0(u^i C^*_{[a,b],\beta_{\check{\phi}}}(H)) \subset u^iC^*_{[a,b],\beta_{\check{\phi}}}(H)$
and $\partial_1(u^i C^*_{[a,b],\beta_{\check{\phi}}}(H)) \subset \oplus_{j=0}^{i-1} u^j C^*_{[a,b],\beta_{\check{\phi}}}(H)$ for all $i$.
Here $\partial_0$ is equal to the differential defining
$HF^*_{[a,b],\beta_{\check{\phi}}}(H,J)$.
Also $\partial(u^i \gamma_-) = u^{i-1}\gamma_+$ plus $1$-periodic orbits of higher action for all $i \geq 1$.
The differential is $\Z$-linear but not necessarily
$\Z[u]$-linear.

%
%
%
\item \label{item:propertycontinuation}
 If $(H_\sigma,J_\sigma)_{\sigma \in \R}$
is a smooth family of pairs compatible with $(W_{\check{\phi}},\pi_{\check{\phi}},\theta_{\check{\phi}})$
with non-increasing slope
joining
$(H,J)$ and $(\check{H},\check{J})$ 
then there is a group homomorphism
$$HF^*_{S^1,\beta_{\check{\phi}}}(\check{H},\check{J}) \lra{} HF^*_{S^1,\beta_{\check{\phi}}}(H,J).$$
If in addition $\frac{d}{d\sigma}H_\sigma \leq 0$,
then we have a group homomorphism
$$HF^*_{S^1,[a,b],\beta_{\check{\phi}}}(\check{H},\check{J}) \lra{} HF^*_{S^1,[a,b],\beta_{\check{\phi}}}(H,J)$$
for all $a < b$.
These are called {\it continuation maps}.
They do not depend on the choice of path $(H_\sigma,J_\sigma)_{\sigma \in \R}$
and the composition of two continuation maps is a continuation map.
Also if $(H_\sigma,J_\sigma)$ does not depend on $\sigma \in \R$ then the associated continuation map is the identity map.
If $a = -\infty$ and $b = \infty$ and if $H_\sigma = H + f(\sigma)$ for some function $f : \R \lra{} \R$ then the corresponding continuation map is also an isomorphism.
\item \label{item:actionrestrictionproperty}
 If $(H,J)$, $(\check{H},\check{J})$
 are compatible with
 $W \equiv (W_{\check{\phi}},\pi_{\check{\phi}},\theta_{\check{\phi}})$ and
 \begin{itemize}
\item  satisfy the maximum principle with respect to some $V \subset W_{\check{\phi}}$, 
 	\item all of $H,\check{H}$'s $1$-periodic orbits of action in $[a,b]$ representing elements of $\beta_{\check{\phi}}$ are contained in $V$ and
 	\item  $(H,J)|_V = (\check{H},\check{J})|_V$,
 \end{itemize}
then
$$HF^*_{S^1,[a,b],\beta_{\check{\phi}}}(\check{H},\check{J}) \cong HF^*_{S^1,[a,b],\beta_{\check{\phi}}}(H,J).$$
This is due to the fact that their chain complexes are identical.
Also if we have two additional pairs $(H',J'),(\check{H}',\check{J}')$ satisfying the same properties
and a smooth non-decreasing family of pairs $(H'_\sigma,J'_\sigma)_{\sigma \in \R}$ and
$(\check{H}'_\sigma,\check{J}'_\sigma)_{\sigma \in \R}$ compatible with $W$ joining
$(H,J)$ and $(H',J')$ and joining
$(\check{H},\check{J})$ and $(\check{H}',\check{J})'$ respectively,
satisfying the maximum principle with respect to $V$
and which are equal inside $V$ for all $\sigma$ then the induced continuation maps commute with the above isomorphisms.
\end{S1HF}

\begin{defn}
	We define
	$$SH^*_{S^1,\beta_{\check{\phi}}}(W_{\check{\phi}},\pi_{\check{\phi}},\theta_{\check{\phi}}) \equiv \varinjlim_{(H,J)} HF^*_{S^1,\beta_{\check{\phi}}}(H,J)$$
	where the direct limit is taken over all pairs
	$(H,J)$ compatible with
	$(W_{\check{\phi}},\pi_{\check{\phi}},\theta_{\check{\phi}})$
	using the partial ordering $\leq$ on $H$.
	
	Let $\preceq$ be a partial order on a set $S$.
	A {\it cofinal family} is a subset $S' \subset S$
	so that for all $s \in S$, there exists
	an $s' \in S'$ so that $s \preceq s'$.
	In the definition above, it is sufficient to
	compute $SH^*_{S^1,\beta_{\check{\phi}}}(W_{\check{\phi}},\pi_{\check{\phi}},\theta_{\check{\phi}})$
	by taking the direct limit over some cofinal family of pairs $(H,J)$ as above.
\end{defn}

\begin{lemma} \label{lemma:fixedhamiltonian}
If the slope of a pair $(H,J)$ compatible with
$(W_{\check{\phi}},\pi_{\check{\phi}},\theta_{\check{\phi}})$
is greater than $1$ then the natural map
$$HF^*_{S^1,\beta_{\check{\phi}}}(H,J) \lra{} SH^*_{S^1,\beta_{\check{\phi}}}(W_{\check{\phi}},\pi_{\check{\phi}},\theta_{\check{\phi}})$$
is an isomorphism.
\end{lemma}
\proof
Let $(\check{H},\check{J})$ be a pair which is strictly compatible with our mapping cylinder and so that the slope of $\check{H}$ is equal to the slope of $H$.
Then there is a constant $c > 0$ so that
$\check{H} + c > H$ and $H + c > \check{H}$.
Consider the continuation maps
$$HF^*_{S^1,\beta_{\check{\phi}}}(H,J) \lra{\alpha} HF^*_{S^1,\beta_{\check{\phi}}}(\check{H}+c,\check{J}) \lra{} HF^*_{S^1,\beta_{\check{\phi}}}(H+2c,J)
\lra{} HF^*_{S^1,\beta_{\check{\phi}}}(\check{H}+3c,\check{J}).$$
By \ref{item:propertycontinuation},
the composition of any two such maps is an isomorphism
and hence the continuation map
$\alpha$ is an isomorphism.
Therefore it is sufficient for us to assume that $(H,J)$ is strictly compatible with our mapping cylinder.
We can also assume that  $H =  \pi_{\check{\phi}}^* h(s)$ where $h''(s) \geq 0$.

Choose $S>0$ large enough so that $\pi_{\check{\phi}}^{-1}((-S,S) \times \R/\Z)$ contains all the $1$-periodic orbits of $H$ representing a class in $\beta_{\check{\phi}}$.
Let $h_\sigma : \R \lra{} \R, \ \sigma \in [0,\infty)$ be a smooth family of functions so that
\begin{itemize}
\item $h_\sigma(s) = h(s)$ for all $s \in (-S,S)$
and $h_0(s) = h(s)$ for all $s \in \R$,
\item $h'_\sigma(s),h''_\sigma(s),\frac{d}{d\sigma}h_\sigma(s) \geq 0$,
\item $h''_\sigma(s) = 0$ for all large enough $s$ and
\item the slope of $h_\sigma$ tends to infinity as $\sigma$ tends to infinity.
\end{itemize}
By \ref{item:actionrestrictionproperty} combined with Lemma \ref{lemma:maximumprincipalregion},
the natural continuation map
$$HF^*_{S^1,\beta_{\check{\phi}}}(H,J) \lra{} HF^*_{S^1,\beta_{\check{\phi}}}(\pi_{\check{\phi}}^* h_\sigma(s),J)$$
is an isomorphism for all $\sigma \geq 0$.
Also by \ref{item:propertycontinuation},
the natural continuation map
$$HF^*_{S^1,\beta_{\check{\phi}}}(H,J) \lra{} HF^*_{S^1,\beta_{\check{\phi}}}(\pi_{\check{\phi}}^* h_\sigma(s) + \sigma,J)$$
is an isomorphism for all $\sigma \geq 0$.
Since $(\pi_{\check{\phi}}^* h_\sigma(s) + \sigma,J)$
is a cofinal family of pairs with respect to the ordering $\leq$, we get our result by \ref{item:propertycontinuation}.
\qed

\begin{lemma} \label{lemma:symplecticohomologynopositiveaction}
Fix $q \in \R$.
We have
$$SH^*_{S^1,\beta_{\check{\phi}}}(W_{\check{\phi}},\pi_{\check{\phi}},\theta_{\check{\phi}}) = \varinjlim_{(H,J)} HF^*_{S^1,(-\infty,0],\beta_{\check{\phi}}}(H,J)$$
where the direct limit is taken over pairs
$(H,J)$
compatible with our mapping cylinder
satisfying $H|_{\pi_{\check{\phi}}^{-1}((-\infty,q] \times \R / \Z)} < 0$.
\end{lemma}
\proof
Since the continuation map
$$HF^*_{S^1,\beta_{\check{\phi}}}(H,J) \lra{} HF^*_{S^1,\beta_{\check{\phi}}}(H+c,J)$$
is an isomorphism by \ref{item:propertycontinuation}
for every pair $(H,J)$ compatible with our mapping cylinder, we have
$$SH^*_{S^1,\beta_{\check{\phi}}}(W_{\check{\phi}},\pi_{\check{\phi}},\theta_{\check{\phi}}) = \varinjlim_{(H,J)} HF^*_{S^1,\beta_{\check{\phi}}}(H,J)$$
where the direct limit is taken over pairs
$(H,J)$
compatible with our mapping cylinder
satisfying $H|_{\pi_{\check{\phi}}^{-1}((-\infty,q] \times \R / \Z)} < 0$.
Let $A$ be a constant smaller than the action of all the $1$-periodic orbits of $\check{\phi}$.
We say that $h : \R \lra{} \R, \ i \in \N$
is a {\it compatible function with respect to $q$} if
\begin{itemize}
\item $h', h'' \geq 0$, $h$ is bounded below,
\item if $h'(x) = 1$ then $h(x) < x + \kappa A$ and
\item $h_i|_{(-\infty,q]} < 0$.
\end{itemize}
Then $$SH^*_{S^1,\beta_{\check{\phi}}}(W_{\check{\phi}},\pi_{\check{\phi}},\theta_{\check{\phi}}) = \varinjlim_{h} HF^*_{S^1,\beta_{\check{\phi}}}(\pi_{\check{\phi}}^* h(s),J)$$
where the direct limit is taken over compatible functions with respect to $q$ ordered by $\leq$ and where $J$ is an almost complex structure compatible with our mapping cylinder.
Since $\pi_{\check{\phi}}^*h(s)$ has no $1$-periodic orbits representing $\beta_{\check{\phi}}$ of positive action where $h$ is any compatible function with respect to $q$ we have by \ref{item:actionrestrictionproperty},
$$\varinjlim_h HF^*_{S^1,\beta_{\check{\phi}}}(\pi_{\check{\phi}}^* h(s),J)
=
\varinjlim_h HF^*_{S^1,(-\infty,0],\beta_{\check{\phi}}}(\pi_{\check{\phi}}^* h(s),J)$$
proving our result.
\qed

\begin{lemma} \label{lemma:symplecticminusisomorphism}
Let $q \in \R$.
Let $(\check{H},\check{J})$ be compatible with $(W_{\check{\phi}},\pi_{\check{\phi}},\theta_{\check{\phi}})$
then
$$SH^*_{S^1,\beta_{\check{\phi}}}(W_{\check{\phi}},\pi_{\check{\phi}},\theta_{\check{\phi}}) = \varinjlim_{(H,J)} HF^*_{S^1,(-\infty,0],\beta_{\check{\phi}}}(H,J)$$
where the direct limit is taken over pairs $(H,J)$
compatible with our mapping cylinder
satisfying $H|_{\pi_{\check{\phi}}^{-1}((-\infty,q] \times \R / \Z)} < 0$
and $(H,J)|_{\pi_{\check{\phi}}^{-1}([q+1,\infty) \times \R/\Z)} = (\check{H} + C_H,J)|_{\pi_{\check{\phi}}^{-1}([q+1,\infty) \times \R/\Z)}$
for some constant $C_H \in \R$.
\end{lemma}
\proof
Let $(H_i,J_i)_{i \in \N}$ be a cofinal family of pairs with respect to the directed system mentioned in the statement of this lemma.
Such a countable family exists since $$\sup H|_{\pi_{\check{\phi}}^{-1}((-\infty,q] \times (\R/\Z))} < 0$$ for any $(H,J)$ in the directed system above.
We have that $(H_i,J_i)$
is compatible with our mapping cylinder
and $H_i|_{(\pi_{\check{\phi}}^{-1}((-\infty,q] \times \R / \Z))} < 0$
and $(H_i,J_i)|_{\pi_{\check{\phi}}^{-1}([q+1,\infty) \times \R/\Z)} = (\check{H} + C_{H_i},J_i)|_{\pi_{\check{\phi}}^{-1}([q+1,\infty) \times \R/\Z))}$
for some constant $C_{H_i} \in \R$.
We can assume that $H_i < H_{i+1}$ and hence $C_{H_i} < C_{H_{i+1}}$ for all $i \in \N$.
After passing to a subsequence, we can assume that $C_{H_i} > i$ for all $i \in \N$.

Let $(s,t)$ be standard coordinates for the base $\R \times \R/\Z$.
Let $K : W_{\check{\phi}} \lra{} \R$ be a Hamiltonian
equal to $\pi_{\check{\phi}}^* k(s)$
where $k(s) = 0$ for $s \leq q+1$,
$k'(s) > 0$ for $s > q+1$,
and $k'(s)$ is constant for $s > q + 2$.
Since $C_{H_i} > i$ for all $i$, there is a
$\delta > 0$ small enough so that
the set of $1$-periodic orbits of $H_i$
of non-positive
action are equal to the set of $1$-periodic orbits of $H_i + \delta i K$ of non-positive action.
Hence by \ref{item:actionrestrictionproperty},
$$HF^*_{S^1,(-\infty,0],\beta_{\check{\phi}}}(H_i,J_i) = HF^*_{S^1,(-\infty,0],\beta_{\check{\phi}}}(H_i+\delta i K,J_i).$$
Combining this with
Lemma \ref{lemma:symplecticminusisomorphism}
gives us our result.
%
\qed

\begin{lemma} \label{lemma:isomorphicsymplecticcohomologyps}
Suppose that the mapping cylinders
$(W_{\check{\phi}_1},\pi_{\check{\phi}_1},\theta_{\check{\phi}_1})$,
$(W_{\check{\phi}_2},\pi_{\check{\phi}_2},\theta_{\check{\phi}_2})$
are isomorphic.
Then $$SH^*_{S^1,\beta_{\check{\phi}_1}}(W_{\check{\phi}_1},\pi_{\check{\phi}_1},\theta_{\check{\phi}_1}) = SH^*_{S^1,\beta_{\check{\phi}_2}}(W_{\check{\phi}_2},\pi_{\check{\phi}_2},\theta_{\check{\phi}_2}).$$
\end{lemma}
\proof
Since these mapping cylinders are isomorphic, we can assume that $W_{\check{\phi}_1} = W_{\check{\phi}_2}$, $\pi_{\check{\phi}_2} = \pi_{\check{\phi}_2}$ inside $\{r_{\check{\phi}_1} \geq 1-\delta\}$ for some $\delta > 0$ and $\theta_{\check{\phi}_1} = \theta_{\check{\phi}_2} + k$ where $k : W_{\check{\phi}_1} \lra{} \R$ has support disjoint from $\{r_{\check{\phi}_1} \geq 1-\delta\}$.

Let $(H_1,J_1)$ (resp. $(H_2,J_2)$) be a pair strictly compatible with
$$W_1 \equiv (W_{\check{\phi}_1},\pi_{\check{\phi}_1},\theta_{\check{\phi}_1}) \quad \text{(resp.} \ W_2 \equiv (W_{\check{\phi}_2},\pi_{\check{\phi}_2},\theta_{\check{\phi}_2}) \text{),}$$
of slope less than some small $\check{\delta}>0$.
If $\check{\delta}>0$ is small enough then we can construct a pair
$(H_3,J_3)$ compatible with
$W_1$,
which is equal to
$(H_2,J_2)$ in the region
$\pi_{\check{\phi}_2}^{-1}((-1,3) \times \R / \Z)$
and equal to $(H_1,J_1)$ inside
$\pi_{\check{\phi}_2}^{-1}(((-\infty,-2),(4,\infty)) \times \R / \Z)$ so that $H_3$ has no $1$-periodic orbits.

Let $(\check{H}_i,\check{J}_i)_{i \in \N}$ be a family of pairs
strictly compatible with
$W_2$
so that 
\begin{itemize}
\item $(\check{H}_i,\check{J}_i)$ is equal to $(H_2 + C_{\check{H}_i},J_2)$ inside $\pi_{\check{\phi}_2}^{-1}([1,\infty) \times \R/\Z)$  for some constants $C_{\check{H}_i} \in \R$,
\item the restriction $\check{H}_i|_{\pi_{\check{\phi}_2}^{-1}((-\infty,0) \times \R/\Z)}$ is negative and uniformly tends to $0$ in the $C^1$ norm as $i$ tends to infinity and $\check{H}_i(x) \to \infty$ as $i \to \infty$ for all $x$ in ${\pi_{\check{\phi}_2}^{-1}((0,\infty) \times \R/\Z)}$ . 
\end{itemize}
Let $(\widehat{H}_i,\widehat{J}_i)_{i \in \N}$ be a family of pairs
 compatible with $W_1$ so that 
\begin{itemize}
\item $(\widehat{H}_i,\widehat{J}_i)$ is equal to $(H_3 + C_{\check{H}_i},\check{J}_2)$ inside $\pi_{\check{\phi_1}}^{-1}([1,\infty) \times \R/\Z)$,
\item $(\widehat{H}_i,\widehat{J}_i)$ equals $(\check{H}_i,\check{J}_i)$ inside
$\pi_{\check{\phi}_2}^{-1}((-1,3) \times \R/\Z)$,
\item the restriction $\widehat{H}_i|_{\pi_{\check{\phi}_1}^{-1}((-\infty,0) \times \R/\Z)}$ is negative and uniformly tends to $0$ in the $C^1$ norm as $i$ tends to infinity
and $\widehat{H}_i(x) \to \infty$ as $i \to \infty$ for all $x$ in ${\pi_{\check{\phi_1}}^{-1}((0,\infty) \times \R/\Z)}$.
\end{itemize}


Then by Lemma \ref{lemma:symplecticminusisomorphism},
\begin{equation} \label{eqn:secondsymplectichomologygroup}
SH^*_{S^1,\beta_{\check{\phi}_1}}(W_{\check{\phi}_2},\pi_{\check{\phi}_2},\theta_{\check{\phi}_2}) \equiv \varinjlim_{i \in \N} HF^*_{S^1,(-\infty,0]}(\check{H}_i,\check{J}_i).
\end{equation}
and
\begin{equation} \label{eqn:firstsymplectichomologygroup}
SH^*_{S^1,\beta_{\check{\phi}_1}}(W_{\check{\phi}_1},
\pi_ {\check{\phi}_1},\theta_{\check{\phi}_1}) \equiv \varinjlim_{i \in \N} HF^*_{S^1,(-\infty,0]}(\widehat{H}_i,\widehat{J}_i).
\end{equation}

Also by property \ref{item:actionrestrictionproperty},
$HF^*_{S^1,(-\infty,0],\beta_{\check{\phi}_1}}(\widehat{H}_i,\widehat{J}_i) \cong HF^*_{S^1,(-\infty,0],\beta_{\check{\phi}_1}}(\check{H}_i,\check{J}_i)$
for all $i$
and the continuation maps between these groups commute with these isomorphisms.
Hence
$$SH^*_{S^1,\beta_{\check{\phi}_1}}(W_{\check{\phi}_1},\pi_{\check{\phi}_1},\theta_{\check{\phi}_1}) \cong
SH^*_{S^1,\beta_{\check{\phi}_1}}(W_{\check{\phi}_2},\pi_{\check{\phi}_2},\theta_{\check{\phi}_2})$$
by Equations
(\ref{eqn:firstsymplectichomologygroup})
and 
(\ref{eqn:secondsymplectichomologygroup}).

\qed

\begin{lemma} \label{lemma:homologycalculation}
Let $A_-,A_+$ be free abelian groups.
Define $B_- \equiv A_- \otimes \Z[u], \ B_+ \equiv A_+ \otimes \Z[u]$.
Let 
$$\partial \equiv \partial_0 + \partial_1 : B_- \oplus B_+ \lra{} B_- \oplus B_+$$
be a $\Z$-linear differential 
where $	\partial_0(A_-) \subset A_-$
and $\partial_1(A_-) = 0$.
Now suppose that we have a filtration 
$B_- \oplus B_+ = F_0 \supset F_1 \supset F_2 \cdots$
for the chain complex $(B_- \oplus B_+, \partial) $
so that if $V^i_\pm \equiv (B_\pm \cap F_i)/(B_\pm \cap F_{i+1})$
then $\partial_1(uV^i_-) \subset V^i_+$ and
\begin{equation} \label{equation:mapupartial1}
\partial_1|_{uV^i_-} : uV^i_- \lra{} V^i_+
%
\end{equation}
is an isomorphism for all $i \geq 0$.
Then
$$H(B_- \oplus B_+,\partial) = H(A_-,\partial_0).$$
If these groups are graded then the above isomorphism respects this grading.

\end{lemma}
\proof
Define $\check{B} \equiv \partial(uB_-)$
and define
$\partial_{\check{B}} : uB_- \lra{} \check{B}, \quad \partial_{\check{B}}(x) = \partial(x)$.
Since the map (\ref{equation:mapupartial1})
is an isomorphism for all $i \geq 0$, we get that
$\partial_1|_{uB_-} : u B_- \lra{} (B_- \oplus B_+) / (B_- \oplus 0)$ is an isomorphism.
Hence $B_- \oplus B_+ \cong A_- \oplus u B_- \oplus \check{B}$
and the differential with respect to this splitting is the matrix
$$\left(
\begin{array}{ccc}
\partial_0 & 0 & 0 \\
0 & 0 & 0 \\
0 & \partial_{\check{B}} & 0
\end{array}
\right).$$
Computing the homology of this chain complex using the above matrix gives us our result since
$\partial_{B'}$ is an isomorphism.
\qed

\begin{lemma} \label{lemma:HFSHisomorphism}
Let $(W_{\check{\phi}},\pi_{\check{\phi}},\theta_{\check{\phi}})$
be a mapping cylinder.
Then
$$SH^*_{\beta_{\check{\phi}}}(W_{\check{\phi}},\pi_{\check{\phi}},\theta_{\check{\phi}}) = HF^*(\phi,+).$$
\end{lemma}
\proof
Let $(H = \pi_{\check{\phi}}^* h(s),J)$ be a pair strictly compatible with
$(W_{\check{\phi}},\pi_{\check{\phi}},\theta_{\check{\phi}})$ where $H$ has slope $1.5$, $h'' \geq 0$ and where $h'|_{(-\infty,0)} < 1$.
Let $h_t : \R \times \R/\Z \lra{} \R, \quad t \in [0,1]$
be a standard perturbation of $h(s)$ viewed as a Hamiltonian on the symplectic manifold $(\R \times \R/\Z, ds \wedge dt)$.
Then $h_t$ has exactly two $1$-periodic orbits
$\gamma_-,\gamma_+$.
Also $H_t \equiv \pi_{\check{\phi}}^* h_t$ is a standard perturbation of $H$ and the $1$-periodic orbits of $H_t$ project to $\gamma_-$ or $\gamma_+$.
A compactness argument (\cite{BEHWZ:compactnessfieldtheory}) tells us that $(H_t,J)$ satisfies the maximum principle so long as $H_t$ is sufficiently $C^\infty$ close to $H$ and hence we can define $HF^*_{\beta_{\check{\phi}}}(\phi^{H_t}_1)$ in the usual way.
We can also define
$HF^*_{S^1,\beta_{\check{\phi}}}(H,J)$ using the standard perturbation $H_t$ by the same compactness argument.

By \cite[Theorem 1.3]{McLean:monodromy},
we have that
$HF^*_{\beta_{\check{\phi}}}(\phi^{H_t}_1,J)$
is isomorphic to $HF^*_{\beta_{\check{\phi}}}(\phi^{H_t}_1,+) \oplus HF^{*+1}_{\beta_{\check{\phi}}}(\phi^{H_t}_1,+)$.
In fact, in the proof of the above Theorem it was shown that if $A_-$ (resp. $A_+$)
is the free abelian group generated by $1$-periodic orbits of $H_t$ which project to $\gamma_-$ (resp. $\gamma_+$) then the differential $$\partial_{H_t,J} : A_- \oplus A_+ \lra{} A_- \oplus A_+$$
satisfies $\partial_{H_t,J}(A_\pm) \subset A_\pm$
and the homology of
$$\partial_{H_t,J}|_{A_-} : A_- \lra{} A_-$$
is equal to $HF^*(\phi^{H_t}_1,J)$.

If $\check{\phi}$ is a sufficiently generic positive slope perturbation of $\phi$,
then we can find
a sequence $(\alpha_i)_{i \in \N_{\geq 0}}$
so that there are exactly two $1$-periodic orbits
$\check{\gamma}_-,\check{\gamma}_+$ of $H_t$
of action in the interval
$[\alpha_i,\alpha_{i+1})$ and all $1$-periodic orbits are contained in one such interval.
Let $B_\pm \equiv A_\pm \otimes \Z[u]$ where $u$ has degree $-2$.
Let $B_- \oplus B_+ = F_0 \supset F_1 \supset \cdots$
be a filtration where $F_i$ is the $\Z[u]$-submodule
generated by orbits of action $\geq \alpha_i$.
Define $V^i_\pm \equiv (B_\pm \cap F_i)/(B_\pm \cap F_{i+1})$ for all $i \in \N_{\geq 0}$.
By \ref{item:propertychaincomplex}, the differential
$\partial : B_- \oplus B_+ \lra{} B_- \oplus B_+$
computing $HF^*_{S^1,\beta_{\check{\phi}}}(H,J)$
is equal to $\partial_0 + \partial_1$
where $\partial_0(A_-) \subset A_-$,
$\partial_1(A_-) = 0$,
$\partial_0|_{A_-} = \partial_{H_t,J}$,
$\partial_1(uV^i_-) \subset V^i_+$
and $\partial_1|_{uV^i_-} : uV^i_- \lra{} V^i_+$ is an isomorphism for all $i \in \N_{\geq 0}$.

Therefore by Lemma \ref{lemma:homologycalculation}
we have that $HF^*_{S^1,\beta_{\check{\phi}}}(H,J) = H(A_-,\partial_0) = HF^*(\phi,+)$.
Also
$HF^*_{S^1,\beta_{\check{\phi}}}(H,J) = SH^*_{S^1,\beta_{\check{\phi}}}(W_{\check{\phi}},\pi_{\check{\phi}},\theta_{\check{\phi}})$
by Lemma \ref{lemma:fixedhamiltonian} and hence
$SH^*_{S^1,\beta_{\check{\phi}}}(W_{\check{\phi}},\pi_{\check{\phi}},\theta_{\check{\phi}}) = HF^*(\phi,+)$.
\qed

\begin{proof}[Proof of Property \ref{item:HFOBD}.]
This follows from Lemmas
\ref{lemma:obdisomorphism}, \ref{lemma:isotopicimpliesisomorphicHF},
 \ref{lemma:isomorphicsymplecticcohomologyps} and
\ref{lemma:HFSHisomorphism}.
\end{proof}

Now the only issue is that if we have two polynomials with embedded contactomorphic links.
Then we need to show that the associated
contact pairs are isomorphic.
In other words, we need to show that
the normal bundles coincide up to homotopy.
This is contained in the proof of the following lemma.

\begin{lemma} \label{lemma:invarianceforlinks}
Let $f,g : \C^{n+1} \lra{} \C$ be a polynomials with isolated singularities at $0$ with embedded contactomorphic links where $n \geq 1$.
Then $HF^*(\phi^m,+) = HF^*(\psi^m,+)$
where $\phi$ (resp. $\psi$) is the monodromy map of the Milnor open book associated to $f$ (resp. $g$)
as in Example \ref{example:milnoropenbook}.
\end{lemma}
\proof
Let $(L_f \subset S_\epsilon, \xi_{S_\epsilon},\Phi_f)$ and
$(L_g \subset S_\epsilon, \xi_{S_\epsilon},\Phi_g)$
be the contact pairs associated to $f$ and $g$ respectively as in Example \ref{example:contactpairoff}.
Let $\Psi : S_\epsilon \lra{} S_\epsilon$ be the contactomorphism sending $L_f$ to $L_g$.
We need to show that $\Psi$ is in fact a contactomorphism of graded contact pairs by \ref{item:HFOBD}.
Since $H_1(S_\epsilon;\Q) = 0$, we get that $\Psi$ is a graded contactomorphism by Equation (\ref{equation:gradingtrivializationcorrespondence}) in Definition \ref{defn:gradingtrivializationcorrespondence}.
Therefore we just need to show that the composition
$$\cN_{S_\epsilon}L_f \lra{d\Psi|_{L_f}} \cN_{S_\epsilon} L_g \lra{\Phi_g} L_g \times \C \lra{\Psi^{-1} \times \text{id}_\C} L_f \times \C$$
is homotopic to $\Phi_f$.

This is true since the trivialization $\Phi_f$ (and similarly $\Phi_g$) is uniquely determined by the following topological property:
Let $\Psi_f : \check{\cN}_{S_\epsilon} L_f \lra{} S_\epsilon$
be a regularization of $L_f$ as in Definition \ref{defn:regularization}.
Then the trivialization $\Phi_f$
gives us a section $s$ of $\check{\cN}_{S_\epsilon} L_f$
whose image under the trivialization
$\Phi_f$ is a constant section.
Then $s$ is the unique section up to homotopy with the property that the image of $H_1(L_f;\Q) \lra{\Psi_f \circ s} H_1(S_\epsilon - L_f;\Q)$ is zero.
This trivialization could be thought of as a generalization of the Seifert framing of links.

\qed

\section{Appendix C: A Morse-Bott Spectral Sequence}

In this section, we will show that property
\ref{item:spectralsequenceproperty} holds.
Here is a statement of this property:

{\it Let $(M,\theta_M,\phi)$ be a graded abstract contact open book where $\text{dim}(M) = 2n$.
Suppose that the set of fixed points of
a small positive slope perturbation
$\check{\phi}$ of $\phi$ is a disjoint union of
codimension $0$ families of fixed points $B_1,\cdots,B_l$ and let $\iota : \{1,\cdots,l\} \lra{} \N$
be a function where
\begin{itemize}
\item $\iota(i) = \iota(j)$ if and only if the action of $B_i$ equals the action of $B_j$ and
\item $\iota(i) < \iota(j)$ if the action of $B_i$ is less than the action of $B_j$.
\end{itemize}
Then there is a cohomological spectral sequence converging to $HF^*(\phi,+)$ whose $E_1$ page is equal to
\begin{equation} \label{eqn:morsebottspectralsequence}
E_1^{p,q} = 
\bigoplus_{ \{i \in \{1,\cdots,l\} \ : \ \iota(i) = p\} }
H_{n-(p+q)-CZ(\phi,B_j)}(B_p;\Z)
.
\end{equation}

\bigskip

The spectral sequence above is an example of a {\it Morse-Bott spectral sequence}. 
Before we prove this statement we need some preliminary definitions and lemmas.

\begin{defn} \label{defn:localfloer}
$(M,\theta_M,\phi)$ be a graded abstract contact open book.
Let $\check{\phi} : M \lra{} M$ be a small positive slope perturbation of $\phi$ and $(J_t)_{t \in [0,1]}$ a $C^\infty$ generic family of $d\theta_M$-compatible almost complex structures.
Let $a,b \in \R$ be real numbers with the property that no fixed point of $\check{\phi}$
has action equal to $a$ or $b$.

We define $HF^*_{[a,b]}(\check{\phi},(J_t)_{t \in [0,1]})$ in the following way:
Let $\check{\phi}'$ be a $C^\infty$ small generic perturbation of $\check{\phi}$ inside a compact set
so that all of the fixed points of $\check{\phi}'$ are non-degenerate.
Then $HF^*_{([a,b]}(\check{\phi},(J_t)_{t \in [0,1]})$ is defined in the same way as
$HF^*(\check{\phi}',(J_t)_{t \in [0,1]})$
except that we only consider orbits inside the action window $[a,b]$.
This group does not depend on the choice of perturbation $\check{\phi}'$
so long as no fixed point of $\check{\phi}$ has action equal to $a$ or $b$.

We can define this group in the following equivalent way:
Let $CF^*(\check{\phi}')$ be the chain complex for $\check{\phi}'$.
Then the subspace $CF^*_{[a,\infty]}(\check{\phi}')$ consisting of fixed points of action $\geq a$
is a subcomplex.
We define $HF^*_{[a,b]}(\check{\phi}',(J_t)_{t \in [0,1]})$
to be the homology of the quotient complex
$CF^*_{[a,\infty]}(\check{\phi}')/ CF^*_{[b,\infty]}(\check{\phi}')$.

Suppose that $B$ is the set of fixed points of $\check{\phi}$ of action $c$ and suppose that there is some $a<c<b$ so that there are no fixed points of action in $[a,b]-c$.
We define
$$HF^*(\check{\phi},B) \equiv HF^*_{[a,b]}(\check{\phi},(J_t)_{t \in [0,1]}).$$
This does not depend on the choice of $a,b$ or $(J_t)_{t \in [0,1]}$.
\end{defn}

\begin{lemma} \label{lemma:smallperturbationinneighborhood}
$(M,\theta_M,\phi)$ be a graded abstract contact open book.
Let $\check{\phi} : M \lra{} M$ be the composition of $\phi$ with a $C^\infty$ small Hamiltonian so that $\check{\phi}$ has small positive slope.
Let $B \subset M$ be an isolated family of fixed points of $\check{\phi}$.

Let $(J_t)_{t \in [0,1]}$ be a smooth family of almost complex structures cylindrical near $\partial M$.
Then there is an neighborhood $N_B \subset M $ of $B$
so that for any sufficiently small $C^\infty$
perturbation $\check{\phi}'$,
any Floer trajectory of $(\check{\phi}',(J_t)_{t \in [0,1]})$ connecting non-degenerate fixed points $p,\check{p} \in N_B$ of $\check{\phi}'$
is contained inside $N_B$.
\end{lemma}
\proof
We choose a relatively compact open neighborhood $N_B$ of $B$ so that any fixed point of $\check{\phi}$ inside $\overline{N_B}$ is actually contained inside $B$.
Let $\check{N}_B \subset M$ be an open neighborhood of $B$
whose closure is contained in $N_B$.

Let $(\phi_k)_{k \in \N}$ be a sequence of symplectomorphisms of $M$ which $C^\infty$ converge to $\check{\phi}$.
Suppose (for a contradiction) that $\phi_k$ has a fixed point
$p_k \in N_B - \check{N}_B$ for all $k$.
Then after passing to a subsequence,
we have that $p_i$ converges to some $p \in \overline{N_B} - \check{N}_B$.
Since $p$ is a fixed point of $\phi$,
we get that $p \in B$ which is impossible.
Therefore $\phi_k$ has no fixed points inside $N_B - \check{N}_B$ for all sufficiently large $k$.

Now suppose that $p_k,\check{p}_k$ are fixed points of $\phi_k$ and suppose that
we have a sequence of Floer trajectories
$$u_k : \R \times [0,1] \lra{} M$$
of $(\check{\phi}',(J_t)_{t \in [0,1]})$
joining $p_k$ and $\check{p}_k$.
Define $W \equiv \R \times [0,1] \times M$
with a symplectic form $\omega_W \equiv ds \wedge dt + d\theta_M$ where $s,t$ are the standard coordinates on $\R \times [0,1]$.
Let $i$ be the standard complex structure on $\R \times [0,1]  \subset \C$
where $(s,t)$ is identified with $s + it$.
Define $J^W|_{(s,t,x)} \equiv i|_{(s,t)} \oplus J_t|_{(s,t)}$.
Define
$$u^W_k : \R \times [0,1] \lra{} W, \quad u^W_k(s,t) \equiv (s,t,u_k(s,t))$$
for all $k \in \N$.
This is a sequence of $J^W$-holomorphic maps.

Now suppose (for a contradiction)
that the image of $u_k$ is not contained inside $N_B$ for all $k$.
Then after passing ot a subsequence,
there is a sequence of points
$(s_k,t_k) \in \R \times [0,1]$
so that $u_k(s_k,t_k) \in N_B - \check{N}_B$
and $u_k(s_k,t_k)$ converges to some point
$q \in \overline{N_B} - \check{N}_B$.
After reparameterizing the domain by translations in the $s$ direction,
we can assume that $s_k = 0$ for all $k$.
Also after passing to a subsequence we can assume that $t_k \lra{} \check{t} \in [0,1]$ for some $\check{t}$.
Define $w_k \equiv u^W_k|_{[-1,1] \times [0,1]}$.
Then by the main result in
\cite{fish:compactness},
we get that $w_k$ $C^0$ converges to a continuous map
$v : [-1,1] \times [0,1] \lra{} W$ which is smooth and $J^W$-holomorphic on a dense open subset of its domain.

Let $\pi_M : W \lra{} M$ be the natural projection map.
Since $p_k$ and $\check{p}_k$ converge to points in $B$, their difference in action converges to zero which implies that $\int_{[-1,1] \times \R/\Z} v^* d\theta_M = 0$.
Hence $\pi_M \circ v$ is constant.
Since $\check{\phi}(\pi_M(v(0,1))) = \pi_M(v((0,0)))$ and $\pi_M \circ v$ is constant, we get that the image of $\pi_M \circ v$ is a fixed point 
of $\check{\phi}$ inside $N_B - \check{N}_B$.
But this is impossible since $v(0,t) \in N_B - \check{N}_B$.
\qed

\bigskip

As a result of the above lemma, we have the following definition:
\begin{defn}
Let $B \subset M$ be an isolated family of fixed points of some positive slope perturbation $\check{\phi}$ of $\phi$ and let $N_B$ be a neighborhood of $B$ as in Lemma \ref{lemma:smallperturbationinneighborhood}.
Let $\check{\phi}'$
be a $C^\infty$ small perturbation so that all the fixed points of $\check{\phi}'$ inside $N_B$ are non-degenerate.
Since all Floer trajectories of $(\check{\phi}',(J_t)_{t \in [0,1]})$ are contained inside $N_B$, we can define the Floer cohomology group
$HF^*(\check{\phi},B)$ to be defined in the usual way where we only consider fixed points inside $N_B$.
Such a group is called the {\it local Floer cohomology of $B$}. Again it does not depend on the choice of perturbation $\check{\phi}'$ or $(J_t)_{t \in [0,1]}$
although we will not need this fact here.

Note that if $B$ is the only set of fixed points of $\check{\phi}$ of action in the interval $[a,b]$,
then
the above definition coincides with the definition of
$HF^*(\check{\phi},B)$ from Definition \ref{defn:localfloer}.
More generally, if $B$ is a union of isolated families of fixed points $B_1,\cdots,B_l$ all of the same action
then $HF^*(\check{\phi},B) = \bigoplus_{i=1}^l HF^*(\check{\phi},B_i)$.
\end{defn}

\begin{lemma} \label{lemma:codimension0familyoffixedpoints}
$(M,\theta_M,\phi)$ be a graded abstract contact open book.
Let $\check{\phi} : M \lra{} M$ be a small positive slope perturbation of $\phi$.
Suppose that all the fixed points of $\check{\phi}$ of action in $[a,b]$ is equal to $B = \sqcup_{i = 1}^l B_i$ where $B_1,\cdots,B_l$ are codimension $0$ families of fixed points, all of the same action.
Then 
\begin{equation} \label{eqn:localfloerformula}
HF^*(\check{\phi},B) \equiv \bigoplus_{i=1}^l H_{n-*-CZ(\phi,B_i)}(B_i;\Z).
\end{equation}
\end{lemma}
\proof
Let $N_{B_i} \subset M$ be a small neighborhood of $B_i$ with the property that $\check{\phi}$ is the time $1$ flow of a Hamiltonian
$H_{B_i} : N_{B_i} \lra{} (-\infty,0]$
satisfying $B_i = H_{B_i}^{-1}(0)$ for each $i$.
Let $\phi^{H_{B_i}}_1 : N_{B_i} \lra{} N_{B_i}$ be the time $1$ flow of $H_B$ for each $i$.
After possibly shrinking each neighborhood $N_{B_i}$,
we can assume that $N_{B_1},\cdots,N_{B_l}$ are all disjoint.
Let $(J_t)_{t \in [0,1]}$ be a generic
smooth family
almost complex structures cylindrical near $\partial M$.
By Lemma \ref{lemma:smallperturbationinneighborhood},
any sufficiently small $C^\infty$ perturbation $\check{\phi}'$ of $\check{\phi}$
has the property that any Floer trajectory
connecting fixed points inside $\cup_{i =1}^l N_{B_i}$ is actually contained inside $N_{B_j}$ for some $j$.
Therefore
$$HF^*(B) = \bigoplus_{i=1}^l HF^*(\phi^{H_{B_i}}_1,B_i).$$
Hence by \cite[Theorem 7.1]{SZ:morsetheory}
combined with \ref{item:catenationczproperty},
we have that Equation (\ref{eqn:localfloerformula}) holds.
\qed

\bigskip

\begin{proof}[Proof of Property \ref{item:spectralsequenceproperty}.]

Let $\check{\phi}'$ be a $C^\infty$ small perturbation
of $\check{\phi}$ and let $(J_t)_{t \in [0,1]}$ be a $C^\infty$ generic smooth family of almost complex structures cylindrical near $\partial M$.
Let $\alpha_i$ be the action of $B_i$ for each $i \in \{1,\cdots,l\}$.
For each $p \in \N$, choose $\beta_p \in \R$
so that $\alpha_i \neq \beta_p$ for all $i \in \{1,\cdots,l\}$ and so that $\alpha_i > \beta_p$ if and only if $\iota(i) \geq p$.
Let $F_p$ be the subgroup of the chain complex $CF^*(\check{\phi}')$
generated by fixed points of action greater than $\beta_p$.
Then $(F_p)_{i \in p}$ is a filtration on this chain complex.
By Lemma \ref{lemma:codimension0familyoffixedpoints},
$$H^*(F_p/F_{p-1}) =
HF^*_{[\beta_{p-1},\beta_p]}(\check{\phi},B_p) =
\bigoplus_{ \{i \in \{1,\cdots,l\} \ : \ \iota(i) = p\} }
H_{n-*-CZ(\phi,B_j)}(B_p;\Z)$$
for all $p = 1,\cdots, l$
and
$H^*(F_p/F_{p-1})=0$ if $p \in \N - \{1,\cdots,l\}$.
Therefore the spectral sequence associated to the filtration $(F_p)_{p \in \N}$ is (\ref{eqn:morsebottspectralsequence}).
\end{proof}

\bibliography{references}

\end{document}